\documentclass[11pt,reqno]{amsart}

\usepackage{amsmath,amssymb,amsthm,amsfonts}   
\usepackage{bbm} 
\usepackage{graphicx}   
\usepackage{pstricks}   
\usepackage[colorlinks=true, linkcolor=blue, citecolor=magenta]{hyperref} 
\usepackage{url}
\usepackage{tasks}
\usepackage{enumerate}
\usepackage[shortlabels]{enumitem}
\usepackage{caption,subcaption} 

\usepackage{mathtools}      
\mathtoolsset{showonlyrefs} 

\usepackage{mathrsfs}
\usepackage{multicol}   
\usepackage{fullpage}   

\usepackage[foot]{amsaddr}

\usepackage{dsfont} 

\newcommand{\R}{\mathbb{R}}
\newcommand{\Z}{\mathbb{Z}}
\newcommand{\N}{\mathbb{N}}
\newcommand{\F}{\mathbb{F}}
\newcommand{\Q}{\mathbb{Q}}
\newcommand{\C}{\mathbb{C}}
\newcommand{\vc}[1]{\boldsymbol{#1}}
\newcommand{\ind}{\mathds{1}} 	
\newcommand{\U}[1]{T_{#1}}
\newcommand{\Sz}{\mathcal{S}} 	
\newcommand{\Mod}[1]{\;(\textrm{mod } #1)}

\DeclarePairedDelimiter\abs{\lvert}{\rvert}
\DeclarePairedDelimiter\norm{\lVert}{\rVert}
\DeclarePairedDelimiter\japan{\langle}{\rangle}

\DeclareMathOperator{\supp}{supp}
\DeclareMathOperator{\BigO}{\mathcal{O}}

\DeclareFontFamily{U}{mathx}{\hyphenchar\font45}
\DeclareFontShape{U}{mathx}{m}{n}{
      <5> <6> <7> <8> <9> <10>
      <10.95> <12> <14.4> <17.28> <20.74> <24.88>
      mathx10
      }{}
\DeclareSymbolFont{mathx}{U}{mathx}{m}{n}
\DeclareFontSubstitution{U}{mathx}{m}{n}
\DeclareMathAccent{\widecheck}{0}{mathx}{"71}
\DeclareMathAccent{\wideparen}{0}{mathx}{"75}

\newtheorem{thm}{Theorem}[section]
\newtheorem{cor}[thm]{Corollary}
\newtheorem{prop}[thm]{Proposition}
\newtheorem{lem}[thm]{Lemma}

\newtheorem{rmk}[thm]{Remark}
\newtheorem*{rmk*}{Remark}
\newtheorem*{acknowledgements*}{Acknowledgements}

\theoremstyle{definition}
\newtheorem{defin}[thm]{Definition}

\newtheorem{innercustomthm}{Theorem} 	
\newenvironment{customthm}[1]
  {\renewcommand\theinnercustomthm{#1}\innercustomthm}
  {\endinnercustomthm}

\makeatletter
\@namedef{subjclassname@2020}{\textup{2020} Mathematics Subject Classification}
\makeatother

\title{Pointwise Convergence over Fractals for Dispersive Equations with Homogeneous Symbol}

\author[D. Eceizabarrena]{Daniel Eceizabarrena}
\address[D. Eceizabarrena]{Department of Mathematics and Statistics,
		University of Massachusetts Amherst,
		Amherst MA 01003, United States}
\email{eceizabarrena@math.umass.edu}

\author[F. Ponce-Vanegas]{Felipe Ponce-Vanegas}
\address[F. Ponce-Vanegas]{BCAM - Basque Center for Applied Mathematics,
		Mazarredo 14, E48009 Bilbao, Basque Country -- Spain}
\email{fponce@bcamath.org}


\date{\today}

\begin{document}

\begin{abstract}
We study the fractal pointwise convergence for the equation 
$i\hbar\partial_tu + P(D)u = 0$, where the symbol $P$ is real, homogeneous and non-singular.
We prove that for initial data $f\in H^s(\R^n)$ with $s>(n-\alpha+1)/2$
the solution $u$ converges to $f$ $\mathcal{H}^\alpha$-a.e, where
$\mathcal{H}^\alpha$ is the $\alpha$-dimensional Hausdorff measure.
We improve upon this result depending on the dispersive strength of $P$.
On the other hand, for a family of polynomials $P$ and given $\alpha$, 
we exploit a Talbot-like effect 
to construct initial data whose solutions $u$ diverge 
in sets of Hausdorff dimension $\alpha$. 
To compute the dimension of the sets of divergence, 
we adopt the Mass Transference Principle from Diophantine approximation.
We also construct counterexamples for quadratic symbols like the saddle
to show that 
our positive results are sometimes best possible.
\end{abstract}

\maketitle



\section{Introduction}

Let $P\in C^\infty(\R^n\setminus \{0\})$ be a real function that defines the 
following dispersive initial value problem:
\begin{equation} \label{eq:PDE}
\left\{
\begin{aligned}
i\hbar\partial_t u + P(D)u = 0, \\
u(\cdot,0)  = f.
\end{aligned}
\right.
\end{equation}
Here, $\hbar = 1/(2\pi)$ and $D = -i\hbar \partial$,
and we denote the solution to \eqref{eq:PDE} by $T_tf$. 
We tackle the fractal pointwise convergence problem for \eqref{eq:PDE},
that is, given $0 \leq \alpha \leq n$, 
if $\mathcal H^\alpha$ denotes the $\alpha$-Hausdorff measure, 
for which $s $ do we have 
\begin{equation}
\lim_{t \to 0} T_tf(x) = f(x), \qquad \mathcal H^\alpha\text{-almost everywhere}, \qquad \forall f \in H^s(\mathbb R^n) \, \, ?
\end{equation}
In this article we consider the case of homogeneous $P$ of degree $k\ge 1$, $k\in\R$,
that are non-singular
in the sense that $\nabla P(\xi) \neq 0$ for all $\xi\in \R^n\setminus \{0\}$.
In a nutshell, we prove:

\textbf{Positive results}
\begin{itemize}
	\item Theorem~\ref{thm:non-dispersive_B}: convergence holds for $s > (n-\alpha + 1)/2$.
	\item Theorem~\ref{thm:dispersive_B}: depending on the dispersive strength of $P$, 
	we lower the regularity of Theorem~\ref{thm:non-dispersive_B}.
\end{itemize}

\textbf{Negative results} 

We construct counterexamples for some particular symbols. 
As a novelty, to compute the Hausdorff dimension of the sets of divergence
we use the Mass Transference Principle, a technique originated in Diophantine approximation.
\begin{itemize}
	\item Theorem~\ref{thm:divergence_k}: counterexamples for $P(\xi) = \xi_1^k + \ldots + \xi_n^k$ and other similar symbols, generalizing the results of \cite{Pierce2021} to the fractal case $\alpha < n$. 	 
	
	\item Theorem~\ref{thm:examples_saddle}: counterexamples for $P(\xi) = \xi_1^2 + \cdots + \xi_m^2 - \xi_{m+1}^2 - \cdots - \xi_n^2$. 	
\end{itemize}

\subsection{Motivation}
The classical problem of convergence concerns the Schrödinger equation, 
which corresponds to $P(\xi) = |\xi|^2$.
Carleson asked in \cite{MR576038} 
for the minimal regularity $s \geq 0$ such that all functions $f\in H^s(\R^n)$ satisfy
$\lim_{t\to 0} e^{it\Delta/\hbar}f = f$ almost everywhere.
For $n = 1$, he proved that $s \geq 1/4$ is sufficient, while
Dahlberg and Kenig \cite{MR654188} showed that this is also necessary.

For dimensions $n \geq 2$, the problem was subsequently investigated by several authors in
\cite{MR934859, MR904948, MR1413873, MR2264734, MR3241836, demeter16, MR3842310}, 
to mention a few.
In \cite{Bourgain2016}, Bourgain gave a counterexample, discussed in detail in \cite{MR4186521}, to prove that
$s \geq \frac{n}{2(n+1)}$ is necessary,
while Du, Guth and Li proved in \cite{MR3702674} that 
$s>1/3$ is sufficient in $n=2$, and
Du and Zhang \cite{MR3961084} proved that 
$s>\frac{n}{2(n+1)}$ is sufficient for $n\ge 3$.
Thus, the problem has been solved except for the endpoint.

Several variations of the problem have been proposed; for example,
convergence along curves \cite{MR2871144, MR3118482, MR4273651},
convergence for other equations \cite{MR1136935, MR2284549, MR2379688, MR4198507},
and convergence in other manifolds \cite{MR3984026, eceizabarrena2020}, 
some papers addressing more than one version of the problem.

The fractal refinement we consider in this article, 
namely
the convergence $\mathcal H^\alpha$-almost everywhere, 
has been studied, among others, in
\cite{MR2754999, MR3896208, MR3613507, MR3903115}.
The question is to determine the critical regularity
\begin{equation}
s_c(\alpha) = \inf\left\{s\ge 0 \; \mid \,    \lim_{ t \to 0} \U{t}f = f \quad \mathcal H^\alpha\text{-a.e.} \quad \textrm{for every }f\in H^s(\R^n)\right\}.
\end{equation}
Some preliminary properties of $s_c(\alpha)$ are the following:
\begin{enumerate}[(i)]
\item $s_c(\alpha)\le n/2$, because if $f\in H^s(\R^n)$ for some $s>n/2$, then 
$\U{t}f$ is continuous and the solution converges everywhere;
\item $s_c(\alpha)$ is a non-increasing function in $\alpha$; 
\item  $s_c(\alpha)\ge (n-\alpha)/2$, because as shown in \cite{MR1903760}, for $s<(n-\alpha)/2$ the initial datum $f\in H^s(\R^n)$
can diverge in a set of dimension $\alpha$.
\end{enumerate}

Frequently, the problem is rephrased in terms of 
almost everywhere convergence with respect to Frostman measures, 
that is, probability measures supported in the unit ball $B_1$ that
satisfy $\mu(B_r)\le Cr^\alpha$ for all $r>0$.
Denoting the collection of all such measures by $M^\alpha(B_1)$,  
Frostman's lemma implies that the critical regularity can also be computed by 
\begin{equation}
s_c(\alpha) = \inf\left\{s\ge 0 \, \mid \, \lim_{t \to 0} \U{t}f = f \, \, \, \, \mu\textrm{-a.e. }  \textrm{ for every } \mu\in M^\alpha(B_1)\textrm{ and }f\in H^s(\R^n)\,  
	\right\}.
\end{equation}

\subsection{Positive results}

Convergence is usually better in presence of dispersion,
in which case the $L^2$ norm is constant while
the $L^\infty$ norm decays.
The rate of decay can be seen as a measure of the strength of the dispersion. 
As an example, the Schr\"odinger equation, with $P(\xi) = |\xi|^2$, 
has dispersion and satisfies both conditions, 
while the transport equation, with $P(\xi) = \xi$, does not.

Let us see some heuristics when there is no dispersion. 
Let $\beta>0$.
If $s<(n-\beta)/2$, we saw that a function $f\in H^s(\R^n)$ 
can diverge in a set $E\subset\R^n$
of Hausdorff dimension $\beta$. 
Since high frequencies remain concentrated, 
at times $t >0$ the solution $u(\cdot,t)$ might still diverge in a set $E_t$
of dimension $\beta$. 
Based on Cavalieri's principle 
(which, in the fractal setting, works in some situations but is in general false \cite[Ch. 8]{MR2118797}),
one may naively think that 
the sets $E_t$ make a set $F$ of dimension $\beta + 1$ in $\R^n\times\R$.
Moreover, if the sets $E_t$ are more or less disjoint, 
one may expect the projection of $F$ to $\R^n$ to have dimension $\beta+1$.
Thus, in this bad scenario, the solution might diverge in a set of dimension $\alpha = \beta + 1$.
In other words, 
if $s<(n-\alpha+1)/2$ and $f\in H^s(\R^n)$,
we should expect the solution $u$ to diverge in set of dimension $\alpha$.

In this setting, we first prove that $s_c(\alpha)\le (n-\alpha+1)/2$.
This amounts to saying that the scenario described above is the worst outcome. 
Since we expect it to hold when no dispersion is present, 
we call $(n-\alpha+1)/2$ the non-dispersive threshold.

\begin{thm}[The non-dispersive upper bound] \label{thm:non-dispersive_B}
Let $P\in C^\infty(\R^n\setminus\{0\})$ be a non-singular, homogeneous function of
degree $k\ge 1$ with $k\in \R$.
Let also $s>(n-\alpha+1)/2$. Then,
\begin{equation}
\lim_{ t \to 0} \U{t}f = f \quad \mu\text{-a.e.}, \qquad \forall \mu\in M^\alpha(B_1) \quad \text{ and } \quad \forall f\in H^s(\R^n).
\end{equation}
\end{thm}

\begin{center}
\includegraphics[scale=0.95]{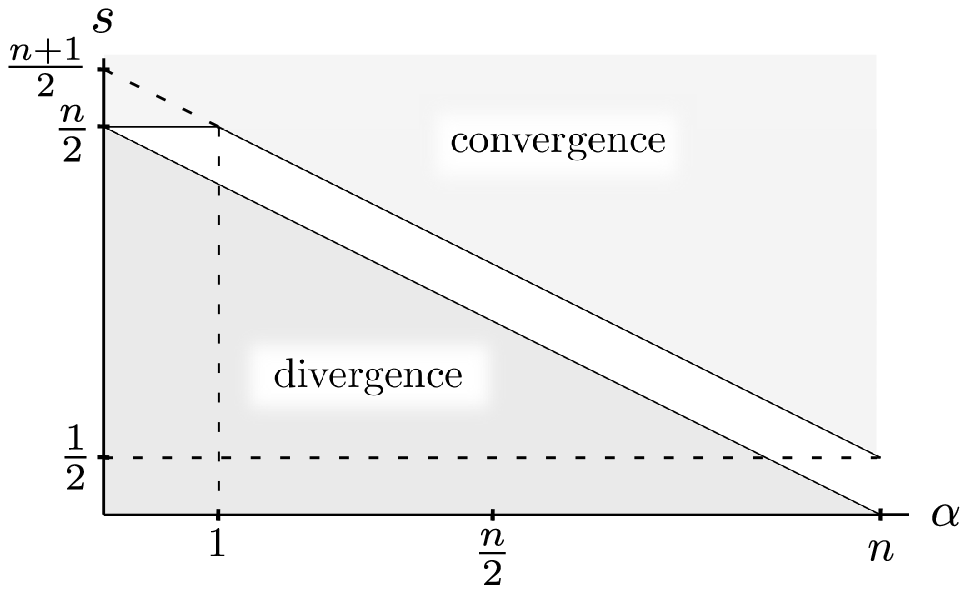}
\end{center}

When $\alpha = n$, this result has already been proved in several cases:
\begin{itemize}
	\item for $P = \abs{\xi}^2$ by 
Vega \cite{MR934859} and Sjölin \cite{MR904948};
\item for $P$ a polynomial of degree two by Rogers \textit{et al.} \cite{MR2284549};
\item for $P$ a polynomial of principal type by Ben-Artzi and Devinatz \cite{MR1136935}.
\end{itemize}
For general $\alpha$, this result was proved
for $P$ elliptic of degree $m\ge 2$ by Sj\"ogren and Sj\"olin \cite{MR997967}.

Let us come back to the dispersive case. 
When high frequencies are dispersed in different directions
at different speeds, 
one could think that the probability that high frequencies concentrate
at many spots most of the time is small, so
one expects that at a fixed point in space
the solution evolves somehow smoothly.
Kato called attention to this fact in \cite{MR759907}.
Therefore, we expect to improve the non-dispersive bound
in the presence of dispersion.
This is actually the case for many symbols $P$; however, 
there are symbols such as $P(\xi) = \xi_1^2-\xi_2^2$ for which 
the non-dispersive bound cannot be improved, as was proved in \cite{MR2284549} for $\alpha = n$.

Our next result shows that, in the presence of dispersion and when $\alpha$ is small, 
we can reach the lowest possible $s_c(\alpha) = (n-\alpha)/2$.
How small $\alpha$ can be chosen depends on the strength of the dispersion.

\begin{thm}[The dispersive bound] \label{thm:dispersive_B}
Let $P\in C^\infty(\R^n\setminus\{0\})$ be a non-singular, homogeneous function of
degree $k\ge 1$, $k\in \R$, and
suppose there exists $\beta >0$ such that $\norm{\U{t}\varphi}_\infty\le C_\varphi\abs{t}^{-\beta}$
for every $\varphi \in \mathcal S(\R^n) $ with Fourier support in $\{ \abs{\xi}\simeq 1 \}$. 
Let $\alpha<\beta$. Then, if $s>(n-\alpha)/2$, 
\begin{equation}
\lim_{ t \to 0} \U{t}f = f \quad \mu\text{-a.e.}, \qquad \forall \mu\in M^\alpha(B_1) \quad \text{ and } \quad \forall f\in H^s(\R^n).
\end{equation}
\end{thm}

If $D^2P$ is non-singular,
then $\norm{\U{t}\varphi}_{\infty}\le C_\varphi\abs{t}^{-n/2}$
for every $\varphi \in \mathcal S(\R^n) $ 
with $\operatorname{supp} \widehat{\varphi} \subset \{ \xi \in \mathbb R^n \, : \, |\xi| \simeq 1\}$; 
see Theorem~1 in Ch. 8.3 of \cite{MR1232192}.
As an immediate consequence, we get:

\begin{cor} \label{thm:dispersive_non-Singular}
If $D^2P(\xi)$ is non-singular for $\xi\neq 0$
and if $\alpha<n/2$, then 
$s_c(\alpha) = (n-\alpha)/2$.
\end{cor}

This result was proved by Barceló \textit{et al.} \cite{MR2754999}
for $P = \abs{\xi}^2$, 
as well as for $P = \abs{\xi}^m$ with $m > 1$
 when $\alpha<(n-1)/2$.

Theorem~\ref{thm:non-dispersive_B} and Corollary~\ref{thm:dispersive_non-Singular}, 
together with the fact that $s_c(\alpha)$ is non-increasing,
imply that if $D^2P$ is non-singular, then
\begin{equation}\label{eq:Non_Singular_Prelim}
\begin{array}{ll}
s_c(\alpha) =  (n-\alpha)/2, &  \textrm{ if }\alpha \le n/2, \\
& \\
s_c(\alpha) \leq n/4, & \text{ if } n/2 \leq \alpha \leq n/2 + 1, \\
& \\
s_c(\alpha) \leq (n-\alpha + 1)/2, & \text{ if } n/2 + 1 \leq \alpha \leq n.
\end{array} 
\end{equation}
When $P$ is a non-elliptic quadratic symbol,
this result was already proved in Theorem~1, eq. (4) of \cite{MR4196386}.
To compare their results with ours,
the reader can use the identity
\begin{equation}  \label{eq:Barron_ours}
s_c(\alpha) = 
	\frac{n-\beta(\alpha,\mathbb{H}_m^{n})}{2},
	\qquad d - 1 \;(\textit{their notation}) = n\;(\textit{ours}).
\end{equation}
See also Theorem~2 and Proposition~2.1 of \cite{MR4196386}.


When $D^2P$ is positive definite,
we can complement Corollary~\ref{thm:dispersive_non-Singular}
with Du and Zhang's deep result in Theorem~2.4 of \cite{MR3961084}, so that
\begin{equation}
\begin{array}{ll}
s_c(\alpha) =  (n-\alpha)/2, &  \textrm{ if }\alpha \le n/2, \\
& \\
s_c(\alpha) \leq n/4, & \text{ if } n/2 \leq \alpha \leq (n+1)/2, \\
& \\
s_c(\alpha) \leq \displaystyle{\frac{n}{2(n+1)}+\frac{n}{2(n+1)}(n-\alpha)}, & \text{ if } (n+1)/2 \leq \alpha \leq n.
\end{array} 
\end{equation}
This is already known for the Schr\"odinger equation.

\subsection{Negative results}\label{sec:Intro_Negative}

We first consider polynomial symbols $P$ of the type
\begin{equation}\label{eq:Polynomial}
P(\xi) = \xi_1^k + W(\xi'), \qquad \qquad \xi = (\xi_1,\xi') = (\xi_1, \ldots, \xi_n) \in \mathbb R^n,
\end{equation}
where $k\ge 2$ is an integer and $W\in \Q[X_2,\ldots,X_n]$ has degree $k$.
The case $P(\xi) = \abs{\xi}^2$, corresponding to $k=2$,
was investigated in \cite{LucaPonceVanegas2021}.
Our main source of inspiration is the work of An, Chu and Pierce in \cite{Pierce2021},
where they tackled the the symbol 
$P(\xi) = \xi_1^k + \xi_2^k +\ldots + \xi_n^k$ for $k\geq 3$
and $\alpha = n$.
Our main result is the following, which is portrayed in Figure~\ref{fig:divergence_k}.

\begin{thm} \label{thm:divergence_k}
Let $P(\xi) =  \xi_1^k + W(\xi')$, where
$W\in\Q[X_2,\ldots,X_n]$ has degree $k\ge 2$ 
and its homogeneous part $W_k$ is non-singular
 in the sense that
$\nabla W_k(\xi')\neq 0$ for every $\xi'\in\C^{n-1}\setminus\{0\}$.

Suppose that either $2\le k\le 2(n-1)$ and
\begin{equation}
s<\begin{cases}
(i)\;\thickspace
	\dfrac{1}{4} + \dfrac{n-1}{4(n(k-1)+1)}+\dfrac{(n-1)(k-1)}{2(n(k-1)+1)} (n-\alpha) & 
	\textrm{for } n-\dfrac{n-1}{2k}\le \alpha\le n, \\[12pt]
(ii)\thickspace
	\dfrac{1}{4} + \dfrac{n-1}{4(n+k-1)}+\dfrac{(n-1)}{2(n+k-1)} (n-\alpha)&
	\textrm{for } n-\dfrac{1}{2}-\dfrac{n-1}{k}\le \alpha \le n-\dfrac{n-1}{2k},
\end{cases}
\end{equation}
or $k>2(n-1)$ and
\begin{equation}
s<\begin{cases}
(i)\;\;\thickspace
\dfrac{1}{4} + \dfrac{n-1}{4(n(k-1)+1)}+\dfrac{(n-1)(k-1)}{2(n(k-1)+1)} (n-\alpha) & 
	\textrm{for } n-\dfrac{n-1}{2k}\le \alpha\le n, \\[12pt]
(ii)\;\thickspace
\dfrac{1}{4} + \dfrac{n-1}{4(n+k-1)}+\dfrac{(n-1)}{2(n+k-1)} (n-\alpha) &
	\textrm{for } n-\dfrac{n-1}{k-n+1}\le \alpha \le n-\dfrac{n-1}{2k}, \\[12pt]
(iii)\thickspace
\dfrac{1}{4} + \dfrac{n-1}{4k} + \dfrac{n-1}{4k}(n-\alpha) &
	\textrm{for } n-\dfrac{k+n-1}{2k-n+1} \le \alpha \le n-\dfrac{n-1}{k-n+1}.
\end{cases}
\end{equation}
\\[-2mm]
Then there exists $f\in H^s(\R^n)$ such that
$\U{t}f$ diverges in a set of Hausdorff dimension $\alpha$. 
\end{thm}

\begin{figure}[t] 
\includegraphics[width=0.9\linewidth]{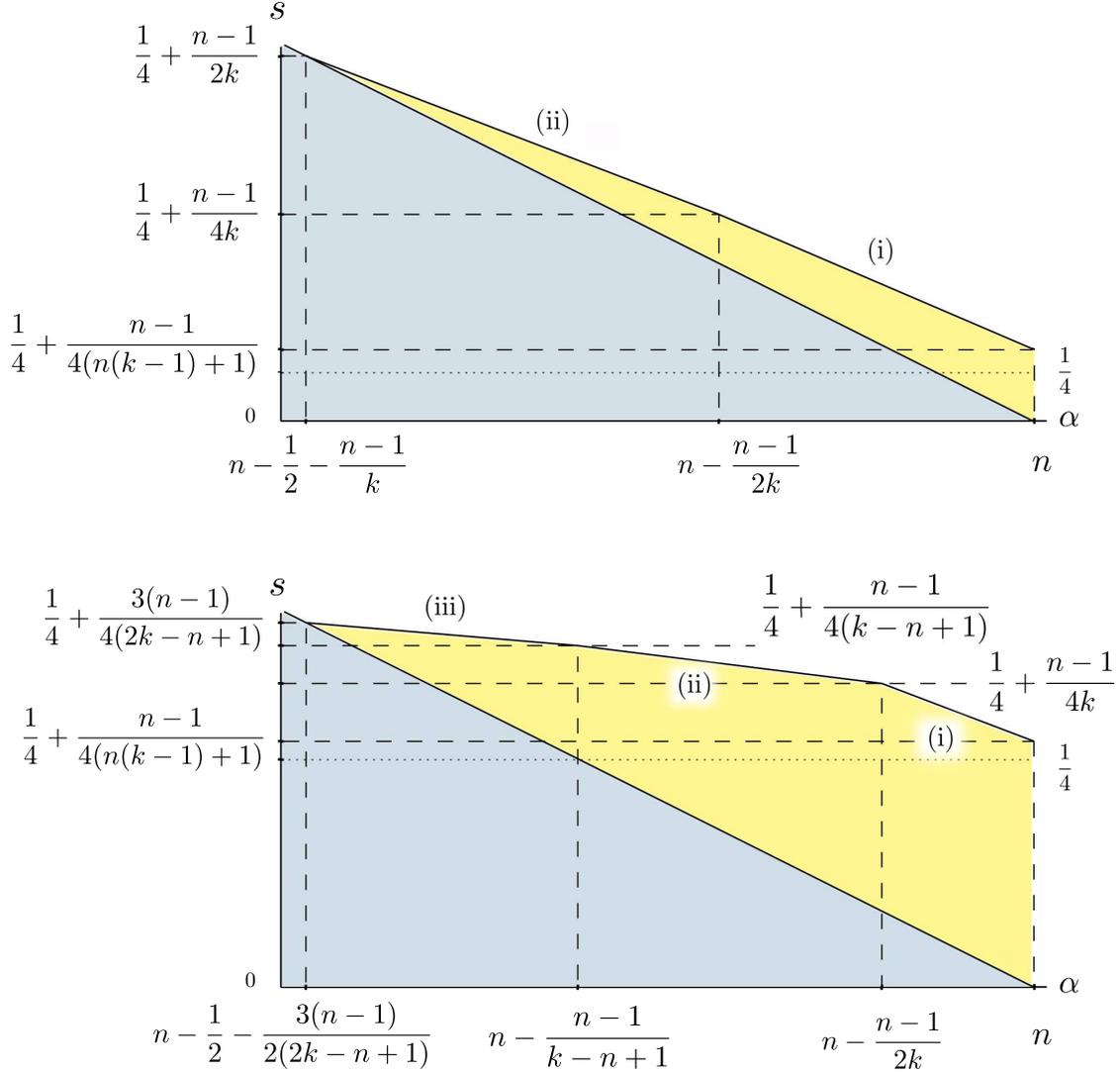}
\caption{Representation of Theorem~\ref{thm:divergence_k}.
The plot at the top is $k\le 2(n-1)$, and 
the plot at the bottom is $k> 2(n-1)$.
In the blue region the initial datum itself diverges,
so only the yellow region is not trivial.
The numbering (i), (ii) and (iii) corresponds with that in the theorem.} \label{fig:divergence_k}
\end{figure}

Theorem~\ref{thm:divergence_k} asks for $\nabla W_k(\xi')\neq 0$
 for every $\xi'\in\C^{n-1}\setminus\{0\}$. 
 However, the theorem might hold for a larger class of polynomials, 
like for those that satisfy the Weil bound.
See Section~\ref{sec:Building_counterexample} for more details.

To build counterexamples,
the recurring idea is to place as many wave packets as possible 
in a carefully chosen plane in $(x,t)$,
in such a way that the oscillations interact coherently,
making the solution to be large in a large set over that plane.
Our proof is based on Bourgain's counterexample \cite{MR4186521} for the Schr\"odinger equation, 
which exploits Gauss sums and the underlying Talbot effect
to compute the solution at rational times.
When $k>2$, the oscillatory sums demand a finer manipulation,
so An, Chu and Pierce \cite{Pierce2021} introduced to this context 
Deligne's Theorem \cite[Theorem 8.4]{MR340258}.

The divergence sets we get are very similar to 
those studied in diophantine approximation.
To compute their Hausdorff dimension, we use the 
Mass Transference Principle, a powerful tool in that field
introduced by Beresnevich and Velani \cite{BeresnevichVelani2006}.
We discuss it in detail in Section~\ref{sec:MTP}, but let us motivate it briefly here.

A typical set considered in diophantine approximation
is the limsup of sets $F_m$ which are unions of balls $B_i$ of
radius $R_m^{-1}$, where $R_m$ is a sequence diverging to infinity. 
To compute the Hausdorff dimension of this limsup, 
one can take the union of such sets up to some scale $m$ 
and sort out the set $F_m$ to get a large subset $F_m'$ that is
somehow uniformly distributed. 
Then, one applies a Frostman measure technique to compute the dimension.
This can be very technical and exhausting;  
the reader can find an example in Section~5 of \cite{LucaPonceVanegas2021}.
On the other hand, the Mass Transference Principle runs 
the whole process automatically and efficiently: 
to compute the dimension of a limsup of balls $B(x_i,r_i)$, 
all it requires is to find an exponent $a$ such that the limsup of the dilated balls 
$B(x_i,r_i^a)$ has full measure. 

The Mass Transference Principle of Beresnevich and Velani \cite{BeresnevichVelani2006} 
is referred to as from balls to balls, since it only permits dilations that take balls to balls.
However, our divergence sets are not union of balls, but rectangles. Hence, this result
does not directly apply. 
Wang, Wu and Xu \cite{WangWuXu2015} proved a Mass Transference Principle from balls to rectangles,
which is still too rigid for our purposes.
Recently, Wang and Wu \cite{WangWu2021} proved a version from rectangles to rectangles.
This is the form we use. 

Finally, we study non-elliptic quadratic symbols, 
which by a linear transformation are reduced to 
\begin{equation}\label{eq:Poly_Saddle}
P(\xi) = \xi_1^2+\cdots + \xi_m^2 - \xi_{m+1}^2-\cdots -\xi_{n}^2
\qquad \qquad \xi = (\xi_1, \ldots, \xi_n) \in \mathbb R^n.
\end{equation}
Here, we assume that $1 \leq m\le n/2$ without loss of generality.
These symbols were also extensively studied in \cite{MR4196386}, where
several positive and negative results were established.

By \eqref{eq:Non_Singular_Prelim}, we only need to work with $\alpha > n/2$. 
For large $\alpha$, we prove that the non-dispersive exponent given 
in Theorem~\ref{thm:non-dispersive_B} is optimal.
For intermediate $\alpha$, we improve the trivial $(n-\alpha)/2$ without reaching the non-dispersive threshold. 
Notice that Theorem~\ref{thm:divergence_k} still applies here;
however, in this case we can build even more regular divergent initial data.
Our counterexamples are inspired by the work of Rogers, Vargas and Vega \cite{MR2284549},
who considered the case $\alpha = n$.
Results are portrayed in Figure~\ref{fig:saddle}.

\begin{thm} \label{thm:examples_saddle}
Let $P$ be a quadratic polynomial with index $1\le m\le n/2$.
\begin{itemize}
\item If $m\le n/2 - 1$, let
\begin{equation}
\begin{array}{ll}
s <   \displaystyle{\frac{n + (n-2m)(n-\alpha)}{2 ( n - 2m + 2 )}} & \qquad \textrm{ for } n/2 \le \alpha\le n-m+1, \\ 
& \\
s < (n-\alpha+1)/2 & \qquad \textrm{ for } n-m+1\le \alpha\le n.
\end{array} 
\end{equation}

\item If $n$ is odd and $m = (n-1)/2$, let 
\begin{equation}
\begin{array}{ll}
s <   (n-\alpha +m +1)/4 &\qquad \textrm{ for } (n+1)/2\le \alpha\le (n+3)/2, \\ 
& \\
s < (n-\alpha+1)/2  &\qquad \textrm{ for } (n+3)/2\le \alpha\le n.
\end{array} 
\end{equation}

\item If $n$ is even and $m = n/2$, let
\begin{equation}
s < (n-\alpha+1)/2 \qquad \textrm{ for } n/2+1\le \alpha\le n.
\end{equation}
\end{itemize}
Then there exists $f\in H^s(\R^n)$ such that
$\U{t}f$ diverges in a set of Hausdorff dimension $\alpha$.
\end{thm}

\begin{figure}[h]
\centering
\includegraphics[width=0.92\linewidth]{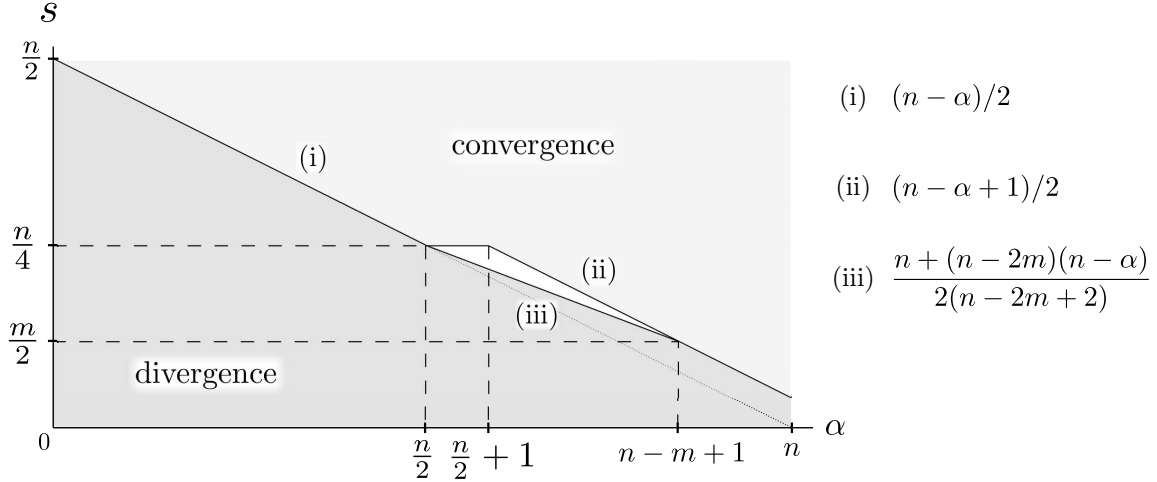}
\caption{Representation of Theorem~\ref{thm:examples_saddle}
	for the case $m\le n/2-1$.
	We do not include here the positive results in \cite{MR4196386};
	the convergence region is actually greater.} \label{fig:saddle}
\end{figure}

Theorems~\ref{thm:examples_saddle} and \ref{thm:non-dispersive_B}
imply $s_c(\alpha) = (n-\alpha + 1)/2$ when $\alpha \geq n - m + 1$. 
In particular, when $n$ is even and $m = n/2$:
\begin{equation}
s_c(\alpha)  = \begin{cases}
(n-\alpha)/2 & \text{ for } \alpha \leq n/2, \\
n/4 & \textrm{ for } n/2 \le \alpha \le n/2 + 1, \\
(n-\alpha+1)/2  & \textrm{ for } n/2 + 1 \le \alpha\le n.
\end{cases}
\end{equation} 
This result is morally subsumed in the work of Barron \textit{et al.} \cite[Corollary 1.1]{MR4196386}.
In fact, they also proved that
the corresponding result when $n$ is odd and $m = (n-1)/2$ is sharp.
Aside from the positive part, 
the difference with Theorem~\ref{thm:examples_saddle} is that
their counterexamples are constructed for an inequality involving a weighted $L^2$-norm
(see Remark~\ref{rmk:Barron_counter}),
while ours are directed to contradict convergence.

We should mention that 
Theorem~1, eq. (2) in \cite{MR4196386} morally improves Theorem~\ref{thm:examples_saddle} when $m\le n/2-1$ and $n-m \le \alpha < n-m+1$.
We defer a more detailed comparison until Section~\ref{sec:saddle}.

\subsection*{Outline of the paper}

\begin{itemize}
	\item Positive results:
\vspace{0.25cm}

\begin{enumerate}
\item[\it Section~\ref{sec:positive}:]  
We prove 
Theorems~\ref{thm:non-dispersive_B} and \ref{thm:dispersive_B}. 
\end{enumerate}

\vspace{0.25cm}

	\item Negative results:
	
\vspace{0.25cm}

\begin{enumerate}
\item[\it Section~\ref{sec:Building_counterexample}:] 
We build the counterexample for Theorem~\ref{thm:divergence_k} 
based on the one proposed by Bourgain in 
\cite{Bourgain2016} and An, Chu and Pierce in \cite{Pierce2021}
and we determine the divergence sets.

\item[\it Section~\ref{sec:MTP}:] 
We discuss the Mass Transference Principle,
which we use to compute the Hausdorff dimension of the divergence sets.
	
\item[\it Section~\ref{sec:FractalSet}:] 
We study in detail the divergence sets,
	which depend on certain parameters. 
	We compute their dimension. 
	Strictly speaking, the theorem follows by studying only some extremal parameters;
	however, the Mass Transference Principle allows us to deal 
	with the whole range of parameters	and ensure that 
	this is the best possible result for this counterexample.
	
\item[\it Section~\ref{sec:SobolevExponents}:] 	
 Once we know the dimension of the divergence sets, 
	we compute the Sobolev regularity of the counterexample
	and conclude the proof of Theorem~\ref{thm:divergence_k}.

\item[\it Section~\ref{sec:saddle}:] We prove Theorem~\ref{thm:examples_saddle}. 
\end{enumerate}
\end{itemize}

\subsection*{Notations}

\begin{enumerate}[(a)]
\item Miscellaneous: $e(z) = e^{2\pi iz}$; $B(a,r) = \{x : \abs{x-a}\le r\}$.
\item Relations: $A\lesssim B$ means that $A\le CB$ for some constant $C>0$;
analogously, we have $A\gtrsim B$ and $A\simeq B$.
When we want to stress some dependence of $C$ on a parameter $N$,
we write $A\lesssim_N B$. 
We write $c\ll 1$ as a shorthand of ``a sufficiently small constant''.
\item Algebra: if $k$ is a field, then 
$\bar{k}\supset k$ is the unique, up to isomorphism, algebraic closure of $k$.
For a prime $q$, $\mathbb{F}_q = \Z/q\Z$.
The projective space $\mathbb{P}^{d}(k)$ is $k^{d+1}\setminus\{0\}$ with the relation
$x\sim y$ if $x = \lambda y$ for some $\lambda\in k$.
If $R$ is a ring, then $R[X_1,\ldots, X_n]$ are the polynomials with coefficients in $R$.
 
\item Size of sets: If $E \subset \mathbb R^n$ is a Lebesgue measurable set, 
then either $\abs{E}$ or $\mathcal H^n(E)$ denote its Lebesgue measure.
If $E$ is a finite set, then $\abs{E}$ is the number of elements. 
For a set $A$, 
$\mathcal{H}^\alpha(A)$ and $\dim_{\mathcal H} A$ stand for the $\alpha$-Hausdorff measure
and the Hausdorff dimension of $A$.
%
\end{enumerate}

\subsection*{Acknowledgments}
Daniel Eceizabarrena is supported by the Simons Foundation Collaboration Grant on Wave Turbulence (Nahmod's Award ID 651469).
Felipe Ponce-Vanegas is funded by the Basque Government through
the BERC 2018-2021 program, by the Spanish State Research Agency
through BCAM Severo Ochoa excellence accreditation SEV-2017-0718,
the project PGC2018-094528-B-I00 - IHAIP
and the grant Juan de la Cierva -Formación FJC2019-039804-I, 
and by the ERC Advanced Grant 2014 669689 - HADE.

We would like to thank Renato Luc\`a for sharing with us all his insights on the problem.
We also thank Lillian Pierce for bringing her collaborative results to our attention and 
explaining some details,
Will Sawin for clarifying some points of Deligne's Theorem for us,
and Alex Barron for calling our attention to
his collaborative work on non-elliptic quadratic symbols.


\section{Convergence Results} \label{sec:positive}

In this section we prove Theorems~\ref{thm:non-dispersive_B} and \ref{thm:dispersive_B}.
As usual, we aim to bound the maximal operator 
$f\mapsto \sup_{0\le t\le 1}\abs{\U{t}f}$.
After space-time rescaling, 
we localize in time like in \cite{MR2264734} and
discretize the maximal operator like in \cite{MR3961084},
which leaves us with a more manageable operator.

The solution of \eqref{eq:PDE} is
\begin{equation}
\U{t}f(x) = \int \widehat{f}(\xi)e(x\cdot\xi + tP(\xi))\,d\xi.
\end{equation}
Since we will restrict the frequencies of $f$ to lie in some annulus,
we can localize with a cut-off $\varphi\in C^\infty_0(\R^n)$
and express the solution as
\begin{equation}
\U{t}f(x) = \int \varphi(\xi)\widehat{f}(\xi)e(x\cdot\xi + tP(\xi))\,d\xi.
\end{equation}
In this form, $\U{t}f$ is the Fourier transform of a measure over 
the graph of $P$, \textit{i.e.} over
$S := \{(\xi,P(\xi))\mid \xi\in\supp\varphi\}$,
so that
\begin{equation}
\U{t}f(x) = (\widehat{f}\,dS)^\vee(x,t).
\end{equation}
We will adopt this point of view at times.

Since our results refer to convergence $\mu$-a.e. for 
measures $\mu\in M^\alpha(B_1)$,
we should define $\U{t}f$ at least at that level of precision.
If $f$ is a measurable function, then
we choose the representative 
\begin{equation} \label{eq:lim_def_f}
\widetilde{f}(x) := \lim_{r\to 0}\frac{1}{\abs{B(x,r)}}\int_{B(x,r)}f,
\end{equation}
whenever the limit exists.
If $f\in H^s(\R^n)$ and $s>(n-\alpha)/2$, 
the limit \eqref{eq:lim_def_f} exists $\mu$-a.e.
Moreover, if we write $f = J_s*\japan{D}^sf$,
where $\japan{D}^sf\in L^2(\R^n)$ and 
$\widehat{J}_s(\xi) := \japan{2\pi\xi}^{-s}$ is the Bessel potential,
then the integral defining the convolution is 
absolutely convergent $\mu$-a.e., 
and $J_s*\japan{D}^sf = \widetilde{f}$ $\mu$-a.e.
We refer the reader to  
Definition~1.4 and Proposition~7.1 of \cite{MR540320}
for details.
Hence, we can write $(\U{t}f)\,\widetilde{ }\,$ also as $J_s*\U{t}\japan{D}^sf$.

By standard arguments, for any $s'>s$,
the solution $\U{t} f$ converges to $f\in H^{s'}(\R^n)$ $\mu$-a.e. if
\begin{equation}
\norm{\sup_{0\le t\le 1}\abs{\U{t}f}}_{L^2(\mu)}\le 
	C R^s\norm{f}_2,\quad\textrm{for every } R\gg 1,
\end{equation}
where $\supp\widehat{f}\subset \{\abs{\xi}\simeq R\}$.
Applying the transformation $(x,t)\mapsto (x/R,t/R^k)$ and using the homogeneity of $P$, 
that is, that $R^{-k}P(R\xi) = P(\xi)$, the expression above is equivalent to
\begin{equation}
\norm{\sup_{0\le t\le 1}\abs{\U{t}f}}_{L^2(\mu)} = 
	R^{-\alpha/2}\norm{\sup_{0\le t\le R^k}\abs{\U{t}f_R}}_{L^2(\mu_R)},
\end{equation}
where $f_R(x) := f(x/R)$ and $d\mu_R(x) := R^\alpha d\mu(x/R)$.
Therefore, our goal has changed to prove
\begin{equation}\label{eq:Maximal_Partial_Objective}
\norm{\sup_{0\le t\le R^k}\abs{\U{t}f_R}}_{L^2(\mu_R)} \le CR^{s-(n-\alpha)/2}\norm{f_R}_2,
\end{equation}
where $\supp\widehat{f}_R\subset \{\abs{\xi}\simeq 1\}$
and 
$\mu_R$ is a measure with support in $B_R$ that
satisfies $\mu(B_R) = R^\alpha$ and
$d\mu_R(B_r)\le Cr^\alpha$ for all $r>0$; 
let us say that $\mu_R\in M^\alpha(B_R)$.

We remark that, with a few modifications, it is possible to extend some arguments
to non-singular, almost $k$-homogeneous symbols $P$, that is, symbols that satisfy
\begin{equation}
0<c_1\le \frac{1}{R^{k-1}}\abs{(\nabla P)(R\xi)} \le c_0,\;\textrm{for }\abs{\xi}\simeq 1,\qquad
	\textrm{and}\qquad \frac{1}{R^{k-2}}\abs{(D^2 P)(R\xi)}\le c_2,
 	\qquad \forall R \gg 1.
\end{equation}

Our next step is to localize in time as Lee did in Lemma~2.3 of \cite{MR2264734}.
We show that instead of proving the maximal estimate for $t < R^k$ as in
\eqref{eq:Maximal_Partial_Objective}, 
it is enough to prove it for $t < R$.

\begin{lem} \label{thm:Short_Time}
Let $P$ be a non-singular symbol, 
\textit{i.e.} $\nabla P(\xi)\neq 0$ for all $\xi\in\R^n\setminus\{0\}$. 
Suppose that for some $q\ge 2$ and $\mu\in M^\alpha(B_R)$ it holds that
\begin{equation} \label{eq:thm:short_time}
\norm{\sup_{0\le t\le R}\abs{\U{t}f}}_{L^q(\mu)}\lesssim R^\beta\norm{f}_2
\end{equation} 
for every $f\in L^2(\R^n)$ such that $\supp \widehat{f} \subset \{\abs{\xi}\simeq 1\}$. Then,	
\begin{equation} \label{eq:thm:long_time}
\norm{\sup_{t\in\R}\,\abs{\U{t}f}}_{L^q(\mu)}\lesssim R^\beta\norm{f}_2,
\end{equation}
where the implicit constant depends on $\nabla P$.
\end{lem}
As a remark, we note that 
in the original argument of \cite{MR2264734} there are $\epsilon$-losses in the power of $R$. 
However, these losses were removed in Lemma~2.1 of \cite{MR2871144};
see also \cite[Lemma~2.1]{MR2970037}.

\begin{proof}
We can assume that $R\ge 1$, otherwise \eqref{eq:thm:long_time} always holds.
We cover the time line $\R$
with disjoint intervals $I_l$ of length $\abs{I_l} = R$ and 
center $R l$, for $l\in\Z$,
so that the supremum at \eqref{eq:thm:long_time} can be replaced by
\begin{equation}
\sup_{t\in\R}\,\abs{\U{t}f(x)} \le 
	\sup_l\sup_{t\in I_l}\abs{\U{t}f(x)}\le 
	\Big(\sum_l\sup_{t\in I_l}\abs{\U{t}f(x)}^q\Big)^{1/q}.
\end{equation}
We take the $L^q(\mu)$-norm at both sides to reach
\begin{equation} \label{eq:local_sum_sup}
\norm{\sup_{t\in\R}\,\abs{\U{t}f}}_{L^q(\mu)} \le 
	\Big(\sum_l\norm{\sup_{t\in I_l}\abs{\U{t}f}}_{L^q(\mu)}^q\Big)^{1/q}.
\end{equation}

Choose a function $\psi\in \Sz(\R)$ 
such that $\psi\ge 1$ in $[-1,1]$ and
$\widehat{\psi}\ge 0$ is supported in $[-5,5]$.
Choose one more function $\varphi\in \Sz(\R^n)$
with the same properties, and 
define the function of localization to $I_l$ as 
\begin{equation}
\psi_{R,l}(t)\varphi_R(x) := 
	\psi(R^{-1}(t-Rl))\varphi((c_0R)^{-1}x).
\end{equation}
The constant $c_0\ge 1$ is just to signal that
we will need to adjust the support later.
We localize $\U{t}f$ and write it as
\begin{align*}
\psi_{R,l}(t)\, \varphi_R(x) \, \U{t}f(x) 
	&= \left[(\widehat{\psi}_{R,l} \, \widehat{\varphi}_R)*(\widehat{f}\,dS)\right]^\vee(x,t) \\
	&= \int\left[\int \widehat{\psi}_{R,l}(\tau-P(\eta)) \,
		\widehat{\varphi}_R(\xi-\eta) \, \widehat{f}(\eta)\,d\eta\right]
		e^{2\pi i(x\cdot\xi+t\tau)}\,d\xi d\tau.
\end{align*}
After the change of variables $(\xi,\tau)\mapsto (\xi, P(\xi)+\tau)$
we find that
\begin{equation}\label{eq:local_slices}
\begin{split}
\psi_{R,l}(t) \, \varphi_R(x) \, \U{t}f(x) &= \\
	&\hspace*{-1cm}\int\left[\int \widehat{\psi}_{R,l}(P(\xi)-P(\eta)+\tau) \,
		\widehat{\varphi}_R(\xi-\eta) \, \widehat{f}(\eta)\,d\eta\right]
			e^{2\pi i(x\cdot\xi+tP(\xi))}\,d\xi\, e^{2\pi it\tau}\, d\tau.
\end{split}
\end{equation}
We redefine the function enclosed by parentheses in \eqref{eq:local_slices} as
\begin{equation} \label{eq:def:f_tau}
\widehat{f}_{\tau,l}(\xi) := 
	\int \widehat{\psi}_{R,l}(P(\xi)-P(\eta)+\tau) \,
		\widehat{\varphi}_R(\xi-\eta) \, \widehat{f}(\eta)\,d\eta.
\end{equation}
For $\abs{\xi-\eta}\le (c_0R)^{-1}$,
the function $\widehat{f}_{\tau,l}$ vanishes 
if $\abs{P(\xi)-P(\eta)+\tau}\ge 5R^{-1}$.
Thus, if we choose $c_0 \ge 10\sup_{\abs{\xi}\simeq 1} \abs{\nabla P(\xi)}$,
by the mean value theorem $\widehat{f}_{\tau,l}$ vanishes if $|\tau| > 10R^{-1}$.
Consequently, it suffices to integrate \eqref{eq:local_slices} in $\abs{\tau}\le 10R^{-1}$, 
so rewrite \eqref{eq:local_slices} as
\begin{align}
\psi_{R,l}(t) \, \varphi_R(x) \, \U{t}f(x) &= 
	\int_{\abs{\tau}\le 10R^{-1}} \widehat{f}_{\tau,l}(\xi)e(x\cdot\xi+tP(\xi))\,d\xi\, e^{2\pi it\tau}\,d\tau,\\
	&= \int_{\abs{\tau}\le 10R^{-1}} \U{t}f_{\tau,l}(x)\, e^{2\pi it\tau}\,d\tau.
\end{align}
\vspace*{2mm}
\begin{center}
\includegraphics[scale=0.6]{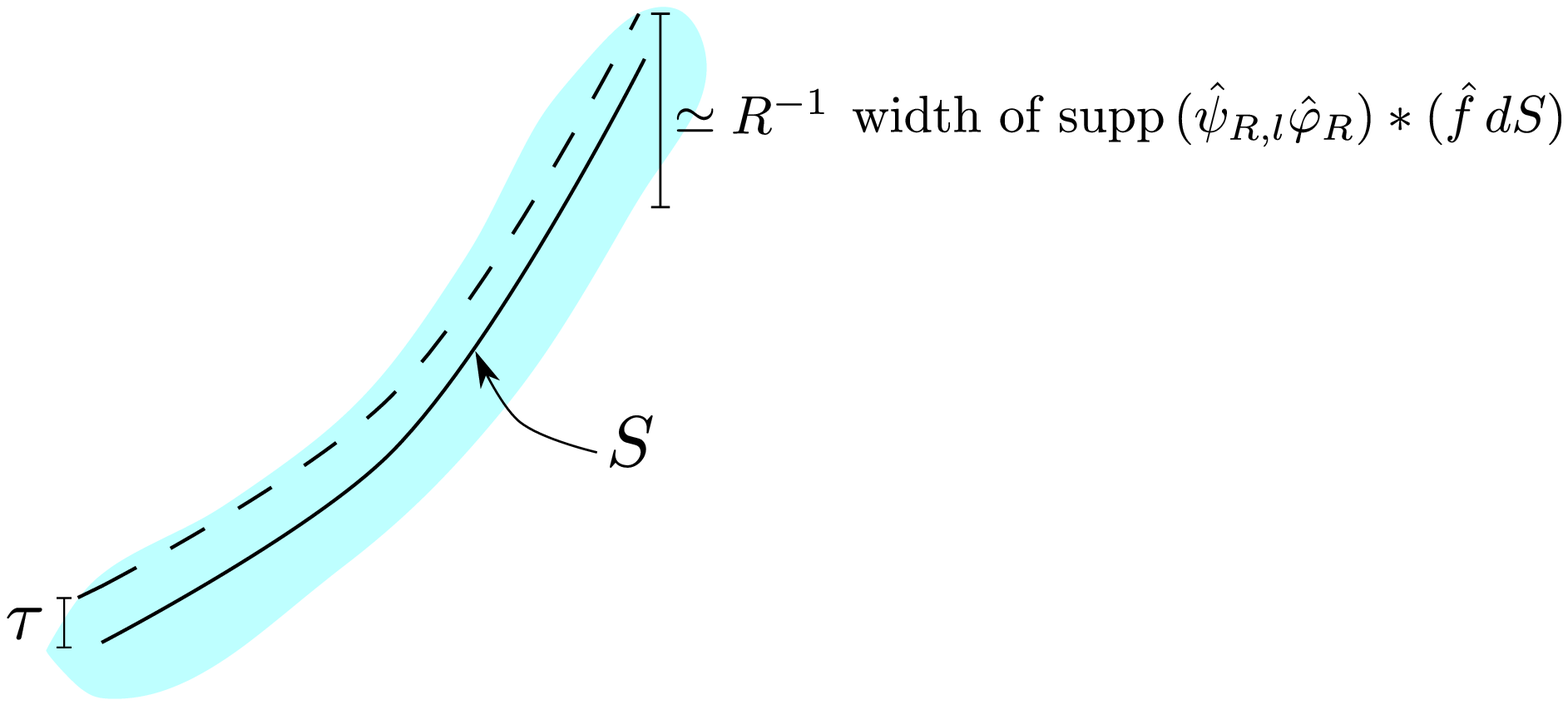}
\end{center}
\vspace*{0mm}

After this localization, 
let us apply the hypothesis \eqref{eq:thm:short_time} 
to each addend in \eqref{eq:local_sum_sup} so that
\begin{align}
\norm{\sup_{t\in I_l}\,\abs{\U{t}f}}_{L^q(\mu)} &\le 
	\norm{\sup_{t\in I_l}\,\abs{\psi_{R,l}(t)\varphi_R\U{t}f}}_{L^q(\mu)} \\
	&\le \int \norm{\sup_{t\in I_l}\,\abs{\U{t}f_{\tau,l}}}_{L^q(\mu)}\,d\tau \\
	&\le CR^\beta\int \norm{f_{\tau,l}}_2\,d\tau.
\end{align}
Hence, \eqref{eq:local_sum_sup} becomes bounded as
\begin{align}
\norm{\sup_{0\le t\le R^k}\abs{\U{t}f}}_{L^q(\mu)} &\le 
	CR^\beta\left[\sum_l \left(\int \norm{f_{\tau,l}}_2\,d\tau\right)^q\right]^{1/q} \\
	&\le CR^\beta\int \Big(\sum_l \norm{f_{\tau,l}}_2^2\Big)^{1/2}\,d\tau
\end{align}
where we used $q\ge 2$ and Minkowski's integral inequality.
Thus, to prove \eqref{eq:thm:long_time} it suffices to prove that
\begin{equation}\label{eq:pre_L2}
\int_{\abs{\tau}\le 10R^{-1}}\Big(\sum_l \norm{f_{\tau,l}}_2^2\Big)^{1/2}\,d\tau \le C\norm{f}_2. 
\end{equation}

For that, by Fubini's theorem, we first write
\begin{equation}
\sum_l \norm{f_{\tau,l}}_2^2 = \int\Big(\sum_l \abs{\widehat{f}_{\tau,l}(\xi)}^2\Big)\,d\xi.
\end{equation}
In \eqref{eq:def:f_tau}, we notice that
\begin{equation}
\widehat{\psi}_{R,l}(P(\xi)-P(\eta)+\tau) = 
	R e^{-2\pi i\tau Rl} e^{2\pi i(P(\eta)-P(\xi))Rl}\, \widehat{\psi}(R (P(\xi)-P(\eta)+\tau)),
\end{equation}
so the absolute value of \eqref{eq:def:f_tau} is
\begin{equation}
\abs{ \widehat{f}_{\tau,l}(\xi) } = 
	R \, \left| \int e^{2\pi i(P(\eta)-P(\xi))Rl} \, a_{R,\tau}(\eta,\xi) \, \widehat{f}(\eta)\,d\eta \right|,
\end{equation}
where the amplitude $a_{R,\tau}(\eta,\xi) := 
\widehat{\psi}(R (P(\xi)-P(\eta)+\tau)) \, \widehat{\varphi}_R(\xi-\eta)$
is a smooth function and, 
as function of $\eta$, it is supported in $B(\xi,(c_0R)^{-1})$.

Let us fix $\xi$ and $\tau$.
Observe that $\{\widehat{f}_{\tau,l}(\xi)\}_{l\in\Z}$ are not far from the Fourier coefficients
of $a_{R,\tau}(\cdot,\xi)\widehat{f}$ 
at frequencies $\{\nabla P(\xi)R l\}_{l\in\Z}$.
In the same way as we can prove the inequality
$\int\abs{\widehat F}^2\ind_E\,dx\le \norm{F}_2^2$ using the crude estimate $\ind_E\le 1$,
and the result is best possible,
we can bound the square norm of $\{\widehat{f}_{\tau,l}(\xi)\}_{l\in\Z}$.

Assume that $\nabla P(\xi)$ points in the direction $e_n$, 
which can always be achieved after a rotation, if needed. 
Let $m = (m',m_n) \in \mathbb Z^{n-1} \times \mathbb Z$ and define the Fourier coefficients
\begin{equation}
b(m) := R \int e^{2\pi i R((\eta'-\xi')\cdot m'+(P(\eta)-P(\xi)) m_n)}
	a_{R,\tau}(\eta,\xi)\widehat{f}(\eta)\,d\eta,
\end{equation}
where $\xi = (\xi',\xi_n) \in \mathbb R^{n-1} \times \mathbb R$.
Then,
\begin{align}
\sum_{l\in\Z} \abs{\widehat{f}_{\tau,l}(\xi)}^2 &\le \sum_{m\in\Z^n}\abs{b(m)}^2 \\
	&= R^2 \, \int_{\R^{2n}} a_{R,\tau}(\eta,\xi) \, \widehat{f}(\eta) \, 
		\overline{a_{R,\tau}(\omega,\xi) \, \widehat{f}(\omega)}
		\Big(\sum_m e^{2\pi iR[(\eta'-\omega')\cdot m'+(P(\eta)-P(\omega))m_n]}\Big)\,d\eta d\omega. 
		\label{eq:squareSumBig}
\end{align}

By the Poisson summation formula, the inner sum is 
\begin{equation}
\sum_m e^{2\pi iR[(\eta'-\omega')\cdot m'+(P(\eta)-P(\omega))m_n]} = 
	R^{-n}\sum_{m\in \Z^n} \delta_{m/R}(\eta'-\omega', P(\eta)-P(\omega)).
\end{equation}
Since $a_{R,\tau}(\cdot,\xi)$ is supported in the ball $B(\xi,1/(c_0R))$,
we may work with $\eta,\omega\in B(\xi,1/(c_0R))$.
Let us check which Dirac deltas contribute to the integral \eqref{eq:squareSumBig}.
Suppose that $(\eta'-\omega', P(\eta)-P(\omega)) = m/R$ for some $m\in \Z^n$.
Then, $\abs{m'}/R = \abs{\eta'-\omega'}\le 2/(c_0R)$, which implies $m' = 0$ if we set $c_0>2$.
Similarly, $\abs{m_n}/R = \abs{P(\eta)-P(\omega)} \le 2\sup_{\abs{\xi}\simeq 1} \abs{\nabla P(\xi)}/(c_0R) \leq 1/(5R)$
implies $m_n = 0$.
Hence, only the Dirac delta at $m=0$ in \eqref{eq:squareSumBig} survives.

We have thus
\begin{equation} \label{eq:Bessel_Dirac_delta}
\sum_{l\in\Z} \abs{\widehat{f}_{\tau,l}(\xi)}^2 
	\le R^{2-n} \, \int_{\R^{n+1}} a_{R,\tau}(\eta,\xi) \, \widehat{f}(\eta) \, 
		\overline{a_{R,\tau}((\eta',\omega_n),\xi) \, \widehat{f}((\eta',\omega_n))} \, 
		\delta_0(P(\eta)-P((\eta',\omega_n)))\, d\omega_n \,d\eta .
\end{equation}
Call $G(\omega_n) = a_{R,\tau}(\eta,\xi) \, \widehat{f}(\eta)\overline{a_{R,\tau}((\eta',\omega_n),\xi) \, \widehat{f}((\eta',\omega_n))} $
so that the integral in $\omega_n$ is
\begin{equation}
\int_\R G(\omega_n) \, \delta_0(P(\eta)-P((\eta',\omega_n)))\,d\omega_n.
	\label{eq:Integral_omega_n}
\end{equation}
We want to change variables $r = P(\eta)-P((\eta',\omega_n))$, 
with $dr = -\partial_n P((\eta',\omega_n)) \, d\omega_n$, where
\begin{align}
- \partial_n P((\eta',\omega_n)) &= -\partial_n P(\xi) 
	+ (\partial_n P(\xi) - \partial_n P((\eta',\omega_n))) \\
	&= -\partial_n P(\xi) + \BigO(1/(c_0R)).
\end{align}
The last equality is due to the regularity of $P$ and the support of $a_{R,\tau}(\cdot,\xi)$, since 
\begin{equation}
\left| \partial_n P(\xi) - \partial_n P(\eta',\omega_n) \right| \leq c_2 |\xi - (\eta',\omega_n)| \leq \frac{c_2}{c_0 R}.
\end{equation}
Now, recall that $\nabla P(\xi)$ points in the direction $e_n$, so
$\abs{\partial_n P(\xi)} = \abs{\nabla P(\xi)}$. 
Since $P$ is non-singular by hypothesis, we have
$\inf_{\abs{\xi}\simeq 1}\abs{\nabla P(\xi)}\ge c_1>0$, 
and thus, choosing $c_0$ as large as needed, we get
$\abs{\partial_n P((\eta',\omega_n))}\ge c_1/2 > 0$.
Consequently,  the change of variables is bijective, so
 \eqref{eq:Integral_omega_n} turns into
\begin{equation}
\int_\R   \frac{ G(\omega_n(r)) \, \delta_0(r) }{\left| \partial_n P((\eta',\omega_n(r))) \right|} \, dr
	= \frac{G(\omega_n(0))}{\abs{\partial_n P((\eta',\omega_n(0)))}}
	 \leq \frac{2}{c_1} \, G(\eta_n). 
\end{equation}
In the last inequality we used $\omega_n(0) = \eta_n$, because $0 = r = P(\eta) - P((\eta',\omega_n))$ 
and $r(\omega_n)$ is one-to-one. 

Going back to \eqref{eq:Bessel_Dirac_delta},
\begin{align}
\sum_l \abs{\widehat{f}_{\tau,l}(\xi)}^2	&\le 
	\frac{2}{c_1} \, R^{2-n} \, \int_{\R^n}\abs{a_{R,\tau}(\eta,\xi)\widehat{f}(\eta)}^2\,d\eta 
	\le \frac{2}{c_1} \, \lVert \widehat \psi \rVert_\infty \, R^{2-n} \, \int_{\R^n}\abs{\widehat{\varphi}_R(\xi-\eta)\widehat{f}(\eta)}^2\,d\eta,
\end{align}
where in the last inequality we removed the dependence on $\tau$.
Now integrate in $\xi$ to get
\begin{equation}
\sum_l \norm{f_{\tau,l}}_2^2 \le \frac{C_\psi}{c_1} \, R^{2-n}
	\int \abs{\widehat{f}(\eta)}^2 \left( \int\abs{\varphi_R(\xi)}^2\,d\xi \right) d\eta
	=  C_{\psi,\varphi}\, \frac{c_0^n}{c_1} \, R^2 \, \norm{f}_2^2.
\end{equation}
Finally, take the square root and integrate in $\tau$, 
recalling that $\abs{\tau}\le 10R^{-1}$.
This proves \eqref{eq:pre_L2}, 
hence the lemma is proved.
\end{proof}

\begin{figure}
\centering
\includegraphics[scale=0.75]{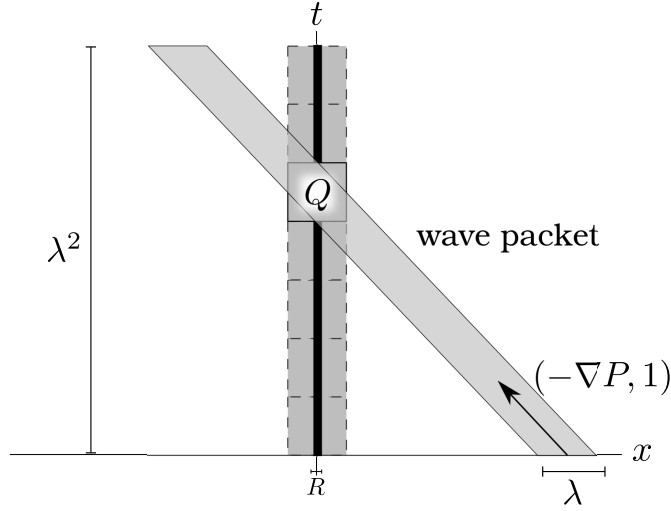}
\caption{Lemma~\ref{thm:Short_Time} can be proved using a wave packet decomposition.
We find it easier to go through scales $\lambda$ from 
$\lambda = R$ up to $\lambda = R^{k/2}$.
We divide the cylinder $B_\lambda\times [0,\lambda^2]$ into cubes $Q$
and assign to each cube all the wave packets passing through it, and then
extract the corresponding contribution $f_Q$ from the initial datum $f$.
We apply the inductive hypothesis to each $\sup_{m\lambda\le t\le (m+1)\lambda}\abs{\U{t}f_Q}$.
All $f_Q$'s are essentially orthogonal, so 
Bessel's inequality holds, \textit{i.e.} 
something like \eqref{eq:pre_L2}.} \label{fig:Lee_proof}
\end{figure}

A proof based on wave packets might be more intuitive, so
we sketch in Figure~\ref{fig:Lee_proof} a proof of a slightly weaker result than Lemma~\ref{thm:Short_Time}.

Now, we perform a further reduction as in \cite[p. 845-846]{MR3961084}.

\begin{lem} \label{thm:Discrete}
Let $W^\alpha(B_R)$ be the collection of weights $w:B_R\subset\R^{n+1}\to \R$ 
with the following properties:
\begin{enumerate}[(i)]
\item $w\ge 0$;
\item $\int_{B_R} w = R^\alpha$;
\item $\int_{B_r(x)}w \le K_w r^\alpha$ holds for every $x\in\R^{n+1}$;
\item $w = \sum_Q w_Q\ind_Q$, where $Q$ are unit cubes that tile $\R^{n+1}$;
\item Every line $t\mapsto (x,t)$ intersects at most one cube in the support of $w$. \label{it:projection}
\end{enumerate}
Suppose that for every $w\in W^\alpha(B_R)$ it holds that
\begin{equation} \label{eq:thm:discrete}
\norm{\U{t}f}_{L^2(w)}\le C_{K_w} R^\beta\norm{f}_2,
\end{equation}
where $\supp \widehat{f}\subset \{\abs{\xi}\simeq 1\}$. Then,
for every $\mu\in M^\alpha(B_R)$ it also holds that
\begin{equation} \label{eq:thm:non-discrete}
\norm{\sup_{0\le t\le R}\abs{\U{t}f}}_{L^2(\mu)}\le C_\mu R^\beta\norm{f}_2.
\end{equation} 
\end{lem}

\begin{rmk} \label{rmk:Barron_counter}
Barron \textit{et al.} \cite{MR4196386} constructed counterexamples for \eqref{eq:thm:discrete} for weights that
do not necessarily satisfy condition \ref{it:projection}.
\end{rmk}

\begin{proof}
Since the Fourier transform of $\abs{\U{t}f}^2$ is supported in $B_1$, 
then we can find $\psi\in\Sz(\R^n)$ such that
$\abs{\U{t}f(x)}^2 = \abs{\U{s}f}^2*\psi(x,t)$.
If we define $\psi_1(x,t) := \sup_{  \abs{(y,s)}\le \sqrt{n}} \left| \psi ((x,t)+(y,s)) \right|$,
which has the same decay properties as $\psi$, then
for every pair of points $(x,t)$ and $(X,T)$ separated by 
a distance less than $\sqrt{n}$ the inequality
$\abs{\U{t}f(x)}^2\le \abs{\U{s}f}^2*\psi_1(X,T)$ holds true.

Let us tile $\R^{n+1}$ with unit cubes $Q = X\times T\subset\R^n\times\R$,
whose center we denote also as $(X,T)$.
According to the previous paragraph, 
the value at $(x,t)$ is controlled by the value
at the center of the cube where it lies,
that is, 
\begin{equation}
\abs{\U{t}f(x)}^2 \le 
	\sum_Q \, \abs{\U{s}f}^2*\psi_1(X,T) \, \ind_Q(x,t).
\end{equation}
For every $X$, we can find some $T(X)$ such that
\begin{equation}
\sup_{0\le t\le R}\abs{\U{t}f(x)}^2 \le 
	\sum_X \, \abs{\U{s}f}^2*\psi_1(X,T(X)) \, \ind_X(x).
\end{equation}
We integrate over $\mu$ to reach
\begin{equation} \label{eq:discrete_I_mu}
\norm{\sup_{0\le t\le R}\abs{\U{t}f}}_{L^2(\mu)}^2\le 
	\sum_X \, \abs{\U{s}f}^2*\psi_1(X,T(X)) \, \mu(X).
\end{equation}

Write now
\begin{equation}
\psi_1(x,t) \le \sum_Q \, b_Q \, \ind_Q(x,t),
\end{equation}
where $b_Q = \sup_{(x,t) \in Q}{\psi_1(x,t)}$ decay rapidly 
because $\psi_1$ decays faster than any power.
Plug it into \eqref{eq:discrete_I_mu} and rearrange terms so that
\begin{equation}
\norm{\sup_{0\le t\le R}\abs{\U{t}f(x)}}_{L^2(\mu)}^2 \le 
	\sum_{Q'}\, b_{Q'} \, \int \abs{\U{t}f(x)}^2\Big(\sum_X\ind_{Q'}(x-X,t-T(X)) \, \mu(X)\Big)\,dx dt.
\end{equation}
Define the weight
\begin{equation}
w(x,t) := \sum_X \, \ind_{X\times T(X)}(x,t) \, \mu(X),
\end{equation}
which satisfies all the properties listed in the statement of the lemma
with $K_w$ independent of $f$.
With it, if $Q' = X' \times T'$, we can write
\begin{equation}
\begin{split}
\norm{\sup_{0\le t\le R}\abs{\U{t}f(x)}}_{L^2(\mu)}^2 & 
	\le \sum_{Q'} \, b_{Q'} \, \int \abs{\U{t}f(x)}^2 \, w(x-X',t-T') \,dx dt \\
	& = \sum_{Q'} \, b_{Q'} \, \int \abs{\U{t + T'}f(x + X')}^2 \, w(x, t) \,dx dt.
	\end{split}
\end{equation}
Observe that 
\begin{equation}
\U{t + T'}f(x + X') = \U{t}(f_{X',T'})(x) \qquad \text{ such that } \qquad \widehat{f_{X',T'}}(\xi) = \widehat f (\xi) \, e^{i \, (X' \xi + i T' P(\xi) ) }. 
\end{equation}
Thus, using the hypothesis \eqref{eq:thm:discrete}, we bound
\begin{equation}
\norm{\sup_{0\le t\le R}\abs{\U{t}f(x)}}_{L^2(\mu)}^2  \leq  C_w^2\, R^{2\beta} \, \sum_{Q'} \, b_{Q'} \, \lVert f_{X',T'} \rVert_2^2 
= C_w^2\, R^{2\beta} \, \lVert f \rVert_2^2  \, \sum_{Q'} \, b_{Q'}.
\end{equation}
The fact that $b_Q$ decay rapidly implies \eqref{eq:thm:non-discrete}.
\end{proof}

In the next lemma we summarize the reductions until now.

\begin{lem}\label{thm:reductions}
Let $P$ be a non-singular and homogeneous symbol. 
Suppose that for every $w\in W^\alpha(B_R)$
as defined in Lemma~\ref{thm:Discrete},
and every $R\gg 1$ it holds that
\begin{equation}
\norm{\U{t}f}_{L^2(w)}\le CR^{s(\alpha)-(n-\alpha)/2}\norm{f}_2,
\end{equation}
where $\supp\widehat{f}\subset\{\abs{\xi}\simeq 1\}$. Then,
for every $\mu\in M^\alpha(B_1)$ the solution
$\U{t}f$ converges $\mu$-a.e. to $f\in H^s(\R^n)$ if $s>s(\alpha)$.
\end{lem} 

With all these reductions we can prove one of the main results, 
Theorem~\ref{thm:non-dispersive_B}.
\begin{customthm}{\ref{thm:non-dispersive_B}}
{\it 
Let $P\in C^\infty(\R^n\setminus\{0\})$ be a non-singular, homogeneous function of
degree $k\ge 1$, $k\in \R$. If $s>(n-\alpha+1)/2$, then
$\U{t}f$ converges  to $f$ $\mu$-a.e. 
for every $f\in H^s(\R^n)$ 
and for every $\mu\in M^\alpha(B_1)$.
}
\end{customthm} 
\begin{proof}
By Lemma~\ref{thm:reductions}, it suffices to prove that
\begin{equation} \label{eq:trace_bound}
\norm{\U{t}f}_{L^2(w)}\le R^{1/2} \, \norm{w}_\infty^{1/2} \, \norm{f}_2,
\end{equation}
where $w\in W^\alpha(B_R)$.
Using Plancherel's theorem and the fact that $w$ is bounded, for every time $t$ we have
\begin{equation}
\int \abs{\U{t}f(x)}^2 \, w(x,t)\,dx \le \norm{w}_\infty\norm{f}_2^2.
\end{equation}
Integrating in $t\in[-R, R]$ 
we get the desired inequality \eqref{eq:trace_bound}.
\end{proof}

In this generality, Theorem~\ref{thm:non-dispersive_B} is best possible for $\alpha>1$.
For example, for the transport equation \eqref{eq:trace_bound} cannot be improved.
What is more, some dispersive symbols like the saddle
behave as
a non-dispersive symbol with especially crafted initial data,
as we see in Theorem~\ref{thm:examples_saddle}.

Recall that for $s < (n-\alpha)/2$ convergence does not hold.
In general, for a fixed symbol $P$, it is a very hard problem
to close the gap between convergence and divergence. 
However, if the symbol is dispersive, 
then we can close it when $\alpha$ is small, as we show now.

\begin{customthm}{\ref{thm:dispersive_B}}
{\it 
Let $P\in C^\infty(\R^n\setminus\{0\})$ be a non-singular, homogeneous function of
degree $k\ge 1$, $k\in \R$.
Suppose there exists $\beta >0$ such that $\norm{\U{t}\varphi}_\infty\le C_\varphi\abs{t}^{-\beta}$
for every $\varphi \in \mathcal S(\R^n) $ with Fourier support in $\{ \abs{\xi}\simeq 1 \}$. 
Then, for $\alpha<\beta$ and $s>(n-\alpha)/2$, 
$\U{t}f$ converges  to $f$ $\mu$-a.e. 
for every $f\in H^s(\R^n)$ 
and for every $\mu\in M^\alpha(B_1)$.
}
\end{customthm} 
\begin{proof}
We use Lemma~\ref{thm:reductions} to prove this. 
It suffices to work with data $f$ 
Fourier supported in the annulus $\{ |\xi| \simeq 1\}$, 
so let $\varphi \in \mathcal S(\R^n)$ be a cut-off function 
such that $\varphi = 1$ in that annulus
and write 
\begin{equation}
T_tf(x) = \int \, \varphi(\xi)\, \widehat{f} (\xi) \, e^{2 \pi i (x \cdot \xi + tP(\xi))} \, d\xi.
\end{equation} 
We will first prove that
\begin{equation} \label{eq:dispersive_Bound}
\norm{\U{t}f}_{L^2(w)}\le C \, \norm{ \left| (dS)^\vee \right| * w}_\infty^{1/2} \, \norm{f}_2, 
\end{equation}
and then we will check that $\norm{ \left| (dS)^\vee \right| * w}_\infty^{1/2} \leq C$, 
which by Lemma~\ref{thm:reductions} implies the result. 

By Plancherel's theorem, defining the operator $E$ such that $Ef = T_t (f^\vee)$, 
it is enough to prove that
\begin{equation} \label{eq:dispersive_Bound_2}
\norm{Ef}_{L^2(w)}\le C \, \norm{ \left|(dS)^\vee \right| * w}_\infty^{1/2} \, \norm{f}_2.
\end{equation}
By duality, 
\begin{equation}\label{eq:Adjoint_With_Weight}
\norm{Ef}_{L^2(w)} = \sup_{\lVert g w^{1/2} \rVert_2=1} \langle Ef w^{1/2}, gw^{1/2} \rangle  \leq \lVert f \rVert_2 \, \sup_{\lVert g \rVert_{L^2(w)}=1} \lVert E^*(g w) \rVert_2
\end{equation}  
where the adjoint operator $E^*$
is essentially the restriction operator attached to $S$, that is,
\begin{equation}
E^*F(\xi) = \varphi(\xi) \,  \int F(x,t) \, e^{- 2\pi i(x\cdot \xi + tP(\xi))}\,dx dt.
\end{equation}
Rearranging the integrals, we get 
\begin{equation}
\lVert E^*(g w) \rVert_2^2 = \langle E^*(g w), E^*(g w)  \rangle = \int gw \,  \overline{ \left(  gw \ast (dS)^\vee \right)} \, dx dt \leq \lVert g \rVert_{L^2(w)} \, \lVert  gw \ast (dS)^\vee  \rVert_{L^2(w)},
\end{equation}
where $(dS)^\vee(x,t) = \int \varphi(\xi)^2 \, e^{i(x\xi + tP(\xi))}\, d\xi$.
Thus, to prove \eqref{eq:dispersive_Bound_2} it is enough to prove 
\begin{equation}\label{eq:Convolution_Bound}
\lVert  gw \ast (dS)^\vee  \rVert_{L^2(w)} \leq C \,  \norm{\left| (dS)^\vee \right| * w}_\infty \, \lVert g \rVert_{L^2(w)}.
\end{equation}
We prove \eqref{eq:Convolution_Bound} by interpolation
between $L^\infty(w)\to L^\infty(w)$ and $L^1(w)\to L^1(w)$.
Indeed, the $L^\infty$ bound follows from
\begin{align}
\abs{ gw * (dS)^\vee(x,t)} &\le 
	\int \left| (gw)(y,s)(dS)^\vee(x-y,t-s) \right| \,dy ds \\
	&  = \int \, \left|  (g \ind_{\supp w})(y,s)  \,  w(y,s)(dS)^\vee(x-y,t-s) \right| \,dy ds \\
	& \le \norm{g}_{L^\infty(w)} \, \left( w \ast \left| (dS)^\vee \right| \right) (x,t),
\end{align}
while the $L^1$ bound holds because
\begin{align}
\int \abs{(gw)*(dS)^\vee} \, w\,dxdt &\le 
	\int \abs{gw(y,s)}\int \abs{w(x,t)(dS)^\vee(x-y,t-s)}\,dxdt\,dyds \\
	&\le \norm{g}_{L^1(w)} \, \norm{\left|(dS)^\vee \right| * w}_\infty.
\end{align}
This implies \eqref{eq:Convolution_Bound} and so \eqref{eq:dispersive_Bound} is proved.

Let us bound $\norm{ \left|(dS)^\vee \right| * w}_\infty$ now.
By the non-stationary phase principle, 
\begin{equation}
\abs{(dS)^\vee(x,t)}\le C\sup_{2\le j\le N}\norm{D^jP}_\infty \, \abs{x}^{-N}, \quad 
	\textrm{for every } N\ge 1 \textrm{ and } \abs{x}\ge 2c_0\abs{t},
\end{equation}
where $c_0 := \sup_{\abs{\xi}\simeq 1}\abs{\nabla P(\xi)}$ and
$C$ depends on $c_0$ and $\varphi$, but
is otherwise independent of time.
On the other hand, by hypothesis we have 
$\abs{(dS)^\vee(x,t)}\le C\abs{t}^{-\beta}$.
Combining the two bounds, we obtain 
\begin{equation}
\left| (dS)^\vee (x,t) \right| \leq \frac{C}{|(x,t)|^\beta}, \qquad \text{ for any } (x,t), 
\end{equation}
so by decomposing the space in dyadic annuli we deduce that
\begin{equation}
\abs{(dS)^\vee(x,t)} \le 
	C\sum_{\lambda\ge 1, \,  \lambda \text{ dyadic}} \, \lambda^{-\beta} \, \ind_{B_\lambda}(x,t).
\end{equation}
Hence,
\begin{align}
\int \left| (dS)^\vee(x-y,t-s) \right| \, w(y,s)\,dyds 
	&\le C\sum_{\lambda\ge 1}\lambda^{-\beta}\int_{B_\lambda(x,t)}w  \\
	&\lesssim \sum_{1\le\lambda\le R}\lambda^{-\beta+\alpha} + R^\alpha\sum_{R\le \lambda}\lambda^{-\beta} \\
	& \leq \sum_{\lambda \geq 1} \lambda^{-\beta + \alpha}, 
\end{align}
which is bounded if $\alpha < \beta$. 
In that case, $\norm{\abs{(dS)^\vee}*w}_\infty\lesssim 1$,
which we insert into \eqref{eq:dispersive_Bound} and 
apply Lemma~\ref{thm:reductions} to conclude the proof.
\end{proof}


\section{Building the counterexample} \label{sec:Building_counterexample}

We begin the proof of Theorem~\ref{thm:divergence_k}. 
Recall that we consider symbols 
\begin{equation}
P(\xi) = \xi_1^k + W(\xi'),
\end{equation}
where $W\in \Q[X_2,\cdots,X_n]$ is a polynomial of degree $k\ge 2$.
It suffices to work with $W\in \Z[X_2,\cdots,X_n]$; 
indeed, given that there exists $ a \in \mathbb Z$ such that $aW\in \Z[X_2,\cdots,X_n]$, 
we study $T_{at}f(x)$ when $t \to 0$ instead of $T_tf(x)$.

\subsection{A preliminary datum}\label{sec:Preliminary_Datum}

Let $0 < c \ll 1$ and choose two positive functions $\phi_1 \in \mathcal S(\R)$ and
$\phi_2 \in \mathcal S(\R^{n-1})$ such that
 $\operatorname{supp} \widehat{\phi}_i \subset B(0,c)$. 
Let $\psi\in \mathcal S(\R^{n-1})$ be another positive function 
with $\operatorname{supp} \psi \subset B(0,1)$.
For $R > 1$ and $D>1$, let us define
\begin{equation}\label{eq:Counterexample}
f_R(x) =  \phi_1(R^{1/2}x_1) \, e^{2\pi i R x_1 } \, 
	\phi_2(x') \, \sum_{m' \in\Z^{n-1}} \psi\left(\frac{Dm' }{R} \right)e^{2\pi i D m' \cdot x'}
\end{equation}
where $ x = (x_1,\ldots, x_n) = (x_1,x') \in B(0,1)$.
Defining 
\begin{equation}
g(x_1) = \phi_1(R^{1/2}x_1) \, e^{2\pi i R x_1 } \qquad \textrm{ and } \qquad
h(x') = \phi_2(x') \, \sum_{m' \in\Z^{n-1}} \psi\left(\frac{Dm' }{R} \right)e^{2\pi i D m' \cdot x'},
\end{equation}
we may write $f_R(x) = g(x_1)\, h(x')$.
In particular, $\widehat{f_R}(\xi) = \widehat{g}(\xi_1)\, \widehat{h}(\xi')$, where 
\begin{equation}
\widehat{g}(\xi_1) = R^{-1/2} \, \widehat{\phi}_1\big( R^{-1/2} (\xi_1 - R) \big) \qquad \text{ and } \qquad \widehat{h}(\xi') = \sum_{m' \in\Z^{n-1}} \psi\left(\frac{Dm'}{R} \right) \widehat{\phi}_2 (\xi' - D m' ).
\end{equation}
Thus, the $L^2(\R^n)$ norm of $f_R$ is 
\begin{equation}\label{eq:Norms_Counterexample}
\lVert f_R \rVert_2 \simeq R^{-1/4} \, \left( \frac{R}{D} \right)^{(n-1)/2}.
\end{equation}

Let us analyze the evolution of $f_R$.
In the $x_1$ variable, the evolution is
\begin{equation}\label{eq:Evolution_First_Component}
\begin{split}
\U{t} g(x_1) & = R^{-1/2} \, \int \widehat{\phi}_1(R^{-1/2}\, \xi_1) \, e(x_1\, (\xi_1 + R) + t( \xi_1 + R)^k)\, d\xi_1 \\
	& = e^{2\pi i (x_1 R + t R^k)} \, \int \widehat{\phi}_1( \xi_1) \,
	e\Big( R^{1/2}\,  \xi_1 (x_1 +  t k R^{k-1}) + t \sum_{j=2}^k \binom{k}{j}\, \xi_1^j\, R^{k-j/2} \Big)\, d\xi_1.
\end{split}
\end{equation}
Observe that $\left| \sum_{j=2}^k \binom{k}{j}\, \xi_1^j\, R^{k-j/2}   \right| \leq R^{k-1}\,  \sum_{j=2}^k \binom{k}{j} $, so 
asking $\abs{t} < 1/R^{k-1}$ we essentially have 
\begin{equation}\label{eq:Heuristic_First_Component}
\left| \U{t} g(x_1) \right| \simeq \left| \phi_1\left( R^{1/2} \big(x_1 + k R^{k-1} t \big) \right)  \right|.
\end{equation}
In the remaining variables $x'$, the evolution is
\begin{align}
\U{t} h(x') &=  \sum_{m' \in\Z^{n-1}} \psi \left(\frac{Dm'}{R}\right) \int \widehat{\phi}_2(\xi') \, 
	e\big(x'\cdot(\xi' + Dm') + t \, W(\xi'+Dm')\big) \, d\xi' \\
\begin{split}\label{eq:Evolution_Last_Components}
	&= \sum_{m'\in\Z^{n-1}} e^{2\pi i(x'\cdot Dm' + t \, W_k(Dm'))} \\
	&\qquad \psi \left(\frac{Dm'}{R}\right) \int \widehat{\phi}_2(\xi') \, 
		e\left( x'\cdot\xi'+ t(W(Dm'+\xi') - W_k(Dm'))\right) \, d\xi',
	\end{split}
\end{align}
where $W_k$ is the homogeneous part of $W$. 
In particular, $W_k$ is homogeneous of degree $k$.
Let us evaluate the solution at the points
\begin{equation}\label{eq:Points_For_Solution}
t = \frac{p_1}{D^k \, q} \qquad\textrm{and}\qquad x' = \frac{p'}{D\, q} + \epsilon, 
\end{equation}
where $p = (p_1,p') \in \mathbb Z^n$ and $q \in \mathbb N$ with $q\simeq Q$ for some $Q \ge 1 $ to be chosen later, and $|\epsilon| \lesssim R^{-1}$. 
Thus,
\begin{equation}\label{eq:h_pre_perturbation}
\begin{split}
\U{t} h(x') &  = \sum_{m'\in\Z^{n-1}} e\left( \frac{p'\cdot m' + p_1W_k(m')}{q} \right)  \\
& \qquad \qquad \left[ e^{2\pi i\,\epsilon \cdot Dm'}  \psi \left(\frac{Dm'}{R}\right) 
	\int \widehat{\phi}_2(\xi') \, e\left(x'\cdot\xi'+ t(W(Dm'+\xi') - W_k(Dm'))\right) \, d\xi' \right] \\
& = \sum_{m'\in\Z^{n-1}} e\left( \frac{p'\cdot m' + p_1W_k(m')}{q} \right) \,   \zeta(m'),
	\end{split} 
\end{equation}
where we define 
\begin{equation} \label{eq:def_zeta}
\zeta(m') = e^{2\pi i\,\epsilon \cdot Dm'}  \psi \left(\frac{Dm'}{R}\right) 
	\int \widehat{\phi}_2(\xi') \, e\left( x'\cdot\xi'+ t \, (W(Dm'+\xi') - W_k(Dm'))\right) \, d\xi'.
\end{equation}
Heuristically, if we assume that $R/D\gg Q$, the function $\zeta$
is roughly constant
at scale $q$, 
so we should be able to sum in blocks modulo $q$. 
In that case, since the support of $\zeta$ is contained in $|m'| \leq R/D$, 
the solution would be
\begin{equation} \label{eq:h_pre_Weil}
\left| \U{t} h(x') \right| \simeq 
	\left(\frac{R}{DQ}\right)^{n-1}\Bigg| \sum_{r' \in\F^{n-1}_q} e\left( \frac{p'\cdot r' + p_1W_k(r')}{q} \right) \Bigg|.
\end{equation}
At this point, we have two goals:
\begin{itemize}
	\item To make \eqref{eq:h_pre_Weil} rigorous, and
	\item To estimate the exponential sum in \eqref{eq:h_pre_Weil}.
\end{itemize}
As we shall see, estimating the exponential sum
is necessary to prove \eqref{eq:h_pre_Weil},
so let us begin by that. 

First, notice that the exponential sum is the discrete inverse Fourier transform
of the function $S(r_1, r') = \delta_0(W_k(r')-r_1)$
in the variables $(r_1, r')\in\F^n_q$, that is,
\begin{equation}\label{eq:Weil_Fourier}
S(r_1, r') = \delta_0(W_k(r')-r_1) \qquad \Longrightarrow \qquad \widecheck{S}(p) = \sum_{r'\in\F^{n-1}_q} e\left( \frac{p'\cdot r' + p_1W_k(r')}{q} \right).
\end{equation}
Notice also that $\widecheck{S}(p)$ is the extension operator 
attached to the surface $\{(W_k(r'), r')\in\F^n_q\}$.



Estimating this kind of sums is difficult in general. 
In  \cite{Pierce2021},
An, Chu and Pierce applied Deligne's Theorem~8.4 from \cite{MR340258},
a deep result in arithmetic geometry, which we write here.
The reader is also referred to Theorem~11.43 in \cite{MR2061214}.

\begin{thm}[Deligne's theorem adapted] \label{thm:Weil_Bound}
Let $f\in\Z[X_1,\ldots,X_d]$ be a nonzero polynomial of degree $k$ 
with homogeneous part $f_k$. Suppose that
\begin{enumerate}[(i)]
\item $q$ is a prime number and $q\nmid k$.
\item $\nabla f_k(x)\neq 0$ for every $x\in\overline{\F_q}^d\setminus\{0\}$.
\label{it:non-singular}
\end{enumerate}
Then,
\begin{equation} \label{eq:Weil_Bound}
\Big|\sum_{x\in\F^d_q} e(f(x)/q)\Big|\le (k-1)^d q^{d/2}.
\end{equation}
\end{thm}

We first remark that we need \eqref{eq:Weil_Bound}
to hold for all but finitely many primes $q$.
In this setting, if a polynomial $f\in\Z[X_1,\ldots,X_d]$ 
satisfies the hypotheses of Deligne's theorem
for all but finitely many primes, 
then we say that $f$ satisfies the Weil bound.

Our second remark is that 
if $f$ has degree $\geq 2$ and satisfies 
the Weil bound, 
then the polynomial  $af(x) + l(x)$ with 
any integer $a$ which is coprime with $q$
and any linear form $l$ 
also satisfies the Weil bound.
Thus, for our sum in \eqref{eq:Weil_Fourier} it is enough
to check that $W_k$ satisfies the Weil bound.

Let us comment on condition \ref{it:non-singular} now, 
about which the reader can find a discussion in \cite{sawinMO}.
It is equivalent to the statement that the variety (scheme) 
defined by $f_k$ in $\mathbb{P}^{d-1}(\F_q)$ is smooth.
However, we want to find conditions 
that make a polynomial satisfy the Weil bound
that are easier to verify than condition \ref{it:non-singular}.
We do that in the following corollary.
\begin{cor}\label{thm:Nullstellensatz_Cor}
Let $f\in\Z[X_1,\ldots,X_d]$ be a homogeneous polynomial. 
If $\nabla f(x)\neq 0$ for every $x\in\C^d\setminus\{0\}$, then
$f$ satisfies the Weil bound, 
that is, 
\eqref{eq:Weil_Bound} holds for all but finitely many primes.
\end{cor}

To prove it, we need Hilbert's Nullstellensatz, 
which we recall here in the version written in 
\cite[Theorem 14, Ch. VII.3]{MR0120249}.
\begin{thm}[Hilbert's Nullstellensatz]
Let $k$ be a field. If $F, F_1,\ldots, F_m$ are polynomials in $k[X_1,\ldots, X_d]$ and
if $F$ vanishes at every common zero of $F_1,\ldots, F_m$ 
in an algebraically closed extension $K$ of $k$, then
there exists an exponent $\rho \in \mathbb N$ and polynomials $A_1,\ldots, A_m \in k[X_1,\ldots, X_d]$ such that
\begin{equation}
F^\rho = A_1F_1 + \cdots + A_mF_m.
\end{equation}
\end{thm}

\begin{proof}[Proof of Corollary~\ref{thm:Nullstellensatz_Cor}]
We must prove that $\nabla f(x)\neq 0$ for every $x\in\overline{\F_q}^d\setminus\{0\}$
and all but finitely many primes $q$.

By Hilbert's Nullstellensatz with $k = \Q$ and $K = \C$,
we can find a natural number $N$ and polynomials $A_{ij}\in \Q[X_1,\ldots, X_d]$, for $i,j=1,\ldots,d$,
such that
\begin{equation}
X_i^N = A_{i1} \, \partial_1 f + \cdots + A_{id} \, \partial_d f, 
	\qquad \quad  \textrm{for all } i = 1, \ldots, d.
\end{equation} 
We multiply the polynomials above by some $M\in\Z$ such that $MA_{ij}\in\Z[X_1,\ldots,X_d]$ for every $i,j$, and hence
\begin{equation}\label{eq:Polynomials}
M\, X_i^N = (MA_{i1}) \, \partial_1 f+\cdots+(MA_{id}) \, \partial_d f,
	\qquad  \quad  \textrm{for all } i = 1, \ldots, d.
\end{equation}
These are equalities among polynomials with integer coefficients. 
Thus, for every prime $q \nmid M$ (in particular, for $q > M$) 
we can reduce modulo $q$, 
so that \eqref{eq:Polynomials} holds in $\mathbb F_q$. 
This implies that if $\nabla f(x) = 0$ for some $x \in \overline{\mathbb F_q}^d$,
then necessarily $x = 0$.
The proof is thus complete.
\end{proof}

Thus, Corollary~\ref{thm:Nullstellensatz_Cor} implies that 
if the homogeneous part $W_k$ of $W \in \mathbb Z[X_2, \ldots, X_n]$ 
satisfies 
$\nabla W_k(x) \neq 0$ for every $x \in \mathbb C^{n-1}\setminus\{0\}$,
then both $W$ and $W_k$
satisfy the Weil bound. 
In particular, we bound \eqref{eq:Weil_Fourier} by  
\begin{equation}\label{eq:Bound_Exp_Sum}
\big| \widecheck{S}(p) \big|  = \Bigg| \sum_{r'\in\F^{n-1}_q} e\left( \frac{p'\cdot r' + p_1W_k(r')}{q} \right) \Bigg| \leq (k-1)^{n-1} \, q^{(n-1)/2}
\end{equation}
for every $p_1 \not\equiv 0 \pmod{q}$ and for all but finitely many primes $q$.

Once we have estimated the exponential sum, 
let us justify \eqref{eq:h_pre_Weil}.
We do that in an abstract setting, like we did with the exponential sum.
\begin{lem} \label{thm:perturbation}
Let $f\in\Z[X_1,\dots,X_d]$ be a polynomial of degree $\ge 2$ that
satisfies the Weil bound.
Let also $\zeta\in C_0^\infty$ and define the discrete Laplacian $\widetilde\Delta$ by
\begin{equation}
\widetilde{\Delta} \zeta(m) = \sum_{j=1}^d \big( \zeta(m + e_j) + \zeta(m - e_j) - 2\zeta(m) \big).
\end{equation}
Assume that, for some $L >0$, $\zeta$ is supported in $B(0,L)$ and 
$\norm{\widetilde\Delta^N\zeta}_\infty\lesssim_N L^{-2N}$ for every $N \in \mathbb N$.
Then,
\begin{equation} \label{eq:thm:perturbation}
\sum_{m\in \Z^d} \zeta(m)e(f(m)/q) = 
	\left(\frac{1}{q^d}\sum_{m\in\Z^d} \zeta(m)\right)\sum_{l\in\F_q^d} e(f(l)/q) + 
	\BigO_N\left(q^{d/2}\left(\frac{L}{q}\right)^{d-2N}\right)
\end{equation}
for any integer $N> d/2$.
\end{lem}
\begin{proof}
We rearrange the sum into blocks modulo $q$ so that
\begin{equation} \label{eq:block_sum}
\sum_{m\in\Z^d} \zeta(m)e(f(m)/q) 
	= \sum_{r\in\F^d_q}\Bigg(\sum_{m \equiv r\Mod{q}}\zeta(m)\Bigg)e(f(r)/q) 
	= \sum_{r\in\F^d_q}  Z(r) \, e(f(r)/q),
\end{equation}
where we define $Z(r) = \sum_{m \equiv r\Mod{q}}\zeta(m)$ for every $r \in \mathbb F_q^d$. 
Decompose $Z$ into frequencies 
\begin{equation}
\widehat{Z}(\omega) = \frac{1}{q^d}\sum_{r \in \F^d_q}Z(r)e^{-2\pi i\omega\cdot r/q} \qquad \text{ such that } \qquad Z(r) = \sum_{\omega\in\F^d_q} \widehat{Z}(\omega)e^{2\pi i \omega\cdot r/q},
\end{equation}
and replace them into \eqref{eq:block_sum} so that
\begin{equation}
\sum_{m\in\Z^d} \zeta(m)e(f(m)/q) = 
	\widehat{Z}(0)\sum_{m\in\F^d_q}e(f(m)/q) \,+
	\sum_{ \omega \in \mathbb F_q^d \setminus \{0\} }\widehat{Z}(\omega)\Bigg( \sum_{r\in\F^d_q}e\left(\frac{f(r)+\omega\cdot r}{q}\right)\Bigg).
\end{equation}
Since $\widehat{Z}(0) = q^{-d}\sum_{m\in\Z^d}\zeta(m)$, to prove \eqref{eq:thm:perturbation} 
we need to control the error term
\begin{equation}
E = \sum_{ \omega \in \mathbb F_q^d \setminus \{0\} }\widehat{Z}(\omega)\Bigg( \sum_{r\in\F^d_q}e\left(\frac{f(r)+\omega\cdot r}{q}\right)\Bigg).
\end{equation}
Observe that
\begin{align}
\left| E \right| \leq  \norm{\widehat{Z}}_{\ell^1(\F^d_q\setminus\{0\})} \, \sup_{\omega \in \mathbb F^d_q} \Bigg|\sum_{r \in\F^d_q} e\left(\frac{f(r)+\omega\cdot r}{q}\right)\Bigg|
		\label{eq:error_term}
\end{align}
From the Weil Bound in Theorem~\ref{thm:Weil_Bound} we know that 
\begin{equation}
\Big|\sum_{r \in\F^d_q} e\left(\frac{f(r)+\omega\cdot r}{q}\right)\Big|\lesssim q^{d/2}.
\end{equation}
To bound $\widehat{Z}(\omega)$, we sum by parts using the discrete Laplacian $\widetilde{\Delta}$. 
First, observe that for a fixed $\omega$, 
\begin{align}
-\widetilde{\Delta} \left( e^{-2\pi i\omega\cdot r/q} \right)
	& = \sum_{j=1}^d  \left( 2e^{-2\pi i \, \omega\cdot r/q} - e^{-2\pi i\, \omega\cdot(r+e_j)/q} - e^{-2\pi i \, \omega\cdot (r-e_j)/q} \right) \\ 
	&=  e^{-2\pi i\omega\cdot r/q} \, \sum_{j=1}^d \sin^2 (\pi \omega_j/q) = e^{-2\pi i\omega\cdot r/q}A(\omega),
\end{align}
where $\{e_j\}_j$ is the canonical basis and $A(\omega) = 4 \, \sum_{j=1}^d \sin^2 (\pi \omega_j / q)$.
Thus, 
\begin{equation}
\widehat{Z}(\omega) = - \frac{1}{q^d \, A(\omega)} \,  \sum_{r \in\F^d_q}Z(r) \, \widetilde{\Delta} \left( e^{-2\pi i \omega\cdot r/q} \right)
	= - \frac{1}{q^d \, A(\omega)} \, \sum_{r \in\F^d_q} \,  \widetilde{\Delta} \left(Z(r) \right) e^{-2\pi i\omega\cdot r/q}.
\end{equation}
Iterating, we get
\begin{equation}
\widehat{Z}(\omega) =  \frac{(-1)^N}{q^d \, (A(\omega))^N} \, \sum_{r \in\F^d_q} \,  \widetilde{\Delta}^N \left(Z(r) \right) e^{-2\pi i\omega\cdot r/q}, \qquad \forall N \in \mathbb N.
\end{equation}
Since $\sin^2(\pi x)\ge c\norm{x}^2$, where
$\norm{x}$ is the distance of $x$ to the closest integer, then
$A(\omega) \gtrsim \sum_{j=1}^d \norm{\omega_j/q}^2$, so
\begin{equation}
\begin{split}
\big| \widehat{Z}(\omega) \big|  & \lesssim \frac{1}{q^d} \, \Big( \sum_{j=1}^d \norm{\omega_j/q}^2 \Big)^{-N} \, \Bigg|  \sum_{r \in\F^d_q} \,  \widetilde{\Delta}^N \left(Z(r) \right) e^{-2\pi i\omega\cdot r/q} \Bigg| \\
& \leq \Big( \sum_{j=1}^d \norm{\omega_j/q}^2 \Big)^{-N}\, \lVert \widetilde\Delta^N Z \rVert_\infty, \qquad \forall N \in \mathbb N.
\end{split}
\end{equation}
As long as $2N > d$, we can write
\begin{equation}
\sum_{\omega \in \mathbb F_q^d\setminus\{0\}} \Big( \sum_{j=1}^d \norm{\omega_j/q}^2 \Big)^{-N} \simeq 
\sum_{\omega \in \{1, \ldots, \frac{q-1}{2} \}^d} \frac{ q^{2N}}{ |\omega|^{2N} } \simeq_N q^{2N},
\end{equation}
so 
\begin{equation}
\lVert \widehat Z \rVert_{\ell^1(\mathbb F_q^d \setminus \{0\})} \lesssim q^{2N } \,  \lVert  \widetilde\Delta ^N Z \rVert_\infty, \qquad \forall N \in \mathbb N,
\end{equation}
and thus,
\begin{equation}
\left| E \right| \lesssim q^{2N + d/2} \lVert  \widetilde\Delta^N Z \rVert_\infty, \qquad \forall N \in \mathbb N.
\end{equation}
Now, from the definition of $Z$ and the hypotheses for $\zeta$ we see that
\begin{equation}
\abs{\widetilde\Delta^N Z(r)} = \Big|\sum_{l\in\Z^d} \widetilde\Delta^N\zeta(l q + r)\Big| \lesssim
	\left( \frac{L}{q} \right)^d \norm{\widetilde\Delta^N\zeta}_\infty \lesssim_N \frac{L^{d-2N}}{q^d}, \qquad \forall r \in \mathbb F_q^d,
\end{equation}
so
\begin{equation}
\abs{E} \lesssim_N q^{d/2}\, \left(\frac{L}{q}\right)^{d-2N}.
\end{equation}
\end{proof}

We can now rigorously estimate $\U{t}h$ 
as it appeared in \eqref{eq:h_pre_perturbation}.
For that, we use Lemma~\ref{thm:perturbation} with $d=n-1$
 and $L = R/D$.
We only need to bound $\widetilde \Delta^N \zeta$ 
as indicated in the hypotheses
for the function $\zeta$ defined in \eqref{eq:def_zeta}.

\begin{lem}
Let $\zeta$ be defined as in \eqref{eq:def_zeta}, and let $L = R/D$. Then
$\norm{\widetilde\Delta^N\zeta}_\infty\lesssim_N L^{-2N}$ for every $N \in \mathbb N$.
\end{lem}
\begin{proof}
Since $\norm{\widetilde\Delta^N\zeta}_\infty\lesssim \sup_{\abs{\alpha} = 2N}\norm{\partial^\alpha\zeta}_\infty$,
where $\alpha = (\alpha_1,\ldots,\alpha_{n-1})$ is a multi-index,
it suffices to bound
\begin{equation}
\partial^\alpha\zeta(m')	
 = \int \widehat{\phi}_2(\xi') \, 
	e^{2\pi ix'\cdot\xi'} \, \partial_{m'}^\alpha \left[\psi \left(\frac{m'}{L}\right) 
	e\left( \epsilon \cdot m' \frac{R}{L}  +  t \, \left(W\left(\frac{R}{L}\, m' + \xi'\right)-W_k\left(\frac{R}{L}m' \right)\right)\right) \right] \, d\xi'. 
	\end{equation}
Let us write $W(\xi) = \sum_{ i=0  }^k  W_{k-i} (\xi)$,
where $W_j$ is the $j$-homogeneous part of $W$. 
Then, 
\begin{equation}
\begin{split}
W\left(\frac{R}{L}\, m' + \xi'\right)-W_k\left(\frac{R}{L}m' \right) 
&  = W_k \left(\frac{R}{L}\, m' + \xi'\right)-W_k\left(\frac{R}{L}m' \right) 
+ \sum_{ i=1 }^k W_{ k-i }\left(\frac{R}{L}\, m' + \xi'\right)\\
& = R^k \left(  W_k \left(\frac{m'}{L} + \frac{\xi'}{R} \right)-W_k\left(\frac{m'}{L}\right)  \right) 
+ \sum_{ i=1 }^k R^{k-i} W_{ k-i }\left(\frac{m'}{L}+ \frac{\xi'}{R}\right).
\end{split}
\end{equation}
Call $\delta = R\epsilon$ and $\tau = R^{k-1}t$ 
so that $\abs{\delta}\le 1$ and $\abs{\tau}\le 1$.
Defining the phase function
\begin{equation}
F(z') = \delta \cdot z'  + R\tau\,\left(W_k(z' + \xi'/R) - W_k(z') + \sum_{ i=1 }^k R^{-i} W_{ k-i }(z'+ \xi'/R) \right),
\end{equation}
we may write
\begin{equation}
\partial^\alpha\zeta(m')	 
 = \int \widehat{\phi}_2(\xi') \, 	e^{2\pi ix'\cdot\xi'} 
 \, \partial_{m'}^\alpha \left[ \psi(m'/L) \, e(F(m'/L)) \right ] \, d\xi'.
\end{equation}
Since for a function $G$ we have $\partial_{m'}^\alpha \left( G(m'/L) \right) = L^{-2N} \, \partial^\alpha G (m'/L) $, it suffices to prove that
\begin{equation}
\left| \partial_{z'}^\alpha \Big[\psi(z') \, e(F(z'))\Big] \right| \lesssim 1 
	\qquad \textrm{uniformly for }\abs{z'}\le 1.
\end{equation}
Since $\abs{\partial^\alpha F(z')} \lesssim 1$ uniformly in $z'$, $\epsilon$, $\tau$, $R$ and $\xi'$, 
the result follows from the Leibniz rule.
\end{proof}

Thus, we apply Lemma~\ref{thm:perturbation}
with the first integer $N$ that satisfies $N > (n-1)/2$. 
Due to the smallness of all phases in the definition of $\zeta$, 
we get $\sum_{m' \in \mathbb Z^{n-1}} \zeta(m') \simeq (R/D)^{n-1}$,
so we obtain
\begin{equation}\label{eq:Ready_For_Weil}
\left| \U{t} h(x') \right| =
	\left(\frac{R}{DQ}\right)^{n-1} \abs{\widecheck{S}(p)} + 
	\BigO\left(Q^{(n-1)/2}\Big(\frac{R}{DQ}\Big)^{-1}\right).
\end{equation}
Now, in \eqref{eq:Bound_Exp_Sum}
we estimated $\widecheck{S}(p)$ from above, 
but we also need an estimate from below. 
With the aid of Deligne's theorem, An, Chu and Pierce showed in \cite[Proposition~2.2]{Pierce2021} that
$\abs{\widecheck{S}(p)}\gtrsim q^{(n-1)/2}$ for most $p \in \F^n_q$. 
For completeness, we repeat here their arguments.
\begin{lem}\label{thm:G_q}
If $W_k$ satisfies the Weil Bound, then, for $q\gg 1$,
\begin{equation}
\abs{\widecheck{S}(p)} = 
	\Bigg|\sum_{r'\in\F^{n-1}_q} e\left( \frac{p'\cdot r' + p_1W_k(r')}{q} \right)\Bigg| \ge \frac{1}{10} \, q^{(n-1)/2}
\end{equation}
for every $p \in G(q)\subset\mathbb{F}_q^n$, 
where $\abs{G(q)}\ge C_{k,n} q^n$.
\end{lem}
\begin{proof}
Following \eqref{eq:Weil_Fourier}, Plancherel's theorem gives
\begin{equation}
\sum_{p\in\F^n_q}\abs{\widecheck{S}(p)}^2 = 
	q^n\sum_{r \in\F^n_q} \delta(W_k(r')-r_1) = q^{2n-1}.
\end{equation}
For $c_1 >0$, define $G(q) = \{p\in\F^n_q \mid \abs{\widecheck{S}(p)}\ge c_1q^{(n-1)/2}\}$.
Since the Weil bound implies $\abs{\widecheck{S}(p)}\le (k-1)^{n-1}q^{(n-1)/2}$,  we have
\begin{equation}
q^{2n-1} = \sum_{p\in\F^n_q}\abs{\widecheck{S}(p)}^2 \le 
	q^{2n-2} + \abs{G(q)\setminus\{0\}}(k-1)^{2(n-1)}q^{n-1} + c_1^2 \, q^{2n-1}.
\end{equation}
By choosing $c_1$ small enough we obtain $\abs{G(q)}\gtrsim q^n$.
\end{proof}

Applying Lemma~\ref{thm:G_q} in \eqref{eq:Ready_For_Weil}, at every $p \in G(q)$ one gets
\begin{equation} \label{eq:sol_h_lowerBound}
\abs{\U{t} h(x')} \gtrsim
	\Big(\frac{R}{DQ}\Big)^{n-1}  \, Q^{(n-1)/2} + \BigO\left(Q^{(n-1)/2}\Big(\frac{R}{DQ}\Big)^{-1}\right) 
	\simeq \left(\frac{R}{DQ^{1/2}}\right)^{n-1},
\end{equation}
where we require $R/D\gg Q$ for the last operation.
We remark that the case $q = Q = 1$, 
which is relevant in some parts of the following sections, 
is not covered by the arguments above.
However, \eqref{eq:sol_h_lowerBound} still holds by direct computation.

We join estimates \eqref{eq:Heuristic_First_Component} and
\eqref{eq:sol_h_lowerBound} to obtain
\begin{equation}
\abs{\U{t}f_R(x)} \gtrsim  \left| \phi_1( R^{1/2} (x_1 + k R^{k-1} t) )  \right| \left(\frac{R}{DQ^{1/2}}\right)^{n-1},
\end{equation}
for every $x = (x_1,x')\in [-1,1]^n$ and $t< 1/R^{k-1}$ like in \eqref{eq:Points_For_Solution}. 
That means that if we also ask for  $| x_1 + k R^{k-1} t| < R^{-1/2}$, 
we get 
\begin{equation}\label{eq:Heursitic_Evolution}
\abs{\U{t}f_R(x)} \gtrsim \Big(\frac{R}{DQ^{1/2}}\Big)^{n-1}.
\end{equation}
We gather all such requirements for $x$ in the following definition.
\begin{defin}[Admissible Slabs]
Let $k \in \mathbb N$, $R>1$, $D>1$ and $Q >1$, and
let $G(q) \subset \F^n_q$ be the set given by Lemma~\ref{thm:G_q}.
The slab
\begin{equation}\label{eq:Slabs}
E_{p,q,R} = 
	B_1 \left(  k\, \frac{R^{k-1}}{D^k} \frac{p_1}{q} , \frac{1}{R^{1/2}}  \right)  \, \times \, B_{n-1} \left( \frac{1}{D}\, \frac{p'}{q} , \frac{1}{R}  \right)
\end{equation}
is admissible whenever $Q/2 \leq q \leq Q$ and
$p = (p_1,p') = (p_1, \ldots, p_n)\in G(q)$.
Here, when we write $p\in G(q)$ we mean $p \pmod{q} \in G(q)$.
\end{defin}

The definition of these slabs allows us
to synthesize the discussion so far in the next proposition,
which combines \eqref{eq:Norms_Counterexample} 
and \eqref{eq:Heursitic_Evolution}.
The conclusion is similar to that in \cite[Proposition 5.1]{Pierce2021}.
\begin{prop}\label{thm:Proposition_Pierce}
Let $k \in \mathbb N$, $R>1$, $D>1$ and $Q >1$, and
let $W\in\Q[X_2,\cdots,X_n]$ of degree $k$ have
homogeneous part $W_k$ that
satisfies the Weil bound. 
Let also $f_R$ be the initial datum \eqref{eq:Counterexample}. 
If $t = p_1/(D^k q)$, $q\simeq Q$
and $E_{p,q,R}$ is an admissible slab, then  
\begin{equation}\label{eq:Solution_Is_Large}
\frac{\left| T_tf_R(x) \right|}{\norm{f_R}_2}  
	\gtrsim R^{1/4} \, \left(  \frac{R}{DQ} \right)^{\frac{n-1}{2}}
	\qquad \textrm{for all } x \in E_{p,q,R} \cap ([-1,0]\times [-1,1]^{n-1}).
\end{equation}
\end{prop}

The rough idea of why this indicates 
divergence of $T_tf_R$ is the following.
If we denote by $F$ the union of all admissible slabs and 
set the measure $\mu(A) = \abs{A\cap F}/\abs{F}$ supported in $F$, then
we get
\begin{equation}
\big\lVert \sup_{0<t<1} \big|  T_t f_R \big| \big\rVert_{L^1(\mu)} 
	\gtrsim R^{1/4} \, \left(  \frac{R}{DQ} \right)^{\frac{n-1}{2}}\lVert f_R  \rVert_2, \qquad \forall R \gg 1.
\end{equation}
Eventually, we will choose $D$ and $Q$ to be powers of $R$, 
so we may define the exponent $s(k, \alpha)$ by 
\begin{equation}\label{eq:Sobolev_Exponent_Origin}
R^{1/4} \, \left(  \frac{R}{DQ} \right)^{\frac{n-1}{2}} = R^{s(k, \alpha)}.
\end{equation}
This way,
the inequality above morally contradicts 
the maximal estimate for data with regularity $s < s(\alpha)$, 
which is the standard way to prove convergence. 

However, we will actually prove the divergence property directly.
We will sum the $f_R$ above to build a new datum, 
and we will also build a set $F$ by taking a limsup of the 
slabs $E_{p,q,R}$ of different $q$ and $R$.
This will be a fractal set 
whose Hausdorff dimension $\alpha$ will depend 
on the aforementioned powers $D$ and $Q$, 
and where the evolution of the datum,
of Sobolev regularity $s < s(k, \alpha)$,
will diverge. 

As marked, $s(k, \alpha)$ will depend
both on the homogeneity degree $k$ and 
on the Hausdorff dimension chosen $\alpha$.  
However, we drop the dependence on $k$ 
from the notation for simplicity, so we will simply write $s(\alpha)$. 
We refer to this regularity as the Sobolev exponent.

\subsection{The counterexample}\label{sec:Counterexample}

We just saw how the datum $f_R$ heuristically contradicts the usual maximal estimate. 
However, it does not give a counterexample for the convergence property.
To do that, we combine several $f_R$ at different scales $R$.  
We propose
\begin{equation}\label{eq:Divergent_Counterexample}
f(x) = \sum_{m \geq m_0} \frac{m}{R_m^{s(\alpha)}} \, \frac{f_{R_m}}{\lVert f_{R_m} \rVert_2}
\end{equation}
for a large enough $m_0 \in \mathbb N$, where $R_m = 2^m$ for every $m \geq m_0$. 
The parameters $D=D(R_m)$ and $Q=Q(R_m)$ corresponding to each component 
$f_{R_m}$ will be powers of $R_m$, and we will denote them simply by $D_m$ and $Q_m$ respectively.

We first remark that $f \in H^s(\R^n)$ for every $s < s(\alpha)$, as 
the triangle inequality shows that
\begin{equation}
\lVert f \rVert_{s} \leq \sum_{m \geq m_0} \frac{m}{  R_m^{s(\alpha) - s} } < \infty.
\end{equation}

In the setting of Proposition~\ref{thm:Proposition_Pierce}, given that the estimate \eqref{eq:Solution_Is_Large}
depends on $Q$ rather than on the particular choice of $q \simeq Q$, we define  the sets
\begin{equation}\label{eq:Divergence_Set_Levels}
F_m = \bigcup_{Q_m/2 \leq q \leq Q_m} \,   \bigcup_{p \in G(q)} E_{p,q,R_m}, \qquad \qquad m \geq m_0.
\end{equation}
This means that for every $m \geq m_0$, the component $f_{R_m}$ satisfies \eqref{eq:Solution_Is_Large} in the set $F_m$.

We now prove the main result of this section, 
which measures the size of the solution $T_tf$ 
in the sets $F_m$.
\begin{prop}\label{thm:Proposition_Divergent_Counterexample}
Let $f$ be defined in \eqref{eq:Divergent_Counterexample} and $M \geq m_0$.
If $x \in F_M \cap \left( [-1,-1/10] \times [-1,1]^{n-1} \right)$, then
there exists a time $t(x) \simeq 1/R_M^{k-1}$ such that 
$\big| T_{t(x)} f(x) \big| \gtrsim M$.
\end{prop}
From this 
we immediately get the divergence property we are looking for.
\begin{cor}\label{thm:Corollary_Divergent_Counterexample}
Let $F = \limsup_{m \to \infty} F_m$. 
Then, for $f$ defined in \eqref{eq:Divergent_Counterexample},
\begin{equation}
 \limsup_{t\to 0} \big| T_t f(x) \big| = \infty, 
 	\qquad \qquad \forall x \in F \cap \left( [-1,-1/10] \times [-1,1]^{n-1} \right).
\end{equation}
\end{cor}
\begin{proof}[Proof of Corollary~\ref{thm:Corollary_Divergent_Counterexample}]
If $x \in F$, then there are infinitely many $m \in \mathbb N$ such that $x \in F_m$. 
That is, there exists an increasing sequence of natural numbers $(a_n)_{n \in \mathbb N}$ such that 
$x \in F_{a_n}$ for all $n \in \mathbb N$. 
According to Proposition~\ref{thm:Proposition_Divergent_Counterexample}, 
for every $n \in \mathbb N$ there exist a time $t_n(x)$ such that 
$t_n(x) \simeq R_{a_n}^{-(k-1)}$ and $| T_{t_n(x)}f(x)| \gtrsim a_n$. 
Thus, $\lim_{n\to\infty} | T_{t_n(x)}f(x) | = \infty$, which 
implies the result because $\lim_{n\to\infty} t_n(x) = 0$. 
\end{proof}

Let us prove Proposition~\ref{thm:Proposition_Divergent_Counterexample}. 
In short, it holds because the main contribution 
to $T_tf$ in the set $F_M$ comes from the component $f_{R_M}$, while 
the effect of the rest of components is negligible.
\begin{proof}[Proof of Proposition~\ref{thm:Proposition_Divergent_Counterexample}]
Fix $M \in \mathbb N$ such that $M \geq m_0$ and let $x \in F_M$. 
For every $m \geq m_0$, we want to compute the contribution of each of the components $f_{R_m}(x)$. 
Recall that $x \in F_M$ implies that there are $q \simeq Q_M$ and $p \in G(q)$ such that $x \in E_{p,q,R_M}$, so 
Proposition~\ref{thm:Proposition_Pierce} suggest the choice
\begin{equation}\label{eq:Time_For_x}
t = t(x) = p_1/(D_M^kq).
\end{equation}
We separate cases:

$\bullet$ 
Case $m = M$. 
By the choice of $t$ in \eqref{eq:Time_For_x}, 
and recalling the definition of $s(\alpha)$ in \eqref{eq:Sobolev_Exponent_Origin},
Proposition~\ref{thm:Proposition_Pierce} implies
\begin{equation}\label{eq:Main_Contribution}
M \,  \frac{ \left| T_t f_{R_M}(x) \right|}{ R_M^{s(\alpha)}\,  \lVert f_{R_M} \rVert_2  } \gtrsim  M.
\end{equation}

$\bullet$
Case $m \neq M$.
The objective is to see that the contribution of $T_tf_{R_m}(x)$ for $m \neq M$ is much smaller, so that
\eqref{eq:Main_Contribution} dominates.  
Thus, we want to bound $\left| T_t f_{R_m}(x) \right|$ from above. 
From \eqref{eq:Evolution_Last_Components} we can directly bound
\begin{equation}
\left| T_t h_{R_m} (x') \right| 
	\leq \sum_{l\in\Z^{n-1}}\psi(D_ml/R_m) \int \big| \widehat{\phi}_2(\xi') \big| \, d\xi' 
	\leq C\left(\frac{R_m}{D_m}\right)^{n-1} \, \lVert \widehat{\phi}_2 \rVert_1, 
\end{equation}
so we write
\begin{equation}\label{eq:Components_Decomposed}
\big| T_t f_{R_m}(x) \big| \lesssim \big| T_t g_{R_m}(x_1) \big| \, \left( \frac{R_m}{D_m} \right)^{n-1}.
\end{equation}
Let us analyze the first component. According to \eqref{eq:Evolution_First_Component}, we can write
\begin{equation}\label{eq:Oscillatory_Integral}
\big| T_t g_{R_m}(x_1) \big| 
	= \Big| \int \widehat{\phi}_1(\eta)\, e^{i\, \lambda \, \theta(\eta)}\, d\eta \Big|, \qquad \text{ where } \qquad \left\{ \begin{array}{l}
\displaystyle \lambda = R_m^{1/2}\, \big|  x_1 + kR_m^{k-1} t \big|, \\
 \\
 \displaystyle \theta(\eta) = \pm \eta + \frac{t}{\lambda}\, \sum_{\ell=2}^k \binom{k}{\ell}\, \eta^\ell\,  R_m^{ k - \ell/2},
\end{array} \right.
\end{equation}
and the sign on $\pm \eta$ depends on the sign of $x_1 + tR_m^{k-1} t$. 
We want to see that
\begin{equation}\label{eq:Derivative_Of_Phase}
\theta'(\eta) = \pm 1 + \frac{t}{\lambda} \, \sum_{\ell = 2}^k \binom{k}{\ell} \, \ell \, \eta^{\ell - 1} \, R_m^{k - \ell/2} \neq 0
\end{equation}
to use the principle of non-stationary phase.
Since $ x \in E_{p,q,R_M}$ and $x_1\in [-1,-1/10]$, then
the condition $|x_1 + kR_M^{k-1} t| < R_M^{-1/2}$ implies, for $M\ge m_0\gg 1$, that
\begin{equation}
\frac{1}{R_M^{k-1}}   
	\lesssim \frac{1}{kR_M^{k-1}}\bigg(-x_1- \frac{1}{R_M^{1/2}}\bigg)
	\leq t \leq \frac{1}{kR_M^{k-1}}\bigg(-x_1+ \frac{1}{R_M^{1/2}}\bigg)
	\lesssim \frac{1}{R_M^{k-1}}.
\end{equation}
To see that the second term of $\theta'(\eta)$ in \eqref{eq:Derivative_Of_Phase} is small, we first bound it by
\begin{equation}\label{eq:Error_In_Phase}
 \left|  \frac{t}{\lambda} \, \sum_{\ell = 2}^k \binom{k}{\ell} \, \ell \, \eta^{\ell - 1} \, R_m^{k - \ell/2} \right| \leq \frac{tR_m^{k-1}}{\lambda}\, \sum_{\ell=2}^k \ell \, \binom{k}{\ell} \simeq_k \frac{tR_m^{k-1}}{\lambda}
\end{equation}
We need to consider two more cases:

$\bullet$ Case $m < M$. In this case $R_m/R_M < 1$, so since $t \simeq 1/R_M^{k-1} $, we can write 
\begin{equation}\label{eq:Small_m}
\left| x_1 + kR_m^{k-1} t\right| 
	= \BigO(R^{-1/2}_M) + k(R_M^{k-1}-R_m^{k-1})t
	\simeq 1,
\end{equation}
so 
\begin{equation}\label{eq:Small_m_Lambda}
\lambda = R_m^{1/2}\, \big|  x_1 + kR_m^{k-1}t \big| \simeq R_m^{1/2}, \qquad \qquad \text{ when } m < M,
\end{equation}
and
\begin{equation}
\frac{tR_m^{k-1}}{\lambda} 
	\simeq \frac{(R_m/R_M)^{k-1}}{R_m^{1/2}\, \left|  x_1 + kR_m^{k-1} t \right|} \lesssim \frac{1}{R_m^{1/2}} < 1/2.
\end{equation}

$\bullet$ Case $m > M$. Now $R_m/R_M > 1$, so since $t \simeq 1/R_M^{k-1} $, we can write
\begin{equation}\label{eq:Large_m}
\left| x_1 + kR_m^{k-1} t\right| 
	= \BigO(R^{-1/2}_M) + k(R_m^{k-1}-R_M^{k-1})t
	\simeq (R_m/R_M)^{k-1},
\end{equation}
so 
\begin{equation}\label{eq:Large_m_Lambda}
\lambda = R_m^{1/2}\, \big|  x_1 + kR_m^{k-1}t \big| 
	\simeq R_m^{1/2}\, \left( \frac{R_m}{R_M} \right)^{k-1} > R_m^{1/2}, \qquad \text{ when } m > M,
\end{equation}
and
\begin{equation}
\frac{t R_m^{k-1}}{\lambda} \simeq \frac{1}{R_m^{1/2}} < 1/2.
\end{equation}

In both cases we get $ tR_m^{k-1}/\lambda < 1/2$, so  from \eqref{eq:Derivative_Of_Phase} and \eqref{eq:Error_In_Phase} we deduce that $\left| \theta'(\eta)\right| > 1/2$.
This allows us to integrate \eqref{eq:Oscillatory_Integral} by parts
as many times as we need to obtain the bound
\begin{equation}
\left| T_t g_{R_m}(x_1) \right| \lesssim_N \frac{1}{\lambda^N}, 
 \qquad \textrm{for all } N \in \mathbb N.
\end{equation}
In \eqref{eq:Small_m_Lambda} and \eqref{eq:Large_m_Lambda} we got $\lambda \gtrsim R_m^{1/2}$ for every $m \neq M$, so $ \left| T_t g_{R_m}(x_1) \right| \lesssim_N  R_m^{-N/2} $ for every $m \neq M$. Together with \eqref{eq:Sobolev_Exponent_Origin} and \eqref{eq:Components_Decomposed}, this implies 
\begin{equation}\label{eq:Remainder_Contribution}
  m \,  \frac{ \left| T_t f_{R_m}(x) \right|}{ R_m^{s(\alpha)}\,  \lVert f_{R_m} \rVert_2  } \lesssim m \,  \frac{ R_m^{1/4}}{ R_m^{s(\alpha)}  } \left( \frac{R_m}{D_m} \right)^{\frac{n-1}{2}}  \, \frac{1}{R_m^{N/2}} = \frac{m \,  Q_m^{ \frac{n-1}{2}} }{R_m^{N/2}} < \frac{1}{R_m}, \qquad \qquad \text{ for all } m \neq M,
\end{equation}
where the last inequality is true if we choose $N$ as large as we need (recall that $Q_m$ is a some power of $R_m$).
Finally, joining \eqref{eq:Main_Contribution} and \eqref{eq:Remainder_Contribution}, 
for $x \in F_M$ we have found a time $t = t(x)$ in \eqref{eq:Time_For_x} such that
\begin{equation}
\begin{split}
\left|  T_{t(x)} f(x) \right| & = \left| \sum_{m \geq m_0} m \, \frac{T_t f_{R_m}(x)}{R_m^{s(\alpha)}\, \lVert f_{R_m} \rVert_2}  \right| \geq M \, \frac{\left| T_t f_{R_M}(x)\right|}{R_M^{s(\alpha)}\, \lVert f_{R_M} \rVert_2}  - \sum_{m \neq M} m \, \frac{\left|  T_t f_{R_m}(x) \right| }{R_m^{s(\alpha)}\, \lVert f_{R_m} \rVert_2} \\
& \geq M - \sum_{m \neq M} \frac{1}{R_m} \geq  M - \sum_{m =1}^\infty \frac{1}{2^m} \\
& \gtrsim M,
\end{split}
\end{equation}
the last inequality being true if $M \geq m_0$ and $m_0$ is large enough.

Incidentally, our computations show that $f$ is essentially equal to $f_{R_M}$
around the plane $\{(x,t)\mid x_1 + kR_M^{k-1} t = 0\}$ when $x\in [-1,-1/10]\times [-1,1]^{n-1}$,
so $f$ is continuous and well defined everywhere around that plane ---
recall our convention \eqref{eq:lim_def_f}.
\end{proof}

To conclude the proof of Theorem~\ref{thm:divergence_k}, 
we are missing to compute the Hausdorff dimension of the set 
of divergence,
$F = \limsup_{m \to \infty} F_m$. 
We do that using the Mass Transference Principle.


\section{The Mass Transference Principle} \label{sec:MTP}

The Mass Transference Principle was introduced in \cite{BeresnevichVelani2006} in the setting of the Duffin-Schaeffer conjecture, 
which has recently been solved,
as a method to compute the $\alpha$-Hausdorff measure of limsup sets. 
Let $x \in \mathbb R^n$ and $r >0$, and 
denote by $B = B(x,r)$ the ball with center at $x$ and radius $r$.
Let $0\le a < 1$ and define the dilated ball $B^a = B(x,r^a)$.
\begin{thm}[Mass Transference Principle, Theorem~2 in \cite{BeresnevichVelani2006} ]\label{thm:MTP}
Let $\{B_i\}_{i\in\mathbb N}$ be a sequence of balls in $\mathbb R^n$. 
Assume that the radii $r_i>0$ of $B_i$ satisfy $\lim_{i\to\infty}r_i=0$, and suppose that 
\begin{equation}\label{eq:MTP_Hypothesis}
\mathcal H^n\left( B\cap \limsup_{i\to\infty} B_i^{s/n}  \right) = \mathcal H^n(B), \qquad \text{ for every ball } B \subset \mathbb R^n.
\end{equation}
Then, 
  \begin{equation}\label{eq:MTP_Conclusion}
\mathcal H^s\left( B\cap \limsup_{i\to\infty} B_i  \right) = \mathcal H^s(B), \qquad \text{ for every ball } B \subset \mathbb R^n.
\end{equation}
\end{thm} 
Observe that the hypothesis \eqref{eq:MTP_Hypothesis} means that $\mathcal H^n(\mathbb R^n \setminus \limsup_{i\to\infty}B_i^{s/n}) = 0$, that is, 
that the set $\limsup_{i\to\infty} B_i^{s/n}$ has full measure in $\mathbb R^n$.
On the other hand,  if $s < n$, every ball $B \subset \mathbb R^n$ satisfies $\mathcal H^s(B) = \infty$. 
Thus, the conclusion of Theorem~\ref{thm:MTP} is that $\mathcal H^s(\limsup_{i\to\infty} B_i) = \infty$,  which 
in turn implies that $\dim \left( \limsup_{i\to\infty} B_i \right) \geq s$. 
This means that if we can rescale the balls so that their limsup has full Lebesgue measure,
the Mass Transference Principle gives a lower bound for the Hausdorff dimension of the limsup of the original balls. 

Theorem~\ref{thm:MTP} is a powerful tool to obtain deep results by very simple computations. 
As an example, let us explain how the celebrated Jarn\'ik-Besicovitch theorem \cite{Jarnik1929,Jarnik1931} can be easily obtained using the Mass Transference Principle and  the Dirichlet approximation theorem. 

In Diophantine approximation, a very well-known result 
(consequence either of the Dirichlet approximation theorem or of the theory of continued fractions) 
is that every real number can be approximated by infinitely many rational numbers $p/q$ with an error smaller than $q^{-2}$. 
In other words, 
\begin{equation}\label{eq:Dirichlet_Set}
\left\{ x \in \mathbb R \, : \, \left| x - \frac{p}{q} \right| < \frac{1}{q^2} \, \,\text{ for infinitely many rationals } \frac{p}{q}  \right\} = \mathbb R.
\end{equation}
One may wonder how much better than $q^{-2}$ the error can be made, so 
it is natural to study how large the sets
\begin{equation}\label{eq:JB_Set}
S_\tau = \left\{ x \in [0,1] \, : \, \left| x - \frac{p}{q} \right| < \frac{1}{q^\tau} \, \,\text{ for infinitely many rationals } \frac{p}{q}  \right\}, \qquad \tau \geq 2,
\end{equation}
are. The Jarn\'ik-Besicovitch theorem answers this question.
\begin{thm}[Jarn\' ik-Besicovitch theorem]\label{thm:Jarnik_Besicovitch}
Let $\tau \geq 2$ and $S_\tau$ be defined in \eqref{eq:JB_Set}. Then, 
\begin{equation}
\dim S_\tau = 2/\tau.
\end{equation} 
\end{thm}
As usual, the upper bound is a direct consequence of the covers
$ S_\tau \subset \bigcup_{q \geq Q} B \left( p/q, 1/q^\tau \right) $ 
for all $Q \in \mathbb N$. 
If $s > 2/\tau$, they yield 
\begin{equation}
\mathcal H^s_{1/Q^\tau} (S^\tau) \leq \sum_{ q \geq Q} \frac{q}{q^{\tau s}}
	=  \sum_{ q \geq Q} \frac{1}{q^{\tau s - 1}} \to 0 \qquad \text{ when } Q \to \infty,
\end{equation}
so $\dim S_\tau \leq 2/\tau$. 
The lower bound, independently obtained by Jarn\'ik \cite{Jarnik1929,Jarnik1931} and Besicovitch \cite{Besicovitch1934}, 
is not that easy to prove by hand.
Let us use Theorem~\ref{thm:MTP} instead. 
Recall first that the limsup of a sequence of sets $(F_j)_{j \in \mathbb N}$ 
is defined as 
$\limsup_{j \to \infty} F_j = \bigcap_{J \in \mathbb N} \bigcup_{j \geq J} F_j$, 
so it can be characterized as
\begin{equation}\label{eq:Limsup_Characterization}
x \in \limsup_{j \to \infty} F_j \qquad \Longleftrightarrow \qquad \exists \text{ infinitely many } j \in \mathbb N \text{ such that } x \in F_j.
\end{equation}
For the sets $S_\tau$, it suffices to work with rationals $p/q \in [0,1]$, which can be ordered increasingly in $q$ given that $0 \leq p \leq q$.
Thus, $S_\tau$ is the limsup of the balls $B(p/q,1/q^\tau)$.
Now, by \eqref{eq:Dirichlet_Set} we know that $S_2$ has full Lebesgue measure, and 
\begin{equation}
\mathcal H^n \left( \limsup_{\substack{q \to \infty \\ 0 \leq p \leq q}} B\left(  \frac{p}{q}, \frac{1}{q^\tau} \right)^{2/\tau}  \right) = \mathcal H^n \left( \limsup_{\substack{q \to \infty \\ 0 \leq p \leq q}} B\left(  \frac{p}{q}, \frac{1}{q^2} \right)  \right) = \mathcal H^n \left( S_2 \right) = 1.
\end{equation}
Theorem~\ref{thm:MTP} implies that 
\begin{equation}
\mathcal H^{2/\tau} \left( S_\tau \right) =  \mathcal H^{2/\tau} \left( \limsup_{\substack{q \to \infty \\ 0 \leq p \leq q}} B\left(  \frac{p}{q}, \frac{1}{q^\tau} \right)  \right) = \infty
\end{equation}
and thus $\dim S_\tau \geq 2/\tau$.

In a similar way, we will use the Mass Transference Principle 
to compute the Hausdorff dimension of 
the divergence set of $T_tf$. 
However, Theorem~\ref{thm:MTP} employs a transition from balls to balls, 
while our sets are limsups of rectangles, not balls.
But the Mass Transference Principle has been adapted to tackle 
transitions from balls to rectangles in \cite{WangWuXu2015} and 
from rectangles to rectangles recently in \cite{WangWu2021}. 
We make use of the latter.
For that, we need to adapt our setting to the framework 
in \cite[Section 3.1]{WangWu2021}.
For the sake of comparison, we use the same notation as there.

Recall that our divergence set $F$ is the limsup of the sets $F_m$ in \eqref{eq:Divergence_Set_Levels}.
To define a numbering $J$ of the slabs in $F$,
we number the slabs within each $F_m$,
so we index each slab in $F$ as $\alpha = (m,\nu)$,
where $\nu$ refers to the numbering within $F_m$.
The resonant sets $\{\mathcal{R}_\alpha\mid \alpha\in J\}$ are thus the centers of the admissible slabs \eqref{eq:Slabs},
so the scaling property is $\kappa = 0$ (see \cite[Definition~3.1]{WangWu2021}).

Also following \cite{WangWu2021}, we define the function $\beta:J\to\R_+$ as $\beta(m,i) = 2^m = R_m$
and the functions $u_m = l_m = \beta(m,i)$ so that 
$J_m = \{\alpha\in J\mid l_m\le\beta(\alpha)\le u_m\}$ index the slabs in $F_m$.
If we define $\rho(u) = u^{-1}$, then our slabs can be written as
\begin{equation}
E_{p,q,R_m} = B(\mathcal{R}_\alpha, \rho(R_m))^{\vc{b}}
	= \prod_{i = 1}^n B \left( (\mathcal{R}_{\alpha})_i , \rho(R_m)^{b_i} \right),
\end{equation}
where $\vc{b} = (1/2,1,\ldots, 1)$.

The following is \cite[Definition~3.3]{WangWu2021}, which
means that the dilated rectangles $ B(\mathcal{R}_\alpha, \rho(R_m))^{\vc{a}}$, 
with $\boldsymbol{a} = (a_1, \ldots, a_n)$ and $a_i\le b_i$, fill the space.

\begin{defin}[Uniform local ubiquity for rectangles] \label{def:UniformLocalUbiquity}
A system $\{\mathcal{R}_\alpha\mid \alpha\in J\}$ is uniformly locally ubiquitous
with respect to $(\rho, \vc{a})$ if there exists a constant $c>0$
such that for any ball $B$
\begin{equation} \label{eq:def:UniformLocalUbiquity}
\mathcal{H}^n\Big(B\cap\bigcup_{\alpha\in J_m} B(\mathcal{R}_\alpha, \rho(R_m))^{\vc{a}}\Big)\
	\ge c\,\mathcal{H}^n(B) \qquad \textrm{for all } m\ge m_0(B).
\end{equation}
\end{defin}

Uniform local ubiquity is actually stronger than full measure, and 
it can be used to get some refinements of the Mass Transference Principle,
but we do not exploit them to avoid unnecessary technical details. 

We state now an adapted version of Theorem~3.1 of \cite{WangWu2021}
considering the simplification in Proposition~3.1 there.

\begin{thm}[Mass Transference Principle from rectangles to rectangles]\label{thm:MTP_Rectangles}
In $\R^n$, let $\{\mathcal{R}_\alpha\mid \alpha\in J\}$ be a uniformly locally ubiquitous 
system with respect to $(\rho,\vc{a})$.
Let $\vc{b} = (b_1, \ldots, b_n)$ be an exponent with $b_i\ge a_i$ for $i=1,\ldots,n$, and
define the set 
\begin{equation}
W(\vc{b}) 
	= \left\{ x \in \mathbb R^n  \, : 
		\, x \in B(\mathcal{R}_\alpha, \rho(R_m))^{\vc{b}} \textrm{ for infinitely many } \alpha \in J \,  \right\}.
\end{equation}
Then, 
\begin{equation}\label{eq:MTP_Dimension}
\dim W(\vc{b}) \geq \min_{A \in \mathcal A} \left\{  \sum_{j \in K_1(A)} 1 +  \sum_{j \in K_2(A)} \left( 1 - \frac{b_j - a_j}{A} \right) + \sum_{j \in K_3(A)}  \frac{a_j}{A}   \right\},
\end{equation}
where $\mathcal A = \{ b_1, \ldots, b_n\}$ and for every $A \in \mathcal A$, 
the sets $K_1(A), K_2(A)$ and $K_3(A)$ form a partition of $\{1, \ldots, n\}$ in the following way:
\begin{equation}
K_1(A) = \{  j \, : \, a_j \geq A \}, \quad K_2(A) = \{ j \, : \, b_j \le A  \}\setminus K_1(A), 
	\quad K_3(A) = \{ 1, \ldots, n \} \setminus (K_1(A) \cup K_2(A)). 
\end{equation}
\end{thm}

\begin{rmk}
Like the original Mass Transference Principle in Theorem~\ref{thm:MTP}, one can interpret this result in terms of dilations.
Denote $ \vc{a} / \vc{b} = (a_1/ b_1, \ldots a_n/b_n) $ so that 
$B(\mathcal{R}_\alpha, \rho)^{\vc{a}}= \left( B(\mathcal{R}_\alpha, \rho)^{\vc{b}} \right)^{\vc{a}/\vc{b}}$. 
Thus, if we dilate the original balls $B(\mathcal{R}_\alpha, \rho)^{\boldsymbol{b}}$ with a dilation exponent $\vc{s} = \boldsymbol{a} / \vc{b}$ to
obtain balls $B(\mathcal{R}_\alpha, \rho)^{\vc{a}}$ whose limsup has full Lebesgue measure, then 
the dimension of $W(\vc{b})$ is given by \eqref{eq:MTP_Dimension}, 
which can be written in terms of $\vc{b}$ and $\vc{s}$ only.
\end{rmk}

\begin{rmk}
Theorem~\ref{thm:MTP} is essentially recovered by setting $\vc{b} = (1, \ldots, 1)$ and  $\vc{a} = (a, \ldots, a)$, with $a<1$, so that 
$K_2(1) = \{1, \ldots, n\}$ and $\dim W(\vc{b}) \geq na = s$.
Indeed, $B^{s/n} = B^a$.
\end{rmk}


\section{The set of divergence and its Hausdorff dimension} \label{sec:FractalSet}

We now apply the Mass Transference Principle in Theorem~\ref{thm:MTP_Rectangles}
 to compute the Hausdorff dimension of the divergence set $F = \limsup_{m \to \infty}F_m$.

\subsection{The set and the geometric parameters}
Let $R > 1$ and $D,Q \geq 1$ be our usual two parameters, 
which are powers of $R$. Based on the slabs \eqref{eq:Slabs}, we define 
\begin{equation}\label{eq:SetsOfDivergence}
\begin{split}
F_{R} & = \bigcup_{Q/2 \leq q \leq Q}  \, \bigcup_{p \in G(q)}  E_{p,q,R} \\
& = \bigcup_{Q/2 \leq q \leq Q}  \, \bigcup_{p \in G(q)}  B_1 \left(  k\, \frac{R^{k-1}}{D^k} \frac{p_1}{q} , \frac{1}{R^{1/2}}  \right)  \, \times \, B_{n-1} \left( \frac{1}{D}\, \frac{p'}{q} , \frac{1}{R}  \right).
\end{split}
\end{equation}
The sets $F_m$ in \eqref{eq:Divergence_Set_Levels} that 
give rise to the divergence set $F$ correspond to $F_m = F_{R_m}$ for the particular choice of $R = R_m = 2^m$.
As we will see, the Hausdorff dimension of $F$ depends on the choice of $D$ and $Q$.
In short, to use the Mass Transference Principle in Theorem~\ref{thm:MTP_Rectangles}, 
we need to find a dilation of the slabs in \eqref{eq:SetsOfDivergence} such that 
the limsup of the dilated slabs cover the whole $[-1,1]^n$ 
(or, more precisely, such that they satisfy \eqref{eq:def:UniformLocalUbiquity}). 

For that purpose, it is convenient to replace $D$ and $Q$ with two new parameters that 
control the separation between slabs. 
For each $q \simeq Q$, the separation in the coordinate $x_1$ of 
two consecutive slabs in \eqref{eq:SetsOfDivergence} is $kR^{k-1}/(D^kQ)$, so 
let us define the parameter $u_1$ by
\begin{equation}\label{eq:Geometric_Parameter_1}
R^{u_1} = \frac{D^k Q}{R^{k-1}}.
\end{equation}
In the same way, the separation in each of the remaining coordinates $x_j$ is $1/(DQ)$, so we define
\begin{equation}\label{eq:Geometric_Parameter_2}
R^{u_2} = DQ.
\end{equation}
Direct computation shows that \eqref{eq:Geometric_Parameter_1} and \eqref{eq:Geometric_Parameter_2} can be reversed as
\begin{equation}\label{eq:D_and_Q_in_terms_of_R}
D = R^{1 - \frac{u_2 - u_1}{k - 1}}, \qquad \qquad Q = R^{\frac{ku_2 - u_1}{k-1} - 1}.
\end{equation} 
We have thus a one to one correspondence 
between the pairs $(D,Q)$ and $(u_1,u_2)$. 
Even if $(D,Q)$ are natural to build 
the counterexample \eqref{eq:Counterexample}, 
$(u_1,u_2)$ are more convenient to work geometrically.

\subsection{Basic restrictions for the parameters}
The new parameters $(u_1,u_2)$ cannot be arbitrary.
In this subsection we determine the basic restrictions 
they have to satisfy and delimit their domain in the plane. 

First, for each $q$, the separation in $x_1$ should be 
small enough so that there is more than a single slab in $[0,1]$. 
For that, we ask  
\begin{equation}\label{eq:Restriction_1}
\frac{1}{R^{u_1}} \le 1 \qquad \Longrightarrow \qquad u_1 \geq 0.
\end{equation}
Also, the slabs will not intersect in direction $x_1$ if we ask 
\begin{equation}\label{eq:Restriction_2}
\frac{1}{R^{1/2}} \leq \frac{1}{R^{u_1}} \quad \Longrightarrow \quad u_1 \leq 1/2.
\end{equation}
Analogously, in the coordinates $x_2, \ldots, x_n$ we require 
\begin{equation}\label{eq:Restriction_3}
\frac{1}{R^{u_2}} \leq 1 \quad \Longrightarrow \quad u_2 \geq 0 \qquad \text{ and } \qquad \frac{1}{R} \leq \frac{1}{R^{u_2}} \quad \Longrightarrow \quad u_2 \leq 1.
\end{equation}
One more restriction comes from the condition $Q \geq 1$, which 
by \eqref{eq:D_and_Q_in_terms_of_R} implies 
\begin{equation}\label{eq:Restriction_4}
\boxed{ ku_2 - u_1 \geq k-1. }
\end{equation}
Inequalities \eqref{eq:Restriction_1}, \eqref{eq:Restriction_2}, \eqref{eq:Restriction_3}
and \eqref{eq:Restriction_4} determine the basic region of validity for $(u_1,u_2)$, 
which we show in Figure~\ref{fig:Trapezoid_Basic}.
We remark that, in particular, we have 
\begin{equation}\label{eq:u2_Greater_Than_1_2}
u_2 \geq \frac{u_1}{k} + \frac{k-1}{k} \geq \frac{k-1}{k} \geq \frac12
\end{equation}
as a consequence of \eqref{eq:Restriction_1} and $k \geq 2$.

\begin{figure}[t]
\centering
\includegraphics[width=0.7\linewidth]{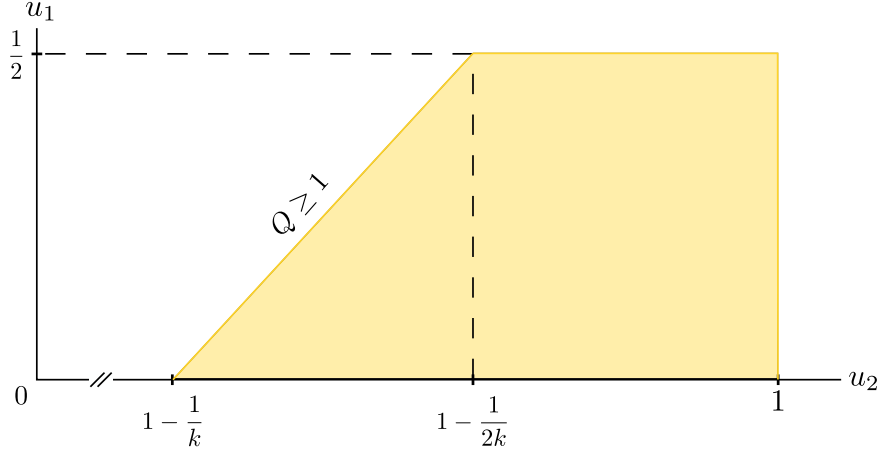}
\caption{In yellow, the basic trapezoidal region for the parameters $(u_1,u_2)$.}
\label{fig:Trapezoid_Basic}
\end{figure}

To apply the Mass Transference Principle, we need to impose 
further restrictions on $(u_1,u_2)$.

\subsection{Preparing the setting for the Mass Transference Principle}\label{sec:Preparing_For_MTP}
We now turn to the setting of 
the Mass Transference Principle in Theorem~\ref{thm:MTP_Rectangles}.
Recall from \eqref{eq:SetsOfDivergence} that the slabs are 
\begin{equation}\label{eq:Slabs_2}
E_{p,q,R} = B_1 \left(  k\, \frac{R^{k-1}}{D^k} \frac{p_1}{q} , \frac{1}{R^{1/2}}  \right)  \, \times \, B_{n-1} \left( \frac{1}{D}\, \frac{p'}{q} , \frac{1}{R}  \right).
\end{equation}
As we mentioned in Section~\ref{sec:MTP},
the radius parameter $\rho(R) = 1/R \to 0$, so we need:
\begin{itemize}
	\item  the exponent 
	\begin{equation}
	\boldsymbol{b} = (b_1, b_2,\ldots, b_2) = (1/2, 1,\ldots, 1),
	\end{equation}
	corresponding to the radii of the original slabs \eqref{eq:Slabs_2}.
	
	\item the exponent 
	\begin{equation}
	\boldsymbol{a} = (a_1, a_2,\ldots, a_2)
	\end{equation}	
	 so that the limsup of the dilated slabs
	 \begin{equation}\label{eq:Dilated_Slabs}
E_{p,q,R, \boldsymbol{a}} 
	= B_1 \left(  k\, \frac{R^{k-1}}{D^k} \frac{p_1}{q} , \frac{1}{R^{a_1}}  \right)  \, \times \, B_{n-1} \left( \frac{1}{D}\, \frac{p'}{q} , \frac{1}{R^{a_2}}  \right)
\end{equation} 
satisfies the uniform local ubiquity condition in Definition~\ref{def:UniformLocalUbiquity}.
In other words, for any ball $B$, 
we have to show that the Lebesgue measure of the set
\begin{equation}\label{eq:Dilated_Sets_F}
 F_{R,\vc{a}}\cap B := \bigcup_{Q/2 \leq q \leq Q} \bigcup_{p \in G(q)} E_{p,q,R,\vc{a}}\cap B
\end{equation}
is large.
	\end{itemize}

The set $F_{R, \vc{a}}$ has a periodic structure. 
It is made up of translations of the unit cell
\begin{equation}
\widetilde{F}_{R,\boldsymbol{a}}  
	:= \bigcup_{Q/2 \leq q \leq Q}  \, \bigcup_{p \in G(q) \cap [0,q)^n}  E_{p,q,R,\boldsymbol{a}},
\end{equation}
as shown on the left hand side of Figure~\ref{fig:Transformed_F}.
A ball $B$ contains approximately $\mathcal H^n(B)\, D^{n-1} \, D^k/(kR^{k-1})$ translated and 
disjoint copies of $\widetilde{F}_{R,\boldsymbol{a}}$
as long as $R\gg 1$ and $D$ satisfies 
\begin{equation}\label{eq:Restriction_5}
\frac{D^k}{R^{k-1}} > 1, \qquad \text{ which implies} 
	\qquad \boxed{ u_2 - u_1 < 1 - \frac{1}{k}. }
\end{equation}
Let us define the transformation $T(x) = (\frac{D^k}{k\, R^{k-1}}\, x_1,Dx')$.
Then,  $T( \widetilde{F}_{R, \boldsymbol{a}}) = \Omega_{R,\boldsymbol{a}}$,
where 
\begin{equation}\label{eq:Dilated_Sets_Omega}
\Omega_{R,\boldsymbol{a}} 
	= \bigcup_{Q/2 \leq q \leq Q} \bigcup_{p \in G(q) \cap [0,q)^n} \, B_1 \left(  \frac{p_1}{q} , \frac{D^k}{k\, R^{k-1}\, R^{a_1}}  \right)  \, \times \, B_{n-1} \left( \frac{p'}{q} , \frac{D}{R^{a_2}}  \right)
\end{equation}
is shown on the right hand side of Figure~\ref{fig:Transformed_F}.
\begin{figure}[t]
\centering
\includegraphics[width=0.9\linewidth]{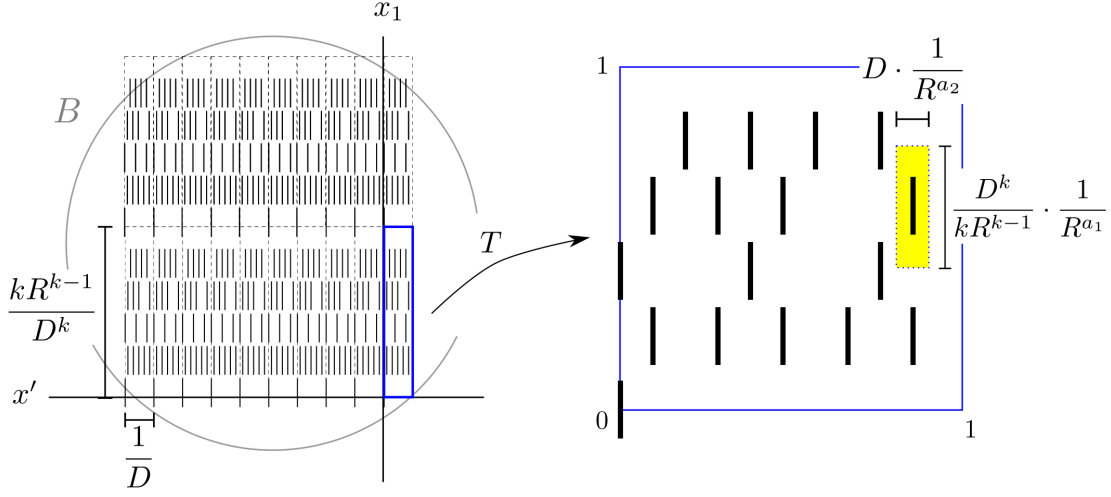}
\caption{At the left, the union $F_R$ of admissible slabs, which has
a periodic structure. 
The unit cell $\widetilde{F}_{R,\boldsymbol{a}}  $ is shown in blue.
At the right, the image 
$T( \widetilde{F}^k_{R, \boldsymbol{a}}) = \Omega_{R,\boldsymbol{a}}$
of the unit cell by the transformation $T$.
In yellow, the image $T(E_{p,q,R,\boldsymbol{a}})$ of 
a dilated admissible slab in \eqref{eq:Dilated_Slabs}.
To use the Mass Transference Principle, 
the set $ \Omega_{R,\boldsymbol{a}}$, which is the union of
the dilated admissible slabs, should cover 
a positive portion of the unit cell.}\label{fig:Transformed_F}
\end{figure}
By the scaling properties of the Lebesgue measure,
\begin{equation}
\mathcal H^n \big(  \widetilde{F}_{R, \boldsymbol{a}}  \big) =  \frac{kR^{k-1}}{D^k}\, \frac{1}{D^{n-1}}\,  \mathcal H^n (\Omega_{R,\boldsymbol{a}}),
\end{equation}
so $\mathcal H^n( F_{R, \vc{a}}\cap B)/\mathcal H^n(B) \simeq \mathcal H^n(\Omega_{R,\vc{a}})$.
Thus, to check whether the dilated slabs $E_{p,q,R,\boldsymbol{a}}$ 
form a uniformly locally ubiquitous system, 
it is enough to prove that $\mathcal H^n (\Omega_{R,\boldsymbol{a}}) \ge c > 0$.  

We now give some basic restrictions for $\boldsymbol{a}$. 
First, we need that the dilated slabs \eqref{eq:Dilated_Slabs} are contained in $[-1,1]^n$, so 
\begin{equation}\label{eq:Restriction_MTP_1}
\frac{1}{R^{a_1}}, \frac{1}{R^{a_2}} \leq 1 \qquad \Longrightarrow \qquad a_1, a_2 \geq 0.
\end{equation}
Also, we want the dilation \eqref{eq:Dilated_Slabs} to be larger 
than the original slab \eqref{eq:Slabs_2}, so
\begin{equation}\label{eq:Restriction_MTP_2}
\frac{1}{R^{a_1}} \geq \frac{1}{R^{1/2}} \quad \Longrightarrow \quad a_1 \leq 1/2, \qquad \text{ and } \qquad \frac{1}{R^{a_2}} \geq \frac{1}{R} \quad \Longrightarrow \quad a_2 \le 1.
\end{equation}

\subsection{Conditions for uniform local ubiquity}

We now look for the conditions on $\boldsymbol{a}$ so that $\mathcal H^n(\Omega_{R,\boldsymbol{a}}) \geq c>0$.
In \cite[Section 4]{Pierce2021}, they found conditions that yield the slightly weaker result $\mathcal H^n(\Omega_{R,\boldsymbol{a}}) \geq 1/\log Q$. 
In the following lines, we avoid the logarithmic loss
by making the dilated slabs $E_{p,q,R,\boldsymbol{a}}$ an $\epsilon$ larger,
and hence the dilation exponent $\boldsymbol{a}$ 
an $\epsilon$ smaller. 
This eventually results in a loss of an $\epsilon$ in the Hausdorff dimension of $F$, 
which will not affect the final result.
We need the following auxiliary lemma, which is \cite[Lemma 4.1]{Pierce2021}.

\begin{lem}[ Lemma 4.1 of \cite{Pierce2021} ]\label{thm:Lemma_Pierce_Separation}
Let $J$ be a finite set of indices and $\{I_j\}_{j \in J}$ be a collection of measurable sets in $\mathbb R^n$.
Suppose that these sets have comparable sizes, that is, 
$B_0 \leq |I_j| \leq B_1$ for every $j \in J$, and that 
they are regularly distributed in the sense that 
\begin{equation}\label{eq:Lemma_Separation_Condition}
\abs{\{ (j,j') \in J \quad : \quad I_j \cap I_{j} \neq \emptyset \}} \leq C_1 \, |J|. 
\end{equation}
Then, 
\begin{equation}
\big| \bigcup_{j \in J} I_j \big| \geq \frac{B_0}{B_1 \, C_1} \, \sum_{j \in J} |I_j|.
\end{equation}
\end{lem}

With this lemma, we can prove the following proposition, 
which is an adapted version of \cite[Proposition 4.2]{Pierce2021}
and should be compared with Minkowski's theorem on 
intersections of lattice points with a convex, symmetric body.
\begin{prop}\label{thm:Proposition_Pierce_Adapted}
Let  $R,D,Q$ be the parameters of our problem. 
Define $\mathcal P_Q = \{ q \in [Q/2,Q] \, : \, q \text{ prime }  \}$, and
let $G(q) \subset [0,q)^n$ satisfy $\abs{G(q)} \simeq q^n$ for every $q \in \mathcal P_Q$.
Suppose there exist $0 < E_1,E_2,E_3,E_4 < 1$ such that 
\begin{equation}\label{eq:Conditions_For_h}
\frac{E_1}{x^\alpha}  \leq h_1(x) \leq \frac{E_2}{x^\alpha} \qquad \text{ and } \qquad \frac{E_3}{x^\beta}  \leq h_2(x) \leq \frac{E_4}{x^\beta}
\end{equation} 
with $ \alpha, \beta \geq 1 $.
If
\begin{equation}
\Omega = \bigcup_{q \in P_Q} \bigcup_{p \in G(q)} \, B_1 \left(  \frac{p_1}{q} , h_1(Q)  \right)  \, \times \, B_{n-1} \left( \frac{p'}{q} , h_2(Q)  \right),
\end{equation} 
then
\begin{equation}
\mathcal H^n(\Omega) \gtrsim  \frac{  |\mathcal P_Q|  \, Q^n \, h_1(Q) h_2(Q)^{n-1}  }{1 +   |\mathcal P_Q| \,  Q^n\,  h_1(Q) h_2(Q)^{n-1} }.
\end{equation}
In particular, if $h_1(Q)\, h_2(Q)^{n-1} \simeq 1/Q^{n+1-\epsilon}$ for $\epsilon >0$,
then
\begin{equation}\label{eq:Omega_Positive_Measure}
\mathcal H^n (\Omega) \geq c>0.
\end{equation}
\end{prop}
\begin{rmk}
Proposition 4.2 of \cite{Pierce2021} corresponds 
to asking $ |\mathcal P_Q|  \, Q^n \, h_1(Q) h_2(Q)^{n-1} \ll 1$.
Indeed, taking into account that $|\mathcal P_Q| \simeq Q / \log Q$, 
it amounts to asking $h_1(Q) h_2(Q)^{n-1} \ll \log Q / Q^{n+1}$. 
Given that $h_1(Q) \leq 1/Q$, then it would be enough to have $h_2(Q)^{n-1} \ll \log Q / Q^n$, and in particular, 
$h_2(Q) \leq 1/Q^{n/(n-1)}$. 
Under those conditions, $\mathcal H^n(\Omega) \gtrsim  |\mathcal P_Q|  \, Q^n \, h_1(Q) h_2(Q)^{n-1} $. 
\end{rmk}
\begin{proof}
We use Lemma~\ref{thm:Lemma_Pierce_Separation}. 
For simplicity, let us call
\begin{equation}\label{eq:Omega_Short}
I_{p,q} = B_1 \left(  \frac{p_1}{q} , h_1(Q)  \right)  \, \times \, B_{n-1} \left( \frac{p'}{q} , h_2(Q)  \right) \quad \text{ so that } \quad \Omega  = \bigcup_{q \in \mathcal P_Q} \bigcup_{p \in G(q)} \, I_{p,q}.
\end{equation}
All the slabs have the same measure $\abs{I_{p,q}} = h_1(Q) \, h_2(Q)^{n-1} $, so $B_0=B_1$. 
Since the set of indices in \eqref{eq:Omega_Short} has size 
$\sum_{q \in \mathcal P_Q} \abs{G(q)} \simeq |\mathcal P_Q | \, Q^{n}$, 
to verify \eqref{eq:Lemma_Separation_Condition} we need to prove
\begin{equation}\label{eq:Verification_In_Proof}
\left| \left\{  (p,q), (\tilde{p},\tilde{q}) \in G(q) \times \mathcal P_Q \, : \, I_{p,q} \cap I_{\tilde{p},\tilde{q}} \neq \emptyset \,  \right\} \right| \leq C_1 \, |\mathcal P_Q | \, Q^n.
\end{equation}
We separate into three cases:

$\bullet$
\textbf{Case 1.} 
When the tuples are identical $(p,q) = (\tilde{p},\tilde{q})$, the condition $I_{p,q} \cap I_{\tilde{p},\tilde{q}} \neq \emptyset $ is  obviously satisfied. 
There are $\sum_{q \in \mathcal P_Q} \abs{G(q)} \simeq |\mathcal P_Q | \, Q^n$ such tuples. 

The two other cases concern non identical tuples. In this case, $I_{p,q} \cap I_{\tilde{p},\tilde{q}} \neq \emptyset$ implies 
\begin{equation}\label{eq:Separation_Proof}
\left| \frac{p_1}{q} - \frac{\tilde{p}_1}{\tilde{q}} \right| \leq h_1(Q) \qquad \text{ and } 
\qquad \left| \frac{p}{q} - \frac{\tilde{p}_j}{\tilde{q}} \right| \leq h_2(Q) , \quad j = 2, \ldots, n.
\end{equation}

$\bullet$ 
\textbf{Case 2.}
When tuples are non-identical with $q = \tilde{q}$, then from \eqref{eq:Separation_Proof}
\begin{equation}
|p_1 - \tilde{p}_1| \leq q h_1(Q) \leq E_2 q/Q^\alpha < 1
\end{equation}
because $\alpha \geq 1$.
That implies $p_1 = \tilde{p}_1$. 
In the same way,
\begin{equation}
|p_j - \tilde{p}_j| \leq q h_2(Q) \leq E_2 q/Q^{\beta} < 1, \qquad j = 2, \ldots, n,
\end{equation}
because $\beta \geq 1$. 
Thus, $p_\ell = \tilde{p}_\ell$ for $\ell = 2, \ldots, n$.
This means that $p = \tilde{p}$, so there are no tuples in this case. 
 
$\bullet$
\textbf{Case 3.}
When tuples are non-identical with $q \neq \tilde{q}$, then from \eqref{eq:Separation_Proof} we get 
\begin{equation}\label{eq:Separation_Proof_2}
\left| \tilde{q}p_1 - q \tilde{p}_1  \right| 
	\leq q \tilde{q} h_1(Q) \leq Q^2\, h_1(Q) \quad \text{ and } \quad \left| \tilde{q}p_\ell - q \tilde{p}_\ell  \right| 
	\leq q \tilde{q} h_2(Q) \leq Q^2\, h_2(Q) , \quad \ell = 2, \ldots, n.
\end{equation}
Fix $q \neq \tilde{q}$ and $\ell \in \{1,\ldots, n\}$. 

In how many ways can we represent a given integer $m$ as  
$\tilde{q}p_\ell - q \tilde{p}_\ell$ with $0 \leq p_\ell  < q$ and $0 \leq \tilde{p}_\ell  < \tilde{q}$?
For that, assume $\tilde{q}r_j  - q \tilde{r}_j$ is other representation of $m$ so that 
$\tilde{q}p_\ell  - q \tilde{p}_\ell  = \tilde{q}r_\ell  - q \tilde{r}_\ell $. 
Then, $\tilde{q} (p_\ell  - r_\ell ) = q (\tilde{p_\ell } - \tilde{r}_\ell )$. 
Since $q,\tilde{q} \in \mathcal P_Q$ are coprime, 
$q$ must divide $p_\ell  - r_\ell $ and necessarily $p_\ell  = r_\ell$.
Hence, the representation of $m$ is unique.

Since each integer $m$ with $|m| \leq Q^2 h_1(Q)$ has at most one representation $\tilde{q}p_j - q \tilde{p}_j$, 
it means that there are at most $2Q^2h_1(Q)+1$ pairs $(p_1,\tilde{p}_1)$ that 
can satisfy \eqref{eq:Separation_Proof_2}. 
Analogously, for $\ell = 2, \ldots, n$ and for an integer $m$ with $|m| \leq Q^2 h_2(Q)$ 
there is at most one pair $(p_\ell,\tilde{p}_\ell)$ that represent $m$. 
Thus, there are at most $2Q^2h_2(Q)+1$ pairs $(p_\ell,\tilde{p}_\ell)$ that can satisfy \eqref{eq:Separation_Proof_2}. 
In all, there are at most $(2Q^2h_1(Q)+1)(2Q^2h_2(Q)+1)^{n-1} \leq 4^n \,  Q^{2n} h_1(Q) h_2(Q)^{n-1}$ pairs $(p,\tilde{p})$ that can satisfy \eqref{eq:Separation_Proof_2}.

Finally, since $q, \tilde{q} \in \mathcal P_Q$, there are at most $4^n \, |\mathcal P_Q|^2  Q^{2n} h_1(Q) h_2(Q)^{n-1}$ pairs of tuples $(p,q),(\tilde{p},\tilde{q})$ in case 3.

We now join the three cases to see that 
\begin{equation}
\begin{split}
\left| \left\{  (p,q), (\tilde{p},\tilde{q}) \in G(q) \times \mathcal P_Q \, : \, I_{p,q} \cap I_{\tilde{p},\tilde{q}} \neq \emptyset \,  \right\} \right| & \lesssim  |\mathcal P_Q | \, Q^n +  |\mathcal P_Q|^2  Q^{2n} h_1(Q) h_2(Q)^{n-1} \\
& =  |\mathcal P_Q  | \, Q^n \left( 1 +  |\mathcal P_Q|  Q^n h_1(Q) h_2(Q)^{n-1}   \right). 
\end{split}
\end{equation}
Then, \eqref{eq:Verification_In_Proof} is satisfied with $C_1 = 1 +   |\mathcal P_Q|  Q^n h_1(Q) h_2(Q)^{n-1}  $ and thus Lemma~\ref{thm:Lemma_Pierce_Separation} implies 
\begin{equation}\label{eq:Lower_Bound_Omega}
\mathcal H^n(\Omega) \gtrsim \frac{\sum_{q \in \mathcal P_Q} \sum_{p \in G(q)} h_1(Q) h_2(Q)^{n-1} }{1 +   |\mathcal P_Q|  \, Q^n \, h_1(Q) h_2(Q)^{n-1} } \simeq  \frac{  |\mathcal P_Q|  \, Q^n \, h_1(Q) h_2(Q)^{n-1}  }{1 +   |\mathcal P_Q| \,  Q^n\,  h_1(Q) h_2(Q)^{n-1} }.
\end{equation}

Let now $\epsilon > 0$ and assume that $h_1(Q) h_2(Q)^{n-1} \simeq 1/Q^{n+1 - \epsilon}$. 
That means that 
\begin{equation}
 |\mathcal P_Q| \,  Q^n\,  h_1(Q) h_2(Q)^{n-1} 
 	\simeq \frac{Q^{n+1}}{\log Q} \, \frac{1}{Q^{n+1-\epsilon}} = \frac{Q^{\epsilon}}{\log Q} \gg 1,
\end{equation}
so \eqref{eq:Lower_Bound_Omega} implies $\mathcal H^n(\Omega) \ge c > 0$.
\end{proof}

Coming back to the set $\Omega_{R,\boldsymbol{a}}$, 
and having in mind that both $Q$ and $D$ are powers of $R$, 
we have 
\begin{equation}
h_1(Q) = \frac{D^k}{k R^{k-1} R^{a_1}} \qquad \text{ and } \qquad h_2(Q) = \frac{D}{R^{a_2}}.
\end{equation}
Proposition~\ref{thm:Proposition_Pierce_Adapted} shows 
that $\abs{\Omega_{R,\boldsymbol{a}}}\ge c > 0$ 
(and consequently that the dilated slabs form a uniform local ubiquity system)
whenever the following three conditions hold:
\begin{equation}
(i)\quad\frac{D^k}{k R^{k-1} R^{a_1}} \leq \frac{1}{Q}; \qquad \text{ and } \qquad (ii)\quad \frac{D}{R^{a_2}} \leq \frac{1}{Q};
\end{equation}
and, for $\epsilon >0$, 
\begin{equation}\label{eq:Choice_Of_h1_h2}
(iii)\quad \frac{D^k}{ R^{k-1} R^{a_1}}\, \left(  \frac{D}{R^{a_2}} \right)^{n-1} = \frac{1}{Q^{n+1-\epsilon}} 
\qquad \Longleftrightarrow \qquad Q^{1-\epsilon}\, R^{u_1}  \, R^{(n-1)u_2} \simeq R^{a_1 + (n-1)a_2 }.
\end{equation}
In view of the definition of $(u_1,u_2)$ in \eqref{eq:Geometric_Parameter_1}, \eqref{eq:Geometric_Parameter_2} and \eqref{eq:D_and_Q_in_terms_of_R}, comparing the exponents we get
\begin{equation}\label{eq:Restriction_MTP_3}
(i) \quad a_1 \geq u_1; \qquad \text{ and } \qquad (ii)\quad a_2 \geq u_2;
\end{equation}
and, for $\epsilon >0$,
\begin{equation}\label{eq:Restriction_MTP_4}
(iii)\quad a_1 + (n-1) a_2  = \frac{k-2+\epsilon}{k-1} \, u_1 + \frac{n(k-1) + 1 - k\epsilon}{k-1} \, u_2 - (1-\epsilon). 
\end{equation}
Together with \eqref{eq:Restriction_MTP_1} and \eqref{eq:Restriction_MTP_2}, these complete the restrictions for the exponent $\boldsymbol{a}$.
We gather all in the following lemma.
\begin{lem}\label{thm:Restrictions_For_a}
Fix the parameters $(u_1,u_2)$ and $\epsilon >0$. 
If $R \gg 1$ and 
if the dilation exponent $\boldsymbol{a} = (a_1,a_2)$ 
satisfies $u_1 \leq a_1 \leq 1/2$, $u_2 \leq a_2 \le 1$ and 
\begin{equation}\label{eq:Restriction_MTP_4_Lemma}
a_1 + (n-1) a_2  = \frac{k-2+\epsilon}{k-1} \, u_1 + \frac{n(k-1) + 1 - k\epsilon}{k-1} \, u_2 - (1-\epsilon),
\end{equation}
then the dilated admissible slabs \eqref{eq:Dilated_Slabs} form
a uniformly locally ubiquitous system as in Definition~\ref{def:UniformLocalUbiquity}.
\end{lem}

\begin{rmk}
This lemma means that, for fixed $(u_1,u_2)$, $(a_1,a_2)$ must be chosen from the intersection of the line \eqref{eq:Restriction_MTP_4_Lemma} with the rectangle $[u_1,1/2] \times [u_2,1]$.
Many times, it will be useful to rewrite \eqref{eq:Restriction_MTP_4_Lemma} as 
\begin{equation}\label{eq:Restriction_MTP_4_Lemma_Bis}
 \frac{k-2}{k-1} \, u_1 + \frac{n(k-1) + 1}{k-1} \, u_2  = a_1 + (n-1) a_2 + 1 + \epsilon \left( \frac{ku_2 - u_1}{k-1} - 1 \right).
\end{equation}
To simplify the reading in the forthcoming sections, we denote this domain by $\mathcal A$, that is, 
\begin{equation}\label{eq:Domain_a1a2}
\mathcal A = \mathcal A_{(u_1,u_2),\epsilon} 
	= \left\{ (a_1,a_2) \in [u_1,1/2] \times [u_2,1] \, : \, \eqref{eq:Restriction_MTP_4_Lemma_Bis} \text{ holds}  \right\},
\end{equation}
which we show in Figure~\ref{fig:Domain_For_Dilation}.
\end{rmk}

\begin{proof}[Proof of Lemma~\ref{thm:Restrictions_For_a}]
Combine the restrictions in \eqref{eq:Restriction_MTP_1}, \eqref{eq:Restriction_MTP_2}, \eqref{eq:Restriction_MTP_3} and \eqref{eq:Restriction_MTP_4}.
\end{proof}

From Lemma~\ref{thm:Restrictions_For_a} we get one more restriction for $(u_1,u_2)$. 
Indeed, $(u_1,u_2)$ must satisfy \eqref{eq:Restriction_MTP_4_Lemma_Bis} for every $\epsilon>0$, so 
from $a_1 \leq 1/2$, $a_2 \le 1$ and 
$ku_2 - u_1 \geq k-1$ in \eqref{eq:Restriction_4}, we find out that
\begin{equation}\label{eq:Restriction_6}
\boxed{
\frac{k-2}{k-1} \, u_1 + \frac{n(k-1) + 1}{k-1} \, u_2 \leq n + \frac12. 
}
\end{equation}
This last restriction, together with \eqref{eq:Restriction_4} and \eqref{eq:Restriction_5}, 
as well as the basic restrictions \eqref{eq:Restriction_1}, 
\eqref{eq:Restriction_2} and \eqref{eq:Restriction_3},
determines the region of validity for $(u_1,u_2)$, 
which we gather in the following lemma.
\begin{lem}\label{thm:Restrictions_For_u1u2}
The parameters $(u_1,u_2)$ must satisfy the conditions 
\begin{equation}
0 \leq u_1 \leq \frac12, \qquad \qquad \frac12 \leq u_2 \le 1,
\end{equation}
as well as 
\begin{align}
&ku_2 - u_1 \geq k-1,&
&\textrm{condition } Q \geq 1, \label{eq:R_Q1} \\
&u_2 - u_1 < 1 - \frac{1}{k},&
&\textrm{condition for shrinking unit cell,} \label{eq:R_Unit_Cell} \\
\shortintertext{and} 
&\frac{k-2}{k-1} \, u_1 + \frac{n(k-1) + 1}{k-1} \, u_2 \leq n + \frac12,&
&\textrm{condition for disjointness.} \label{eq:R_Disjointness} 
\end{align}
\end{lem}

The domain for $(u_1,u_2)$ delimited in Lemma~\ref{thm:Restrictions_For_u1u2}
plays a fundamental role.
In particular,
we want to emphasize the three conditions: $Q\ge 1$ \eqref{eq:R_Q1}, 
\textit{shrinking unit cell} \eqref{eq:R_Unit_Cell}, and 
\textit{disjointness} \eqref{eq:R_Disjointness}.
To simplify references, we denote this domain by
\begin{equation}\label{eq:Domain_u1u2}
\mathcal D = \{ (u_1,u_2) \in [0,1/2] \times [1/2,1] \text{ subject to }  \eqref{eq:R_Q1}, \eqref{eq:R_Unit_Cell} \text{ and } \eqref{eq:R_Disjointness}   \},
\end{equation}
which we show in Figure~\ref{fig:Domain_For_Geometric_Parameters}. 
Observe that the domain $\mathcal D = \mathcal D_{k,n}$ depends on both $k$ and $n$.

\begin{figure}[t]
\centering
\includegraphics[width=0.8\linewidth]{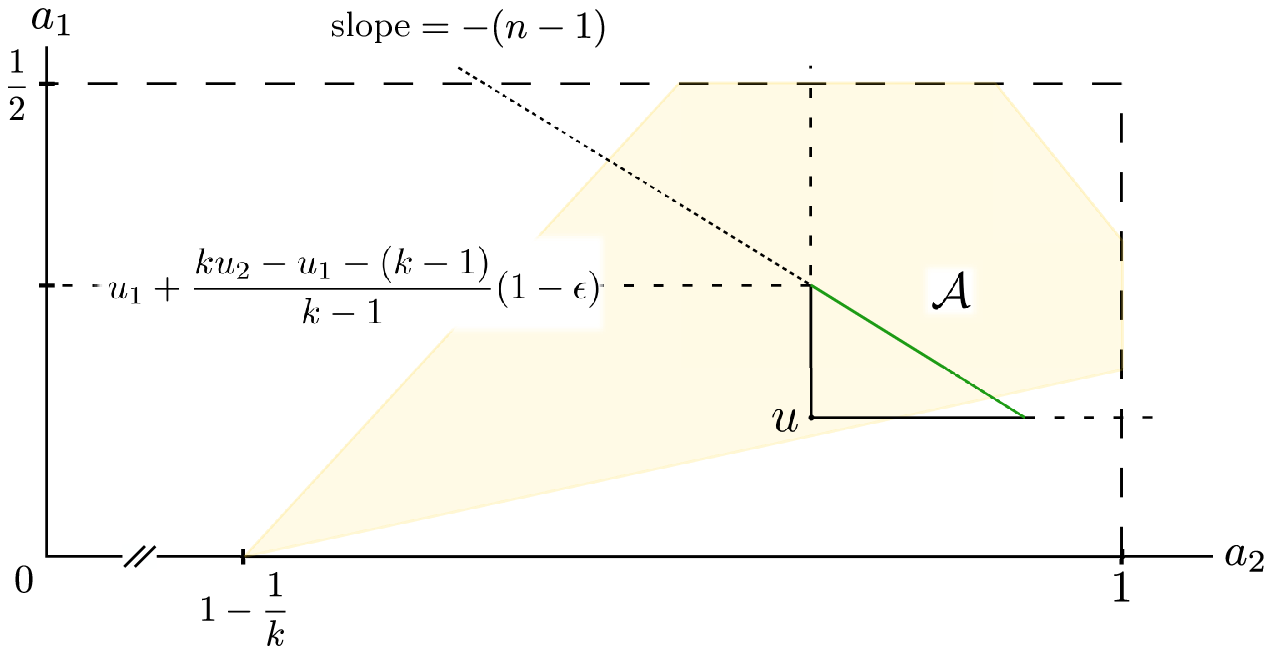}
\caption{For every fixed $(u_1,u_2) \in \mathcal D$ 
(see Figure~\ref{fig:Domain_For_Geometric_Parameters})
the domain $\mathcal A$ for the exponent $\boldsymbol{a} = (a_1,a_2)$ is the solid green line.}
\label{fig:Domain_For_Dilation}
\bigbreak
\includegraphics[width=0.8\linewidth]{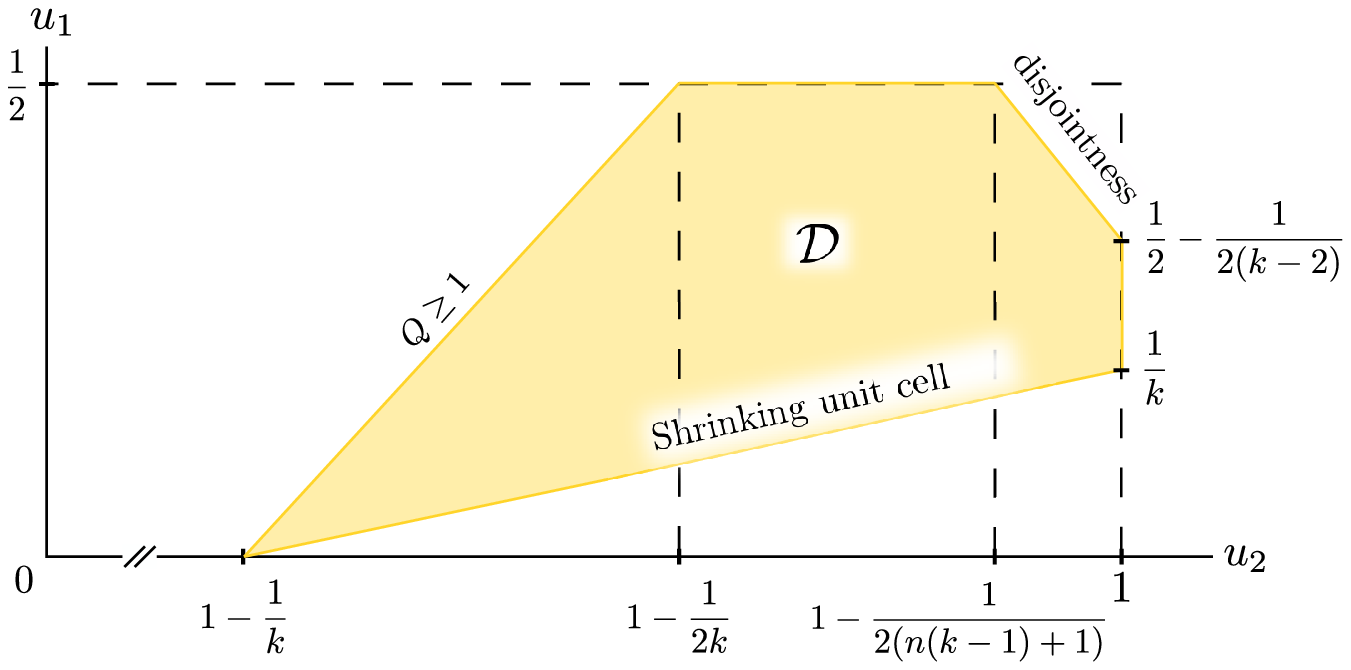}
\caption{The domain $\mathcal D$ for the parameters $(u_1,u_2)$. 
We represent the case $k \geq 5$. For $k=2,3$ and $4$ the boundaries on the right are slightly different.}\label{fig:Domain_For_Geometric_Parameters}
\end{figure}

\begin{rmk}\label{rmk:Additional_Restriction}
From Lemma~\ref{thm:Restrictions_For_u1u2} 
we can also deduce a lower bound for the condition for \textit{disjointness} \eqref{eq:R_Disjointness}.
For that, we find the line parallel to \eqref{eq:R_Disjointness} that crosses the extreme point $(u_1,u_2) = (0,1-1/k)$, which is the intersection of the conditions 
$Q\ge 1$ \eqref{eq:R_Q1} and $u_1 = 0$.
This line is 
\begin{equation}
 \frac{k-2}{k-1} \, u_1 + \frac{n(k-1) + 1}{k-1} \, u_2 = n - \frac{n-1}{k},
\end{equation}
which implies
\begin{equation}\label{eq:Restriction_6_Plus}
\boxed{
n - \frac{n-1}{k} \leq  \frac{k-2}{k-1} \, u_1 + \frac{n(k-1) + 1}{k-1} \, u_2  \leq n + \frac12. 
}
\end{equation}
\end{rmk}

For every $(u_1,u_2) \in \mathcal D$ 
we need to prove that $\mathcal A_{(u_1,u_2),\epsilon} \neq \emptyset$, 
that is, that we can always pick a dilation $(a_1,a_2) \in \mathcal A_{(u_1,u_2),\epsilon}$. 
\begin{lem}\label{thm:Existence_Of_a}
If $(u_1,u_2) \in \mathcal D$ and $\epsilon >0$, 
then $\mathcal A_{(u_1,u_2),\epsilon} \neq \emptyset$. 
\end{lem}

\begin{proof}
Fix $(u_1,u_2) \in \mathcal D$.  We interpret \eqref{eq:Restriction_MTP_4_Lemma_Bis} as a line in the variables $(a_1,a_2)$. 
First, this line crosses the point
\begin{equation}\label{eq:Critical_Point_For_a}
a_2 = u_2, \qquad \qquad a_1 = u_1 +  \frac{ku_2 - u_1 - (k-1)}{k-1}\, (1-\epsilon) \geq u_1,
\end{equation}
where the last inequality holds because $Q\ge 1$ \eqref{eq:R_Q1},
\textit{i.e.} $ku_2 - u_1 \geq k-1$.
Since the line \eqref{eq:Restriction_MTP_4_Lemma_Bis} has a negative slope equal to $-(n-1)$, 
the line crosses $[u_1,\infty] \times [u_2,\infty]$ (see Figure~\ref{fig:Domain_For_Dilation}).

On the other hand, the condition for \textit{disjointness} \eqref{eq:R_Disjointness} implies that 
the line \eqref{eq:Restriction_MTP_4_Lemma_Bis} satisfies
\begin{equation}
a_1 + (n-1)a_2 + \epsilon \Big(\frac{ku_2 - u_1}{k-1}-1\Big) \leq n - \frac12.
\end{equation}
That means that 
\begin{equation}
a_2 = 1 \qquad \Longrightarrow 
	\qquad a_1 \leq \frac12 -  \epsilon\Big(\frac{ku_2 - u_1}{k-1}-1\Big) \le \frac12, 
\end{equation}
so there is an exponent $(a_1,a_2)\in [-\infty,1/2] \times [-\infty, 1]$. 

Summing up, the line \eqref{eq:Restriction_MTP_4_Lemma_Bis} crosses both $[u_1,\infty] \times [u_2, \infty]$ and  $[-\infty,1/2] \times [-\infty, 1]$, so in particular it crosses $[u_1,1/2] \times [u_2,1]$.
Thus, $\mathcal A \neq \emptyset$.
\end{proof}
\begin{rmk}
Observe that the point in \eqref{eq:Critical_Point_For_a} depends exclusively 
on how far the point $(u_1,u_2)$ is from the line $Q\ge 1$ \eqref{eq:R_Q1}.
\end{rmk}

\subsection{A lower bound for the Hausdorff dimension}
We are ready to use the Mass Transference Principle 
to prove a lower bound for the dimension of our 
set of divergence 
\begin{equation}\label{eq:Limsup_Set_Of_Divergence}
F = \limsup_{m \to \infty} F_{m} = \bigcap_{J \in \mathbb N} \, \bigcup_{j > J} F_j,
\end{equation}
where $F_m = F_{R_m}$ is defined in \eqref{eq:SetsOfDivergence} with $R_m = 2^m$ for every $m \in \mathbb N$. 

\begin{prop}\label{thm:Prop_MTP}
Let $F$ be the set \eqref{eq:Limsup_Set_Of_Divergence} 
with parameters $(u_1,u_2) \in \mathcal D$ as in Lemma~\ref{thm:Restrictions_For_u1u2}
or Figure~\ref{fig:Domain_For_Geometric_Parameters}.  
If $\boldsymbol{a} = (a_1,a_2) \in \mathcal A$ as in Lemma~\ref{thm:Restrictions_For_a} 
or Figure~\ref{fig:Domain_For_Dilation},
then 
\begin{equation}
\dim_{\mathcal H} F \geq \min \left\{ \,  a_1 + (n-1) a_2 + 1/2 \, , \,\,\,  n - 1 + 2 a_1 \,  \right\}.
\end{equation}
Consequently, 
\begin{equation}\label{eq:Dimension_Condition_Good}
a_1 \geq (n-1) a_2 - (n- 3/2)  \qquad \Longrightarrow \qquad   \dim_{\mathcal H}  F \geq a_1 + (n-1) a_2 + 1/2,
\end{equation}
while
\begin{equation}\label{eq:Dimension_Condition_Bad}
a_1 \leq (n-1) a_2 - (n- 3/2) \qquad \Longrightarrow \qquad  \dim_{\mathcal H}  F \geq n- 1 + 2 a_1 .
\end{equation}
\end{prop}
\begin{proof}
In Lemma~\ref{thm:Restrictions_For_a} we showed that
the dilated slabs $E_{p,q,R,\boldsymbol{a}}$ form a uniform local ubiquitous system, 
so we use Theorem~\ref{thm:MTP_Rectangles} with 
$\boldsymbol{b} = (1/2,1,\ldots, 1)$ and $\boldsymbol{a} = (a_1,a_2,\ldots, a_2)$.
Thus, 
\begin{align} 
\dim_{\mathcal H}  F 
	&= \dim_{\mathcal H} \limsup_{m \to \infty} F_m \\
	&\geq \min_{A \in \{1,1/2\}} \left\{  \sum_{\ell \in K_1(A)} 1 +  \sum_{\ell \in K_2(A)} \left( 1 - \frac{b_\ell - a_\ell}{A} \right) + \sum_{\ell \in K_3(A)}  \frac{a_\ell}{A}   \right\}.
	  \label{eq:MTP_Applied}
\end{align}
We compute that now:
\begin{itemize}
	\item For $A = 1$ and $a_2<1$, we have 
	\begin{equation}
	K_1 = \{ \ell \, : \, a_\ell \ge 1 \} = \emptyset, 
		\quad K_2 = \{ \ell \, : \, b_\ell \leq 1 \}\setminus K_1 = \{1,\ldots,n\}, 
		\quad K_3 = \{ 1,\ldots, n \} \setminus ( K_1 \cup K_2) = \emptyset,
	\end{equation}
	so the number inside braces in \eqref{eq:MTP_Applied} is 
	\begin{equation}
	\alpha_1 := 1 -  \frac{b_1 - a_1}{A} + (n-1) \left( 1 -  \frac{b_2 - a_2}{A} \right) = a_1 + (n-1) a_2 + 1/2.
	\end{equation}
	When $a_2 = 1$ the result is the same.
	
	\item For $A = 1/2$ and $a_1<1/2$, 
	first observe that Lemmas~\ref{thm:Restrictions_For_a} and \ref{thm:Restrictions_For_u1u2} imply $a_2 \geq u_2 \geq 1/2$,
	so
	\begin{equation}
	K_1 = \{ \ell \, : \, a_\ell \ge 1/2 \} = \{ 2,\ldots ,n \}, 
	\quad K_2 = \{ \ell \, : \, b_\ell  \leq 1/2 \}\setminus K_1 = \{ 1 \}, 
	\quad K_3 = \{ 1,\ldots, n \} \setminus ( K_1 \cup K_2) = \emptyset,
	\end{equation}
	so the number inside braces in \eqref{eq:MTP_Applied} is 
	\begin{equation}
	\alpha_2 := (n-1) + 1 - \frac{b_1 - a_1}{A} = n - \frac{1/2 - a_1}{1/2} = n - 1 + 2a_1.
	\end{equation}
	When $a_1 = 1/2$ the result is the same.
\end{itemize} 
Thus, $\dim_{\mathcal H}  F \geq \min \{ \alpha_1, \alpha_2 \}$.
The minimum switches its value from $\alpha_1$ to $\alpha_2$ on the line
\begin{equation}\label{eq:Boundary_MTP}
a_1 = (n-1) a_2 - (n- 3/2),
\end{equation}
from which we immediately get \eqref{eq:Dimension_Condition_Good} and \eqref{eq:Dimension_Condition_Bad} in the statement.
\end{proof}

Proposition~\ref{thm:Prop_MTP} shows that different choices of $(a_1,a_2) \in \mathcal A$ give us different lower bounds for the dimension. 
Moreover, $\mathcal A$ depends on $(u_1,u_2)$.
Figure~\ref{fig:Separation_For_a_MTP} shows that certain $(u_1,u_2)$ allow $\mathcal A$ 
to intersect both sides of the boundary \eqref{eq:Boundary_MTP}.
However, for some others, Figure~\ref{fig:Separation_For_a_MTP_Bad_Case} suggests
that $\mathcal A$ may lie completely below the boundary,
so only the case \eqref{eq:Dimension_Condition_Bad} is possible.
We now determine whether one case or the other happens.

\begin{figure}[t]
\centering
\includegraphics[width=0.8\linewidth]{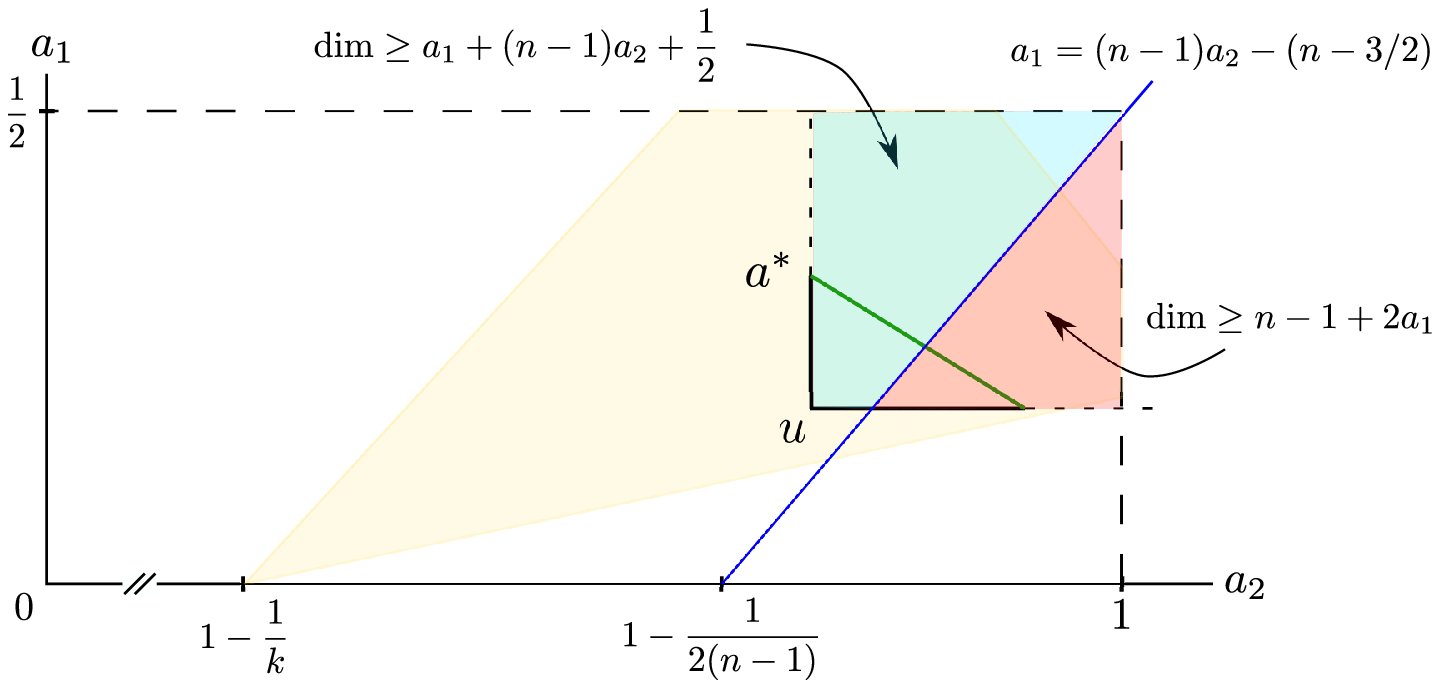}
\caption{In blue, the boundary \eqref{eq:Boundary_MTP} that delimits
 the two different regions for $\boldsymbol{a}$ according to Proposition~\ref{thm:Prop_MTP}. 
In this case, the line $\mathcal A$ has parts in both sides of it. }
\label{fig:Separation_For_a_MTP}
\bigbreak
\includegraphics[width=0.8\linewidth]{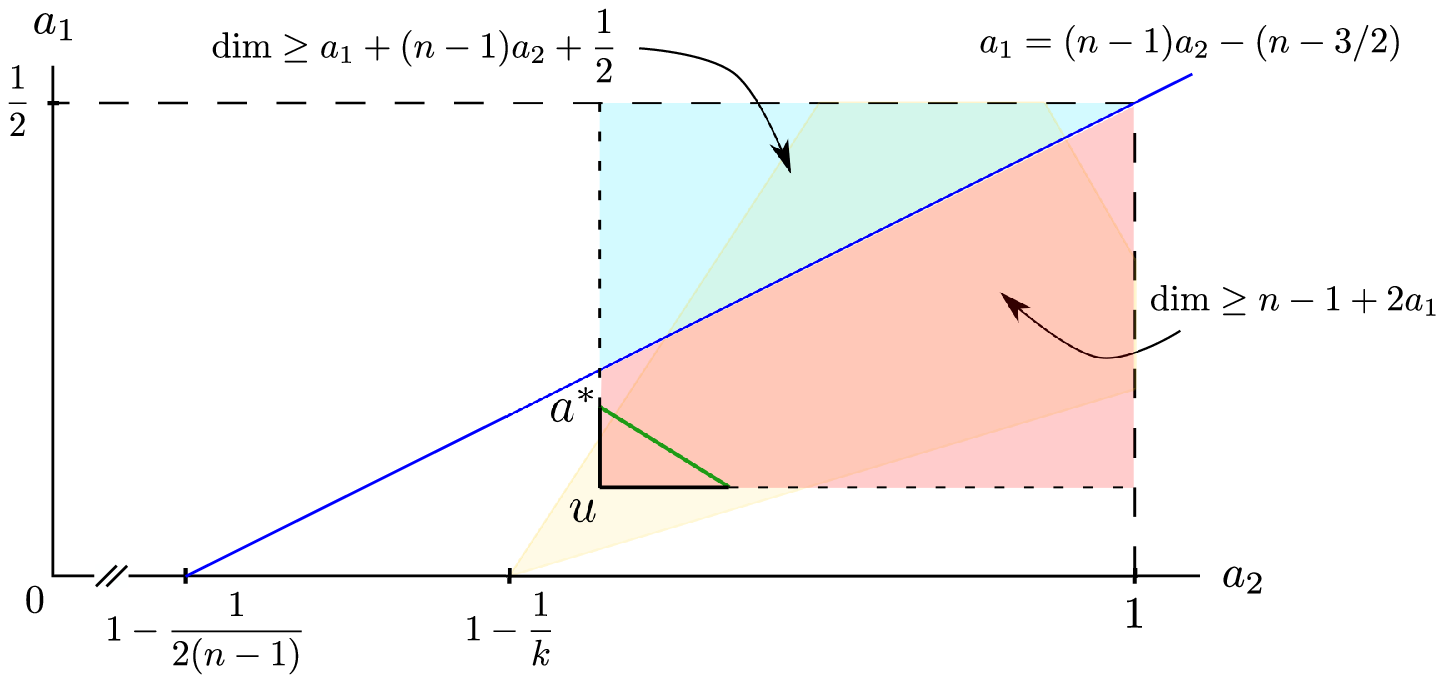}
\caption{In blue, the boundary \eqref{eq:Boundary_MTP} that delimits
 the two different regions for $\boldsymbol{a}$ according to Proposition~\ref{thm:Prop_MTP}. 
 In this case, the line $\mathcal A$ is completely below it. }
\label{fig:Separation_For_a_MTP_Bad_Case}
\end{figure}

\begin{prop}\label{thm:Good_Bad_a}
Let $(u_1,u_2) \in \mathcal D$.
Then, 
\begin{equation}\label{eq:Good_a}
\left( n - 1 - \frac{k}{k-1} \right)\,  u_2 - \frac{k-2}{k-1} \, u_1 < n - \frac52 \quad \Longrightarrow 
	\quad \exists \epsilon, (a_1,a_2) \in \mathcal A_\epsilon \, : \, a_1 \geq (n-1)a_2 - (n-3/2).
\end{equation}
On the contrary, 
\begin{equation}\label{eq:Bad_a}
\left( n - 1 - \frac{k}{k-1} \right)\,  u_2 - \frac{k-2}{k-1} \, u_1  \geq n - \frac52 \quad \Longrightarrow \quad a_1 < (n-1)a_2 - (n-3/2), \quad \forall \epsilon, (a_1,a_2) \in \mathcal A_\epsilon.
\end{equation}
\end{prop}
\begin{rmk}\label{rmk:Above_Below}
This proposition determines two regions for $(u_1,u_2) \in \mathcal D$, 
\begin{equation}
\mathcal D_{\text{above}} = \left\{ (u_1,u_2) \in \mathcal D \, \mid \, \left( n - 1 - \frac{k}{k-1} \right)\,  u_2 - \frac{k-2}{k-1} \, u_1 < n - \frac52  \right\}
\end{equation}
and 
\begin{equation}
\mathcal D_\text{below} = \left\{ (u_1,u_2) \in \mathcal D \, \mid \, \left( n - 1 - \frac{k}{k-1} \right)\,  u_2 - \frac{k-2}{k-1} \, u_1 \geq  n - \frac52  \right\},
\end{equation}
which we show in Figure~\ref{fig:D_Above_Below}. 
They are important to calculate the Sobolev exponent 
in \eqref{eq:Sobolev_Exponent_Origin}. 
\end{rmk}

\begin{figure}
\begin{center}
\includegraphics[width=0.75\linewidth]{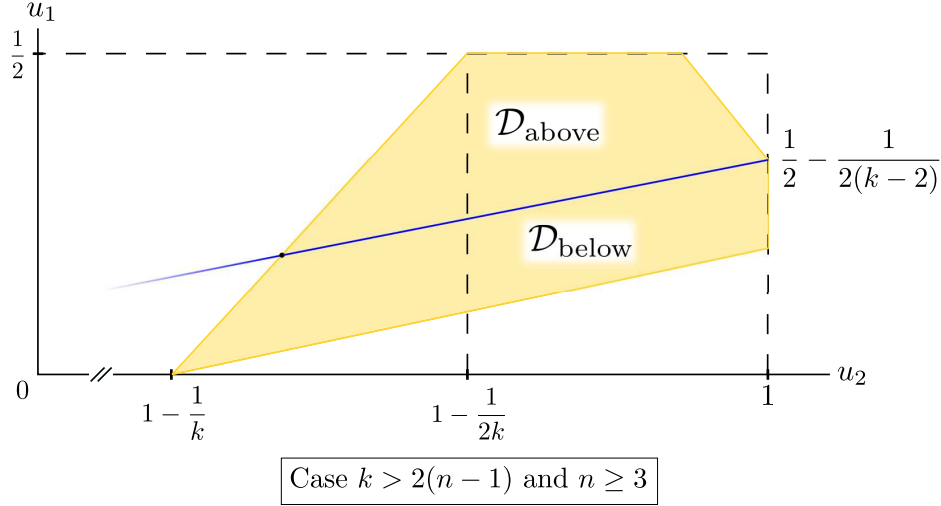}
\end{center}
\caption{The separation of the domain $\mathcal D$ according to Proposition~\ref{thm:Good_Bad_a}.}
\label{fig:D_Above_Below}
\end{figure}

\begin{proof}[Proof of Proposition~\ref{thm:Good_Bad_a}]
We have to prove that
\begin{equation} \label{eq:Border_AboveBelow}
\left( n - 1 - \frac{k}{k-1} \right)\,  u_2 - \frac{k-2}{k-1} \, u_1  \geq n - \frac52
\end{equation}
if and only if $a_1 < (n-1)a_2 - (n-3/2)$ for every $\epsilon > 0$ and $(a_1,a_2)\in\mathcal{A}_\epsilon$,
and by \eqref{eq:Restriction_MTP_4_Lemma_Bis} the last inequality is
the same as
\begin{equation} \label{eq:corner_below}
2a_1 < \frac{k-2}{k-1}u_1+\Big(n+\frac{1}{k-1}\Big)u_2-n + \frac{1}{2}-\epsilon\Big(\frac{ku_2-u_1}{k-1}-1\Big).
\end{equation}
In Figures~\ref{fig:Separation_For_a_MTP} and \ref{fig:Separation_For_a_MTP_Bad_Case}
we see that, for fixed $u\in\mathcal{D}$,
the existence of at least one dilation $\boldsymbol{a}$ in the region \eqref{eq:Dimension_Condition_Good}
depends solely on the location of the point $(a_1^*,a_2^*)$ where
the lines $\mathcal{A}_\epsilon$  \eqref{eq:Restriction_MTP_4_Lemma_Bis} and $a_2 = u_2$ intersect.
In other words, it suffices to prove that \eqref{eq:Border_AboveBelow} holds 
if and only if \eqref{eq:corner_below} holds for $(a_1^*,a_2^*)$ and for every 
$\epsilon > 0$ because $a_1\le a_1^*$ for every $(a_1,a_2) \in \mathcal A$.

The point $(a_1^*,a_2^*)$ is
\begin{equation}\label{eq:Critical_Point_For_a_Bis}
a_2^* = u_2, \qquad 
	a_1^* = \frac{k-2}{k-1}u_1+\frac{k}{k-1}u_2 -1-\epsilon\Big(\frac{ku_2-u_1}{k-1}-1\Big),
\end{equation}
so $(a_1^*,a_2^*)$ satisfies \eqref{eq:corner_below} for all $\epsilon >0$
if and only if
\begin{equation}\label{eq:Aux_Condition_1}
\left( n-1 -  \frac{k}{k-1} \right)\, u_2  -  \frac{k-2}{k-1}\, u_1  
	> n- 5/2 - \epsilon\Big(\frac{ku_2-u_1}{k-1}-1\Big), \qquad \forall \epsilon >0.
\end{equation}
By the condition  $Q\ge 1$ \eqref{eq:R_Q1}, 
this proves the proposition.
\end{proof}

\subsection{The Hausdorff dimension of the divergence set}

Combining Propositions~\ref{thm:Prop_MTP} and \ref{thm:Good_Bad_a}, we are now able to determine the Hausdorff dimension of the divergence set.
Let us begin with an upper bound that comes from the natural covering of $F$.

\begin{prop}\label{thm:Dimension_Upper_Bound}
Let $k \in \mathbb N$, $k \geq 2$. 
If $F$ is the set defined in \eqref{eq:Limsup_Set_Of_Divergence} with parameters $(u_1,u_2) \in \mathcal D$ as in Lemma~\ref{thm:Restrictions_For_u1u2},
then
\begin{equation}\label{eq:thm:Dimension_Upper_Bound}
\dim_{\mathcal H} F \leq \Big(n+\frac{1}{k-1}\Big)\,u_2+\frac{k-2}{k-1}\,u_1-\frac{1}{2}.
\end{equation}
\end{prop}
\begin{proof}
From the definition of the limsup, we have $F \subset \cup_{m\ge M}F_m$ for every $M \in \mathbb N$. 
We will first cover each $F_m$ so that
the union of all those coverings with $m\ge M$ cover $F$.  

From its definition in \eqref{eq:SetsOfDivergence}, 
$F_R$ is a union of slabs of side lengths $R^{-1/2}$ in the first coordinate and 
$R^{-1}$ in the rest of coordinates. 
The smallest scale being $R^{-1}$, 
we cover each slab with approximately $R^{1/2}$ balls of radius $R^{-1}$.
Since the number of slabs in $F_R \cap [-1,1]^n$ is approximately 
\begin{equation}
\frac{D^k\, Q}{R^{k-1}} \, (D\, Q)^{n-1}\, Q = R^{u_1+(n-1)\,u_2+ \frac{ku_2-u_1}{k-1}-1} 
	= R^{(n+\frac{1}{k-1})u_2+\frac{k-2}{k-1}u_1-1},
\end{equation} 
we need around  $ R^{(n+\frac{1}{k-1})u_2+\frac{k-2}{k-1}u_1-\frac{1}{2}}$ balls of radius $R^{-1}$ to cover $F_R$. 
Now, the Hausdorff content of a set $A$ is defined by
\begin{equation} \label{eq:def:HausdorffContent}
\mathcal H^s_{\delta}(A) = 
	\inf \left\{ \, \,  \sum_{j=1}^\infty \left( \operatorname{diam} U_j \right)^s \quad \mid \quad A \subset \bigcup_{j=1}^\infty U_j \text{ such that } \operatorname{diam} U_j \leq \delta \, \,  \right\},  \qquad \delta >0. 
\end{equation}
Choosing $M \in \mathbb N$ such that $R_M^{-1} \leq \delta$, 
the union of the coverings above for $m \geq  M$ gives
\begin{equation}
\mathcal H^s_\delta (F) \leq \sum_{m\ge M}R_m^{(n+\frac{1}{k-1})u_2+\frac{k-2}{k-1}u_1-\frac{1}{2}}  \, R_m^{-s}.
\end{equation}
Taking $\delta \to 0$ implies $M \to \infty$, so if $s > (n+\frac{1}{k-1})u_2+\frac{k-2}{k-1}u_1-\frac{1}{2}$, we get
\begin{equation}
\mathcal H^s (F) = \lim_{\delta \to 0} \mathcal H^s_\delta (F) = 0.
\end{equation}
This implies that $\dim_{\mathcal H} F \leq (n+\frac{1}{k-1})u_2+\frac{k-2}{k-1}u_1-\frac{1}{2}$.
\end{proof}

This upper bound is valid for every $(u_1,u_2) \in \mathcal D$, 
but it is optimal only in one of the situations in Proposition~\ref{thm:Good_Bad_a}.  

\begin{prop}\label{thm:Full_Good}
Let $k \in \mathbb N$, $k \geq 2$. If $F$ is the set \eqref{eq:Limsup_Set_Of_Divergence} with parameters $(u_1,u_2) \in \mathcal D$ as in Lemma~\ref{thm:Restrictions_For_u1u2},
then 
\begin{equation}
\left( n - 1 - \frac{k}{k-1} \right)\,  u_2 - \frac{k-2}{k-1} \, u_1 < n - \frac52 \quad \Longrightarrow \quad  \dim_{\mathcal H} F = \frac{k-2}{k-1} \, u_1 + \frac{n(k-1) + 1 }{k-1} \, u_2 - \frac12.
\end{equation}
\end{prop}
\begin{proof}
Proposition~\ref{thm:Good_Bad_a} ensures the existence of $(a_1,a_2) \in \mathcal A$ satisfying $a_1 \geq (n-1)a_2 - (n-3/2)$, so according to Proposition~\ref{thm:Prop_MTP}, 
\begin{equation}
\dim_{\mathcal H}  F \geq a_1 + (n-1) a_2 + 1/2. 
\end{equation}
Moreover, from \eqref{eq:Restriction_MTP_4_Lemma_Bis} in the definition of $\mathcal A$, we have 
\begin{equation}
a_1 + (n-1) a_2 + \frac12   = \frac{k-2}{k-1} \, u_1 + \frac{n(k-1) + 1 }{k-1} \, u_2 - \frac12 - \epsilon \, \frac{ku_2 - u_1 - (k-1)}{k-1}, 
\end{equation}
so we directly get
\begin{equation}
\dim_{\mathcal H} F  \geq  \frac{k-2}{k-1} \, u_1 + \frac{n(k-1) + 1 }{k-1} \, u_2 - \frac12  - \epsilon \, \frac{ku_2 - u_1 - (k-1)}{k-1}.
\end{equation}
Since this holds for any $\epsilon >0$ and 
by \eqref{eq:R_Q1} we know that $ku_2 - u_1 \geq k-1$, we get  
\begin{equation}
\dim_{\mathcal H} F  \geq  \frac{k-2}{k-1} \, u_1 + \frac{n(k-1) + 1 }{k-1} \, u_2 - \frac12 .
\end{equation}
This lower bound matches the upper bound in 
Proposition~\ref{thm:Dimension_Upper_Bound}, 
so the proof is complete.
\end{proof}

We now tackle the complementary case.

\begin{prop}\label{thm:Full_Bad}
Let $k \in \mathbb N$, $k \geq 2$. Let $F$ be the set \eqref{eq:Limsup_Set_Of_Divergence} with parameters $(u_1,u_2) \in \mathcal D$ as in Lemma~\ref{thm:Restrictions_For_u1u2}.
Then,
\begin{equation}
\left( n - 1 - \frac{k}{k-1} \right)\,  u_2 - \frac{k-2}{k-1} \, u_1 \geq n - \frac52 \quad \Longrightarrow \quad  \dim_{\mathcal H} F = n - 3 + 2 \, \frac{ku_2 + (k-2)u_1}{k-1}.
\end{equation}
\end{prop}
\begin{proof}
We prove the lower bound first.
Proposition~\ref{thm:Good_Bad_a} implies that all exponents $(a_1,a_2) \in \mathcal A$ satisfy $a_1 < (n-1)a_2 - (n-3/2)$, so 
according to Proposition~\ref{thm:Prop_MTP} we have
\begin{equation}
\dim_{\mathcal H}  F \geq n - 1 + 2a_1, \qquad \forall (a_1,a_2) \in \mathcal A.
\end{equation}
Let us choose the one with the largest $a_1$,
which is always \eqref{eq:Critical_Point_For_a_Bis}.
Thus,
\begin{equation}
\dim_{\mathcal H}  F \geq n - 3 + 2\frac{(k-2)u_1 + ku_2}{k-1} - 2\epsilon \frac{ku_2 - u_1 - (k-1)}{k-1}. 
\end{equation}
Since this holds for every $\epsilon >0$ and 
by \eqref{eq:R_Q1} we have $ku_2 - u_1 \geq k-1$, we get
\begin{equation}
\dim_{\mathcal H}  F \geq n - 3 + 2\frac{(k-2)u_1 + ku_2}{k-1}. 
\end{equation}
Regarding the upper bound, the hypothesis implies that 
\begin{equation}
n - 3 + 2 \, \frac{ku_2 + (k-2)u_1}{k-1} \leq \frac{k-2}{k-1} \, u_1 + \frac{n(k-1) + 1 }{k-1} \, u_2 - \frac12,
\end{equation}
so there is a gap with the upper bound in Proposition~\ref{thm:Dimension_Upper_Bound}. 
Actually, in this case the covering used in Proposition~\ref{thm:Dimension_Upper_Bound} 
is not optimal and can be improved as follows: 

From its definition in \eqref{eq:SetsOfDivergence}, let us rewrite $F_R$ as
\begin{equation}\label{eq:Pile_Of_Sheets}
F_{R} = \bigcup_{Q/2 \leq q \leq Q}  \, \bigcup_{p_1} \left[ \, \bigcup_{p' : p \in G(q)} B_1 \left(  k\, \frac{R^{k-1}}{D^k} \frac{p_1}{q} , \frac{1}{R^{1/2}}  \right)  \, \times \, B_{n-1} \left( \frac{1}{D}\, \frac{p'}{q} , \frac{1}{R}  \right) \, \right]. 
\end{equation}
Recall that we restrict to $[-1,1]^n$. 
Then, for every fixed $q$ and $p_1$, we can cover the set inside the brackets using $R^{(n-1)/2}$ balls of radius $R^{-1/2}$. 
Since we have $\simeq Q$ choices for $q$ and $\simeq  QD^k/R^{k-1}$ choices for $p_1$, the number of balls of radius $R^{-1/2}$ that we need to cover $F_R$ is
\begin{equation}
\frac{Q^2 D^k}{R^{k-1}}\, R^{(n-1)/2} = R^{ u_1 + \frac{ku_2 - u_1}{k-1} - 1} \, R^{\frac{n-1}{2}} = R^{\frac{n-3}{2} + \frac{k-2}{k-1}u_1 + \frac{k}{k-1}u_2},
\end{equation} 
where we used  \eqref{eq:Geometric_Parameter_1} and \eqref{eq:D_and_Q_in_terms_of_R}.
Thus, as in the proof of Proposition~\ref{thm:Dimension_Upper_Bound}, taking $M\in \mathbb N$ such that $R_M^{-1} \leq \delta$, we get
\begin{equation}
\mathcal H_\delta^s(F) 
	\leq \sum_{m \geq M} R_m^{\frac{n-3}{2} + \frac{k-2}{k-1}u_1 + \frac{k}{k-1}u_2} \, R_m^{-s/2}.
\end{equation}
Thus, 
\begin{equation}
s > n-3 + 2 \left( \frac{k-2}{k-1}u_1 + \frac{k}{k-1}u_2 \right) \quad \Longrightarrow \quad \mathcal H^s(F) = \lim_{\delta \to 0} \mathcal H_\delta^s(F) = 0,
\end{equation}
and consequently, 
\begin{equation}
\dim_{\mathcal H}  F \leq n - 3 + 2\frac{(k-2)u_1 + ku_2}{k-1}. 
\end{equation}
The proof is complete.
\end{proof}

\begin{rmk}
In \eqref{eq:Pile_Of_Sheets}, we are piling the slabs of $F_R$ in sheets of
width $R^{-1/2}$ in the direction $x_1$:
\vspace*{4mm}
\begin{center}
\includegraphics[scale=0.55]{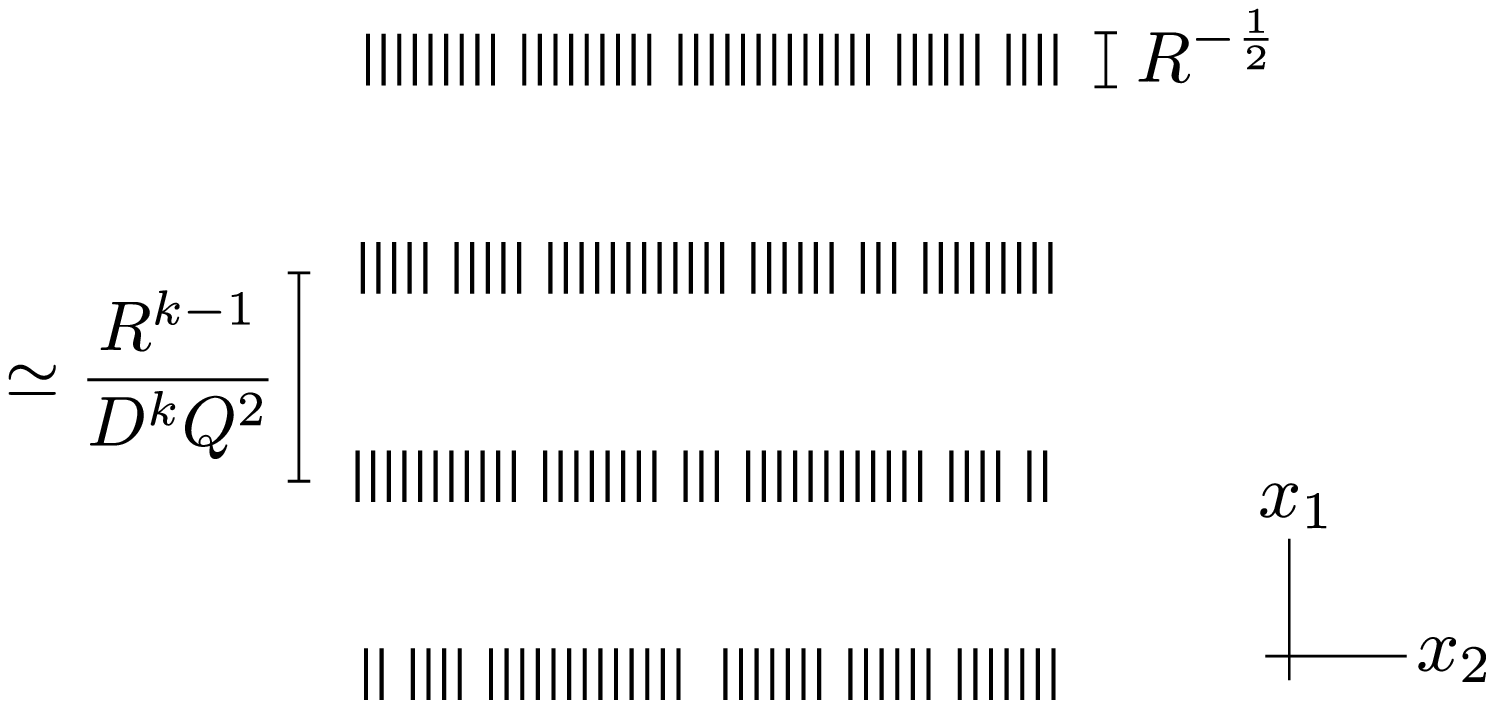}
\end{center}
\vspace*{4mm}
\noindent These sheets are centered at $x_1 = k(R^{k-1}/D^k) (p_1/q)$, so 
the distance between the centers of two sheets is
\begin{equation}
k \frac{R^{k-1}}{D^k} \left( \frac{p_1}{q} -  \frac{p'_1}{q'} \right) 
	\geq k \frac{R^{k-1}}{D^k} \frac{1}{q q'} \simeq  \frac{R^{k-1}}{D^k\, Q^2}.
\end{equation}
Then, the sheets will be pairwise disjoint if this separation is greater than their width $R^{-1/2}$. 
Using \eqref{eq:Geometric_Parameter_1} and \eqref{eq:D_and_Q_in_terms_of_R},
\begin{equation}
\frac{R^{k-1}}{D^k\, Q^2} > \frac{1}{R^{1/2}} \quad \Longleftrightarrow \quad R^{-\frac{k}{k-1}\,u_2-\frac{k-2}{k-1}\,u_1+1} > \frac{1}{R^{1/2}} \quad \Longleftrightarrow \quad \frac{k-2}{k-1}\,u_1 + \frac{k}{k-1}\,u_2 \leq \frac32.
\end{equation}
And indeed, the hypothesis of Proposition~\ref{thm:Full_Bad} together with $u_2 \leq 1$ implies precisely
\begin{equation}
\frac{k-2}{k-1}\,u_1 + \frac{k}{k-1}\,u_2 \leq (n-1)u_2 - (n-5/2) \leq \frac32.
\end{equation}
\end{rmk}

\subsection{Some simpler cases}

We have fully determined the Hausdorff dimension of $F$
in Propositions \ref{thm:Full_Good} and \ref{thm:Full_Bad}.
For some $k$ and $n$, those expressions can be simplified

\subsubsection{The dimension when $k=2$}

This corresponds to the Schr\"odinger equation, 
which was already studied in \cite{Bourgain2016,LucaPonceVanegas2021}. 
We include it here for the sake of completeness and of comparison with $k \geq 3$.

In this case, only Proposition~\ref{thm:Full_Good} applies. 
Indeed, the hypothesis there turns into
\begin{equation}
(n-3)\, u_2 < n - 5/2,
\end{equation}
which is always satisfied because:
\begin{itemize}
	\item when $n=2$, the condition is $u_2 \geq 1/2$, which is always satisfied as we saw in \eqref{eq:u2_Greater_Than_1_2}. 
	\item when $n=3$, the condition is $0 < 1/2$, which is trivially satisfied.
	\item when $n \geq 4$, we need 
	\begin{equation}
	u_2 < \frac{n-5/2}{n-3} = \frac{n- 3 + 1/2}{n-3} = 1 + \frac{1}{2(n-3)} ,
	\end{equation}
	which is also satisfied because we always have $u_2 \leq 1$. 
\end{itemize}
Thus, Proposition~\ref{thm:Full_Good} directly gives the following result.
\begin{prop}\label{thm:Full_Dimension_k2}
Let $k=2$ and $(u_1,u_2) \in \mathcal D$. 
Then, 
\begin{equation}
\dim_{\mathcal H} F = (n + 1) \, u_2 - 1/2.
\end{equation}
\end{prop}

\begin{rmk}
The reader might want to compare this result with Theorems~12 and 14 in \cite{LucaPonceVanegas2021}.
In the notation of that paper,
$u_1 = b$ and $u_2 = ((n-1)a + b + 1)/(n+1)$,
so we have not only computed the dimension over the ``extremal'' lines
$Q = 1$ and $u_1 = 1/2$ as in \cite{LucaPonceVanegas2021}, 
but on the whole region $\mathcal{D}$.
\end{rmk}

\subsubsection{The dimension when $k = 3,\,4$ and $n \geq 3$}

Like for $k=2$, only Proposition~\ref{thm:Full_Good} applies because its hypothesis always holds, that is, 
\begin{equation}\label{eq:Aux_k3}
\left( n - 1 - \frac{k}{k-1} \right)\,  u_2 - \frac{k-2}{k-1} \, u_1 < n - \frac52.
\end{equation}
Indeed, by  the \textit{shrinking unit cell} condition \eqref{eq:R_Unit_Cell} we have
$-u_1< -u_2+1-1/k$, so
\begin{equation}
\left( n - 1 - \frac{k}{k-1} \right)\,  u_2 - \frac{k-2}{k-1} \, u_1
	< (n-3)u_2 + \frac{k-2}{k} 
	\le n-\frac{5}{2} + \frac{k-4}{2k}, 
\end{equation}
where we used $u_2\le 1$ and $n \geq 3$ in the last inequality.
Hence, \eqref{eq:Aux_k3} is always satisfied
as long as $k \leq 4$,
and Proposition~\ref{thm:Full_Good} directly implies the following result.
\begin{prop}\label{thm:Full_Dimension_k3}
Let $k=3,\,4$ and $n \geq 3$. If $(u_1,u_2) \in \mathcal D$,
then 
\begin{equation}
\dim_{\mathcal H} F = \frac{k-2}{k-1}u_1+\frac{n(k-1)+1}{k-1}u_2-\frac{1}{2}.
\end{equation}
\end{prop}


\section{Sobolev exponents for divergence} \label{sec:SobolevExponents}

Once we know the dimension of $F$, 
let $0 < \alpha < n$ and fix $\dim_{\mathcal H} F = \alpha$. 
Recall that in \eqref{eq:Divergent_Counterexample} we built a counterexample that
\begin{itemize}
	\item diverges in $F$,
	\item is in $H^s(\mathbb R^n)$ for every $s < s(\alpha)$ defined in \eqref{eq:Sobolev_Exponent_Origin}, which by \eqref{eq:Geometric_Parameter_1} and \eqref{eq:Geometric_Parameter_2} is rewritten in terms of $(u_1,u_2)$ as
\begin{equation}\label{eq:Sobolev_Exponent_u2}
s(u_2) = \frac14 + \frac{n-1}{2} \, (1-u_2).
\end{equation} 
\end{itemize}
Our objective is to maximize this Sobolev exponent for every fixed $\alpha$.
For that, it suffices to minimize $u_2$ subject to the condition $\dim_{\mathcal H} F = \alpha$.
Propositions~\ref{thm:Full_Good} and \ref{thm:Full_Bad} give the relationship between 
$\alpha$ and $(u_1,u_2)$. 

We first briefly solve the case $k=2$ for the Schr\"odinger equation, and then 
we tackle the general case $k \geq 3$. 

\subsection{$\boldsymbol{k=2}$ (Schr\"odinger equation) }
In this case, by Proposition~\ref{thm:Full_Dimension_k2} we are fixing 
\begin{equation}
\alpha = (n+1) \, u_2 - 1/2, 
\end{equation}
so the Sobolev exponent $s(u_2)$ in \eqref{eq:Sobolev_Exponent_u2} is fully determined by
\begin{equation}
\begin{split}
s(\alpha) = \frac14 + \frac{n-1}{2} \left( 1 - \frac{\alpha + 1/2}{n+1} \right) & =  \frac14 + \frac{n-1}{4(n+1)} + \frac{n-1}{2(n+1)}\, (n-\alpha) \\
& = \frac{n}{2(n+1)} + \frac{n-1}{2(n+1)}\, (n-\alpha).
\end{split}
\end{equation} 
The conditions for $Q\ge 1$ \eqref{eq:R_Q1} and 
for \textit{disjointness} \eqref{eq:R_Disjointness} imply that 
\begin{equation}
1/2 \leq u_2 \leq \frac{n+1/2 }{n+1},
\end{equation}
which restricts the choice of $\alpha$ to 
\begin{equation}
n/2 \leq \alpha \leq n. 
\end{equation}
This result was obtained in \cite{LucaPonceVanegas2021}.

\subsection{$\boldsymbol{k \geq 3}$}

In general, the region $\mathcal D$ for $(u_1,u_2)$ is split into $\mathcal D_{\text{above}}$ and $\mathcal D_{\text{below}}$, defined in Remark~\ref{rmk:Above_Below} and with boundary in the line 
\begin{equation}\label{eq:Boundary_u1u2}
\left( n - 1 - \frac{k}{k-1} \right)\,  u_2 - \frac{k-2}{k-1} \, u_1 = n - \frac52. 
\end{equation}
The dimension of the divergence set is given by 
Proposition~\ref{thm:Full_Good} if $(u_1,u_2) \in \mathcal D_{\text{above}}$ and 
by Proposition~\ref{thm:Full_Bad} if $(u_1,u_2) \in \mathcal D_{\text{below}}$.
Thus, for fixed $\alpha$, we need $(u_1,u_2)$ in the polygonal line 
\begin{equation}\label{eq:Broken_Line_u1u2}
\alpha = \left\{   \begin{array}{ll}
\dfrac{k-2}{k-1} \, u_1 + \dfrac{n(k-1) + 1 }{k-1} \, u_2 - \dfrac12, 
	& \text{ for } (u_1, u_2) \in \mathcal D_{\text{above}}, \\[3mm]
n - 3 + 2 \, \dfrac{ku_2 + (k-2)u_1}{k-1}, & \text{ for } (u_1, u_2) \in \mathcal D_{\text{below}},
\end{array} \right.
\end{equation} 
with smallest $u_2$.
The broken line \eqref{eq:Broken_Line_u1u2} is shown in Figure~\ref{fig:Lines_For_Alpha}.
If we see it as $u_1$ in function of $u_2$, it has negative slope
and runs parallel to the line given by the condition for \textit{disjointness} \eqref{eq:R_Disjointness}
in  the region $\mathcal D_{ \text{above}}$.
That means that
\begin{equation}\label{eq:Smallest_u2}
\boxed{
\text{ smallest } u_2 \quad \equiv \quad \text{ intersection of \eqref{eq:Broken_Line_u1u2} with either } u_1 = 1/2 \text{ or } Q= 1 \text{ \eqref{eq:R_Q1} }.
}
\end{equation}

\begin{figure}[t]
\includegraphics[width=0.95\linewidth]{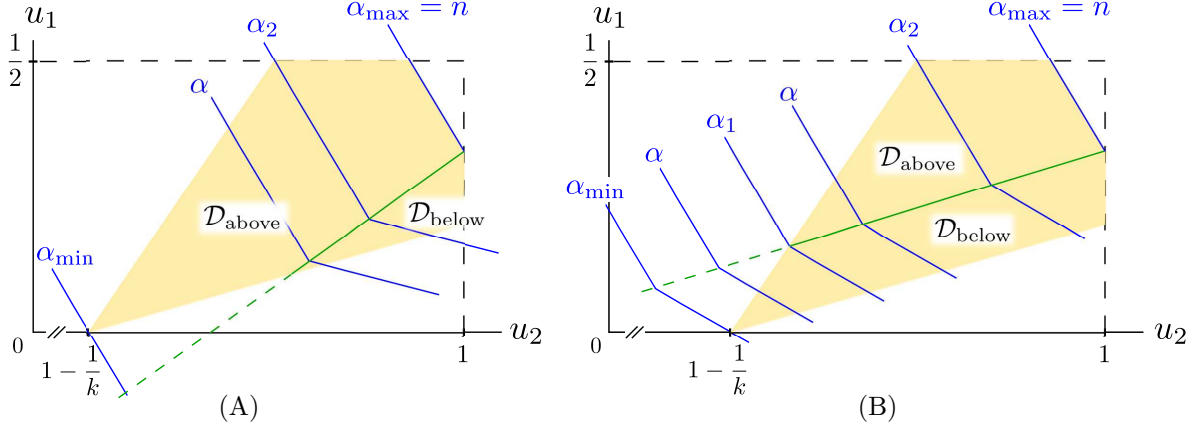}
\caption{Case $n\ge 3$. 
The green line is the boundary $\mathcal{D}_{\textrm{above}}/\mathcal{D}_{\textrm{below}}$ \eqref{eq:Boundary_u1u2}.}
\label{fig:Lines_For_Alpha}
\end{figure}

The situation much depends on whether 
the boundary line \eqref{eq:Boundary_u1u2} intersects $Q= 1$ \eqref{eq:R_Q1}
in the region $\mathcal D$ (like in Figure~\ref{fig:Lines_For_Alpha} B) 
or not (like in Figure~\ref{fig:Lines_For_Alpha} A). 
Observe that \eqref{eq:Boundary_u1u2} crosses the point
\begin{equation}\label{eq:Boundary_Point_1}
u_1 = \frac12 - \frac{1}{2(k-2)}, \qquad u_2 =  1
\end{equation}
and has slope
\begin{equation}\label{eq:Slope_Boundary}
\text{ slope of \eqref{eq:Boundary_u1u2}} = \frac{(n-1)(k-1) - k}{k-2}.
\end{equation}
Given that $k \geq 3$, this slope is positive if and only if $n \geq 3$. 
\begin{itemize}
	\item If $n=2$, the slope of \eqref{eq:Boundary_u1u2} is negative, and 
	it always intersects $Q = 1$ \eqref{eq:R_Q1} at the point
	\begin{equation}\label{eq:Intersection_n2}
		u_1 = \frac12 - \frac{1}{2(k-1)} , \qquad u_2 = \frac{k - 3/2}{k - 1} = 1 - \frac{1}{2(k-1)}.
		\end{equation}
		Since $0 \leq u_1 < 1/2$, we have $(u_1,u_2) \in \mathcal D$.

	\item If $n \geq 3$, the point in \eqref{eq:Boundary_u1u2} with $u_1 = 0$ is
		\begin{equation}
		u_1 = 0, \qquad u_2 = \frac{(n-5/2)(k-1)}{(n-1)(k-1)- k}.
        \end{equation}		 
        Given that $Q = 1$ \eqref{eq:R_Q1} crosses the point $(0,\frac{k-1}{k})$, 
        the intersection happens inside $\mathcal D$ 
        (like in Figure~\ref{fig:Lines_For_Alpha} B) if and only if 
        \begin{equation}\label{eq:The_Two_Cases}
        \frac{(n-5/2)(k-1)}{(n-1)(k-1)- k} < \frac{k-1}{k} \quad \Longleftrightarrow \quad n-1 < \frac{k}{2}.
        \end{equation}
        In this case,  the crossing point is 
        \begin{equation}\label{eq:Crossing_Point}
        u_1 = \frac12 - \frac{n-1}{2(k - (n - 1))} , \qquad u_2  = 1 - \frac{1}{2(k - (n - 1))}.
        \end{equation}
		Observe that both the condition \eqref{eq:The_Two_Cases} and 
		the crossing point \eqref{eq:Crossing_Point} also work for $n=2$.
\end{itemize}

Thus, according to \eqref{eq:The_Two_Cases}, we have two different situations to treat.

\subsubsection{When $n - 1 \geq k/2$}
In this case,  the boundary $\mathcal{D}_{\textrm{above}}/\mathcal{D}_{\textrm{below}}$ \eqref{eq:Boundary_u1u2} does not intersect the line $Q = 1$ \eqref{eq:R_Q1}, 
as shown in Figure~\ref{fig:Lines_For_Alpha}\, A.
In particular, both $Q = 1$ \eqref{eq:R_Q1} and $u_1=1/2$ are in $\mathcal{D}_{\textrm{above}}$.
By \eqref{eq:Smallest_u2}, it suffices to consider the section of the broken line \eqref{eq:Broken_Line_u1u2} in $\mathcal D_{\text{above}}$, 
\begin{equation}\label{eq:Alpha_1_Bis}
\alpha = \frac{k-2}{k-1} u_1 + \frac{n(k-1) + 1}{k-1} u_2 - \frac12.
\end{equation}
The additional restriction \eqref{eq:Restriction_6_Plus} in Remark~\ref{rmk:Additional_Restriction} delimits the range of $\alpha$ to
\begin{equation}
 n - \frac{n-1}{k} \leq \alpha + \frac12 \leq n + \frac12 \quad \Longrightarrow \quad  n - \frac{n-1}{k} - \frac12 \leq  \alpha \leq n.
\end{equation}
We have two cases:
\begin{itemize}
	\item If the line \eqref{eq:Alpha_1_Bis} intersects $u_1=1/2$, then that intersection point has 
	\begin{equation}
	u_2 = \left( \alpha + \frac{1}{2(k-1)} \right) \frac{k-1}{n(k-1)+1},
	\end{equation}
	and from \eqref{eq:Sobolev_Exponent_u2} we obtain the Sobolev exponent 
	\begin{equation}\label{eq:Sobolev_Largest_Alpha}
	s(\alpha) = \frac{nk}{4(n(k-1)+1)} + \frac{(n-1)(k-1)}{2(n(k-1)+1)} (n-\alpha).
	\end{equation}
	It is important to notice that, in this case, $\alpha$ is restricted by $u_2 \geq (k-1/2)/k$, or equivalently,
\begin{equation}
\alpha \geq  \frac{n(k-1) + 1}{k-1}\cdot\frac{k-1/2}{k} - \frac{1}{2(k-1)} = n - \frac{n-1}{2k} = \alpha_2.
\end{equation}

\item If the line \eqref{eq:Alpha_1_Bis} intersects $Q = 1$ \eqref{eq:R_Q1}, then 
the point of intersection is 
\begin{equation}
u_1 = \frac{n-1 - k(n - \alpha - 1/2)}{n+k-1}, \qquad u_2 = 1 - \frac{n - \alpha + 1/2}{n+k-1},
\end{equation}
so from \eqref{eq:Sobolev_Exponent_u2} we get the Sobolev exponent
\begin{equation}\label{eq:Sobolev_Middle_Alpha}
s(\alpha) = \frac{2(n-1)+k}{4(n+k-1)} + \frac{n-1}{2(n + k - 1)}(n-\alpha), \qquad \text{ for } n - \frac{n-1}{k} - \frac12  \leq \alpha  \leq n - \frac{n-1}{2k}.
\end{equation}

\end{itemize}

\subsubsection{When $n - 1 < k /2$}

In this case, the boundary $\mathcal{D}_{\textrm{above}}/\mathcal{D}_{\textrm{below}}$ 
\eqref{eq:Boundary_u1u2} intersects $Q = 1$ \eqref{eq:R_Q1} at the point \eqref{eq:Crossing_Point},
as shown in Figure~\ref{fig:Lines_For_Alpha}\, B. 
Thus, contrary to the previous case, there are values of $\alpha$ for which 
the broken line \eqref{eq:Broken_Line_u1u2} crosses $Q = 1$ \eqref{eq:R_Q1} in the region $\mathcal D_{\text{below}}$.
The largest of such $\alpha$, call it $\alpha_1$, corresponds to the broken line \eqref{eq:Broken_Line_u1u2} crossing the point \eqref{eq:Crossing_Point}. 
Since the broken line takes in $\mathcal D_{\text{below}}$ the form 
\begin{equation}\label{eq:Alpha_2_Bis}
\alpha = n - 3 + 2 \, \frac{ku_2 + (k-2)u_1}{k-1},
\end{equation}
we get
\begin{equation}
\alpha_1 = n - \frac{n-1}{k - (n-1)}.
\end{equation} 
On the other hand, the smallest $\alpha$, call it $\alpha_\text{min}$, corresponds to 
the broken line \eqref{eq:Broken_Line_u1u2} or \eqref{eq:Alpha_2_Bis} crossing the point
$(u_1,u_2) = (0,(k-1)/k)$, that is,
\begin{equation}
\alpha_\text{min} =  n - 1.
\end{equation}
Thus, for $\alpha_\text{min} \leq \alpha \leq \alpha_1$, 
the point of intersection between the broken line \eqref{eq:Alpha_2_Bis} and $Q = 1$ \eqref{eq:R_Q1} is 
\begin{equation}
u_1 = \frac12 - \frac{n-\alpha}{2}, \qquad u_2 = 1 - \frac{n-\alpha + 1}{2k},
\end{equation}
so from \eqref{eq:Sobolev_Exponent_u2} we obtain the Sobolev exponent
\begin{equation}
s(\alpha) = \frac14 + \frac{n-1}{4k} + \frac{n-1}{4k}(n-\alpha).
\end{equation}
Observe that this counterexample is only relevant if $s(\alpha)\ge (n-\alpha)/2$.
Thus, we need to shrink the range of $\alpha$ so that 
\begin{equation}\label{eq:Sobolev_Lowest_Alpha}
s(\alpha) = \frac14 + \frac{n-1}{4k} + \frac{n-1}{4k}(n-\alpha), \qquad \text{ for } 
	\quad n-\frac{1}{2}-\frac{3(n-1)}{2(2k-n+1)} \leq \alpha \leq n - \frac{n-1}{k - (n-1)}.
\end{equation}

For the remaining $\alpha > n - \frac{n-1}{k - (n-1)} = \alpha_1$, 
the broken line \eqref{eq:Broken_Line_u1u2} intersects either 
$Q = 1$ \eqref{eq:R_Q1} or $u_1=1/2$  in the region $\mathcal D_{\text{above}}$
in the same way as it did in the previous case $n-1 \geq k/2$. 
Therefore, we obtain the same Sobolev exponents \eqref{eq:Sobolev_Largest_Alpha} and \eqref{eq:Sobolev_Middle_Alpha}. 

This completes the proof of Theorem~\ref{thm:divergence_k}.


\section{Divergence for saddle-like symbols} \label{sec:saddle}

We now prove Theorem~\ref{thm:examples_saddle}.
For that, we need to build counterexamples for the symbols
\begin{equation}\label{eq:Saddle}
P(\xi) = \xi_1^2+\cdots + \xi_m^2 - \xi_{m+1}^2-\cdots -\xi_{n}^2
\end{equation}
with index $1\le m\le n/2$. 
Although Theorem~\ref{thm:divergence_k} can be used in this case as well,
it does not yield the largest possible sets of divergence.
Instead, we adjust the example of 
Rogers, Vargas and Vega \cite{MR2284549} to the fractal context.

\subsection{When $\boldsymbol{\alpha  \geq n - m + 1}$: Optimal result}\label{sec:Saddle_Optimal}

We first prove that if $\alpha$ is large enough, 
the non-dispersive threshold in Theorem~\ref{thm:non-dispersive_B} is optimal
for the symbol \eqref{eq:Saddle}. 
That means that the solution $\U{t}$ may behave as if there were no dispersion.

\begin{thm}\label{thm:saddle_dim_sharp}
Let $P$ be a quadratic form with index $1\le m\le n/2$.
Assume that $\alpha\ge n-m+1$ and that
\begin{equation}
s< \frac{n-\alpha+1}{2}.
\end{equation}
Then, there exists  $f \in H^s(\mathbb R^n)$ such that $T_tf$ diverges
in a set of Hausdorff dimension $\alpha$.
\end{thm}
\begin{proof}
As usual, we build functions $f_R$ at fixed scales $R \gg 1$
and then sum all of them after verifying that the solutions
$\U{t}f_R$ do not interfere with each other.

First, the change of variables $\xi_i \to \frac{1}{2}(\xi_i + \xi_{m+i})$
and $\xi_{m+i} \to \frac{1}{2}(\xi_i - \xi_{m+i})$, for $i=1, \ldots, m$,
transforms the symbol to
\begin{equation}\label{eq:Transformed_Simbol_Saddle}
P(\xi) = \xi_1\xi_{m+1} + \cdots + \xi_m\xi_{2m} - \xi_{2m+1}^2-\cdots - \xi_n^2.
\end{equation}
Let $\varphi\in\Sz(\R^n)$ with $\widehat{\varphi}$ supported in $B(0,c)$
for some small $c>0$, and consider the preliminary initial datum
\begin{equation}
\widehat{f}_R(\xi) = \sum_{\abs{l''}\le R/D}\widehat{\varphi}(\xi_1-R,\xi',\xi_{m+1}/R, \xi''-Dl'', \xi'''), \qquad l'' = (l_{m+2}, \ldots, l_{2m}) \in \mathbb R^{m-1},
\end{equation}
\begin{center}
\includegraphics[scale=0.8]{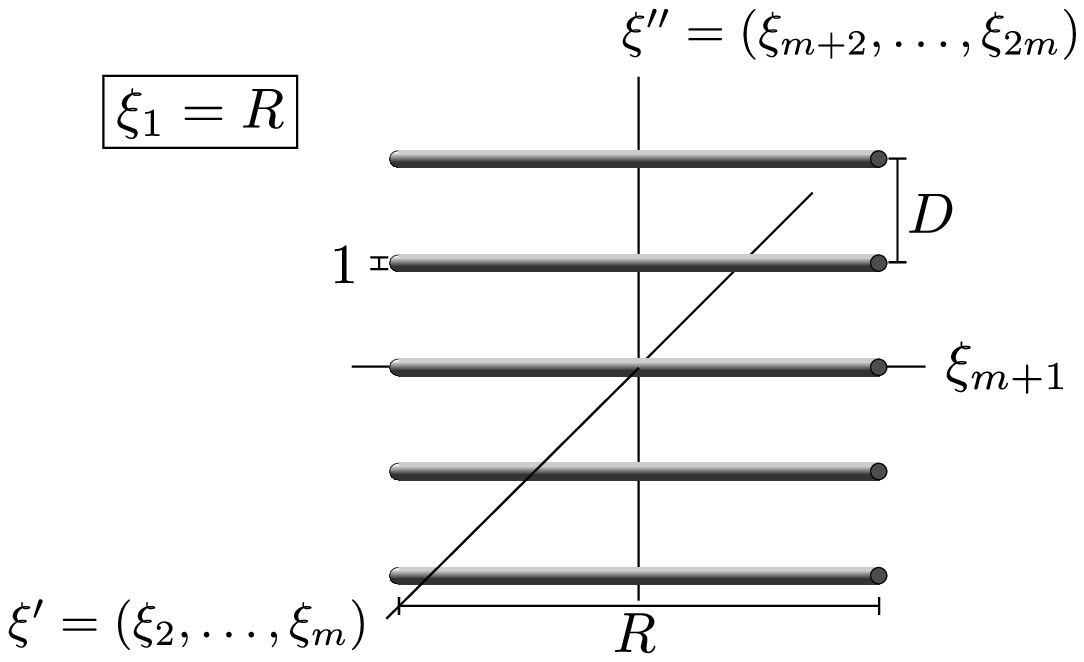}
\end{center}
\vspace*{3mm}
Here, we are denoting $\xi = (\xi_1,\xi',\xi_{m+1},\xi'',\xi''') \in \mathbb R^n$, where $\xi'=(\xi_2,\ldots, \xi_m)$, $\xi'' = (\xi_{m+2}, \ldots, \xi_{2m})$ and $\xi''' = (\xi_{2m+1},\ldots,\xi_{n})$. 
The solution can be written as 
\begin{align}
\U{t}f_R(x) &= \sum_{\abs{l''}\le R/D}\int \widehat{\varphi}(\xi_1-R, \xi', 
	\xi_{m+1}/R, \xi''-Dl'', \xi''') \, e(x\cdot\xi + tP(\xi))\,d\xi \\
	&= R \sum_{\abs{l''}\le R/D}e^{2\pi i(Rx_1 +x''\cdot Dl'')} \\
	&  \hspace*{2cm} \int \widehat{\varphi}(\xi)\,  e\Big(x_1\xi_1+(x'+tDl'')\cdot\xi'+ R(x_{m+1} + R\, t)\, \xi_{m+1} +  \\
 	&\hspace*{5cm} + x''\cdot\xi''+ x'''\cdot\xi''' + t(R\xi_1\xi_{m+1}+\xi'\cdot\xi''-\abs{\xi'''}^2)\Big)\,d\xi.
\end{align}
We evaluate it at times $0\le t\le 1/R$ and in the set $F$ given by the conditions
\begin{equation} \label{eq:saddle_trans_conditions}
\abs{(x_1,x')}\le 1;
\quad x_{m+1} + Rt = 0; \quad 
x'' = D^{-1}\Z^{m-1} + \BigO(R^{-1}); \quad
\textrm{and} \quad 
\abs{x'''}\le 1,
\end{equation}
see Figure~\ref{fig:saddle_ex_trans}. 
Take the absolute value so that
\begin{equation}
\abs{\U{t}f_R(x)} \simeq R \,  \Big| \sum_{\abs{l''}\le R/D}e^{2\pi ix''\cdot Dl''} \Big| \simeq R\, \left( \frac{R}{D} \right)^{m-1}.
\end{equation}
Since $\norm{f_R}_2\simeq R^{1/2}(R/D)^\frac{m-1}{2}$, we get
\begin{equation}\label{eq:Critical_Regularity_For_Saddle}
\frac{\abs{\U{t}f_R(x)}}{\norm{f_R}_2} \simeq R^{1/2}\, \left( \frac{R}{D} \right)^\frac{m-1}{2} = R^{\frac {1+(m-1)(1-a)}{2}},
\end{equation}
where we set $D = R^a$ with $0\le a\le 1$. 
Denote
\begin{equation}\label{eq:Critical_Regularity_For_Saddle_2}
s(a) = \frac{1 + (m-1)(1-a)}{2},
\end{equation}
which is the exponent that will determine the regularity of the datum.
 
Define the scales $R_k = 2^k$ for $k \in \mathbb N$, and 
for some large enough $k_0 \in \mathbb N$ we define the datum
\begin{equation}\label{eq:Datum_For_Saddle}
f(x) = \sum_{k \geq k_0} k \, \frac{f_{R_k}}{R_k^{s(a)}\, \lVert f_{R_k} \rVert_2 }.
\end{equation}
The triangle inequality gives $f \in H^s(\R^n)$ for all $s < s(a)$. 
Denoting by $F_k$ the corresponding set in \eqref{eq:saddle_trans_conditions}, 
the same method we used to prove Proposition~\ref{thm:Proposition_Divergent_Counterexample}
proves that for every $K \geq k_0$ and for every $x \in F_K$, 
there exists a time $t = t(x) \leq 1/R_K$ such that 
\begin{equation}
\left|  T_{t(x)} f(x) \right| 
	= \Bigg| \sum_{k \geq k_0} k \, \frac{T_{t(x)}f_{R_k}(x)}{R_k^{s(a)}\, \lVert f_{R_k} \rVert_2 } \Bigg| 
\gtrsim  K \, \frac{ |T_{t(x)}f_{R_K}(x) | }{R_K^{s(a)}\, \lVert f_{R_K} \rVert_2 } \simeq K.
\end{equation} 
Consequently, taking the limit $K \to \infty$, we get
\begin{equation}\label{eq:Divergent_Solution_Saddle}
\limsup_{t \to 0} \left|  T_{t} f(x) \right|  = \infty, \qquad \forall x \in F,
\end{equation}
where $F = \limsup_{k \to \infty} F_k$.

\begin{figure}[t]
\centering
\includegraphics[scale=0.8]{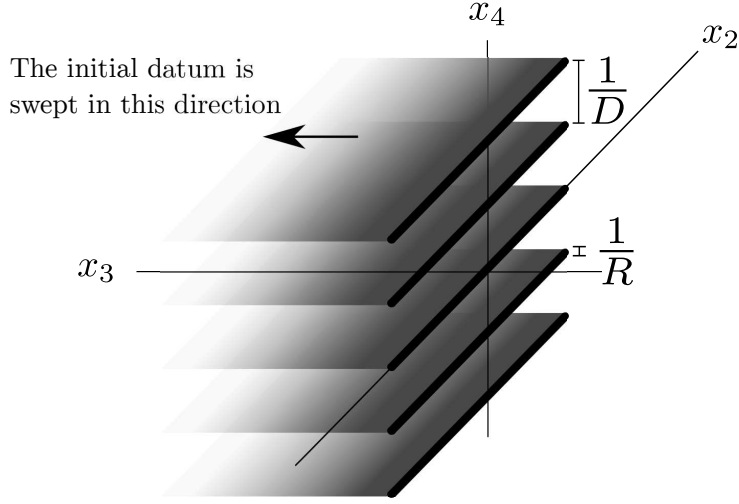}
\caption{A representation of the example in Theorem~\ref{thm:saddle_dim_sharp}
at a fixed scale $R$
when $n = 4$ and $P(\xi) = \xi_1\xi_3 + \xi_2\xi_4$, \textit{i.e.} $m = 2$.
We can only represent a general slice $x_1 =$ constant. 
The initial datum $f$ is concentrated around the dark rods in $x_3 = 0$ and,
as time progresses, the rods are translated to the left until $x_3 = -1$.
Thus,
$\sup_{0\le t\le 1/R}\abs{\U{t}f}$ is large along the sheets, which
cover a set of dimension $\ge$ 2 for each $\abs{x_1}\le 1$.} \label{fig:saddle_ex_trans}
\end{figure}

It only remains to compute the Hausdorff dimension of $F$. 
Let us define the set $ H = \limsup_{k\to\infty}H_k$, where
\begin{equation}
H_k = \left\{x''\in [-1,1]^{m-1} \, : \,  \abs{x''-D_k^{-1} \, p''}\le cR_k^{-1} \textrm{ for some } p'' \in\Z^{m-1}\right\}, \qquad k \in \mathbb N.
\end{equation}
For every $k \in \mathbb N$, the set $F_k$ is essentially equal
to $H_k\times [-1,1]^{n-m+1}$, where 
$[-1,1]^{n-m+1}$ is a cube in the variables $(x_1,x',x_{m+1}, x''')$. 
Then, $\dim_{\mathcal H} F = \dim_{\mathcal H} H + n - m + 1$;
see \cite[Corollary~7.4]{MR2118797}.

Since $H$ is covered by $\bigcup_{k \geq K }H_k$ for all $K \in \mathbb N$, 
and $H_k$ is a union of $D_k^{m-1}$ balls of radius $R_k^{-1}$, we get
\begin{equation}
\mathcal H^s_{R_K^{-1}}( H ) \leq \sum_{k \geq K} D_k^{m-1} R_k^{-s} 
= \sum_{k \geq K} R_k^{a(m-1)-s}, \qquad \forall K \in \mathbb N, 
\end{equation}
so letting $K \to \infty$ we get $\mathcal H^s\left(H \right) =0$ as long as $s > a(m-1)$,
which in turn implies that $\dim_{\mathcal H} H \leq a(m-1)$.

For the lower bound, we use the Mass Transference Principle
in its original form by Beresnevich and Velani in Theorem~\ref{thm:MTP}. 
For that, we need to find $s \geq 0$ such that $\limsup_{k \to \infty} H_k^s$ has full Lebesgue measure,
where 
\begin{equation}
H_k^s = \bigcup_{p'' \in \mathbb Z^{m-1}} B\left(\frac{p''}{D_k},\frac{1}{R_k^{s/(m-1)}} \right), \qquad k \in \mathbb N.
\end{equation}
In particular, it suffices that each $H_k^s$ fills the space, 
which happens when $R_k^{s/(m-1)} \simeq D_k$, that is, when $s = a(m-1)$.
Thus, the Mass Transference Principle implies $\dim_{\mathcal H} H \geq a(m-1)$. 
Consequently, $\dim_{\mathcal H} H = a(m-1) $ and 
\begin{equation}\label{eq:Saddle_Dimension}
\dim_{\mathcal H} F = a(m-1) + n - (m-1) = n - (1-a)(m-1), \qquad \forall a \in [0,1]. 
\end{equation}
In particular, the range for the dimension is $[n-m+1,n]$.

To conclude the proof, let $\alpha \in [n-m+1,n]$ and fix  $\dim_{\mathcal H} F = \alpha$. 
Then, according to \eqref{eq:Saddle_Dimension}, we need to choose 
$a$ such that $n-\alpha = (1-a)(m-1)$, 
so the critical regularity in \eqref{eq:Critical_Regularity_For_Saddle_2}
is $s(\alpha) = (n-\alpha + 1)/2$.
Thus, by \eqref{eq:Divergent_Solution_Saddle}, we found a datum
\eqref{eq:Datum_For_Saddle} that belongs to $H^s(\R^n)$ for every $s < (n-\alpha + 1)/2$
and that diverges in a set $F$ of Hausdorff dimension $\alpha$. 
\end{proof}

\subsection{When $\boldsymbol{\alpha < n - m + 1}$}

In this case, 
the smoothing effect of $\U{t}$ seems unavoidable.
Combining the counterexample of Rogers, Vargas and Vega \cite{MR2284549}
with the Talbot effect in the counterexample of the proof of Theorem~\ref{thm:divergence_k}, 
we give a lower bound for $s_c(\alpha)$ that is strictly smaller than the 
non-dispersive threshold. 
The proof is similar to that of Theorem~\ref{thm:divergence_k};
we sketch it here and skip many details 
that can be found in Sections~\ref{sec:Building_counterexample} and 
 \ref{sec:FractalSet}.

\begin{thm}\label{thm:Saddle_Lower_Alpha}
Let $P$ be a quadratic form with index $1\le m \leq n/2$. 
Assume that $n/2 \leq \alpha < n - m + 1$ and that
\begin{equation} \label{eq:thm:reg_Saddle_Talbot}
s < \frac{n + (n-2m)(n-\alpha)}{2 ( n - 2m + 2 )}.
\end{equation} 
Then, there exists  $f \in H^s(\mathbb R^n)$ such that $T_tf$ diverges
in a set of Hausdorff dimension $\alpha$.
\end{thm}
\begin{rmk}
Let $s(\alpha)$ be right hand side of \eqref{eq:thm:reg_Saddle_Talbot}.
Then $s(n-m+1) = m/2$ matches the exponent of Theorem~\ref{thm:saddle_dim_sharp}. 
On the other hand, $s(n/2) = n/4$ coincides with the exponent in Corollary~\ref{thm:dispersive_non-Singular}.
We show this graphically in Figure~\ref{fig:saddle}.
\end{rmk}

\begin{proof}
As in Theorem~\ref{thm:saddle_dim_sharp}, it suffices to work with the symbol \eqref{eq:Transformed_Simbol_Saddle}.
Given $\varphi \in \mathcal S (\mathbb R^n)$ with $\widehat{\varphi}$ supported in $B(0,c)$
for some small $c > 0$ 
and the parameter $R \gg 1$, 
we propose the datum
\begin{equation}
\widehat{f}_R(\xi) = \sum_{\substack{|l'''| \leq R/D \\ l''' \in \mathbb Z^{n-2m}  }} \widehat{\varphi} ( \xi_1 - R, \xi', \xi_{m+1}/R, \xi''/R, \xi''' - Dl'''  ). 
\end{equation}
Here, $D = R^\gamma$ for some $0 < \gamma < 1$, and
we split the variable $\xi = (\xi_1,\xi',\xi_{m+1}, \xi'',\xi''')$. 
Moreover, $\lVert f_R \rVert_2 \simeq R^{m/2} (R/D)^{(n-2m)/2}$.
The solution takes the form 
\begin{equation}
\begin{split}
T_t f_R(x) & = \sum_{|l'''| < R/D}  \int  \widehat \varphi  ( \xi_1 - R, \xi', \xi_{m+1}/R, \xi''/R, \xi''' - Dl'''  ) \, e( x \cdot \xi + t \, P(\xi) )\, d\xi  \\
& = R^m \, e(Rx_1) \sum_{|l'''| < R/D} \, e( x''' \cdot Dl''' - t D^2 |l'''|^2 ) \\
& \qquad \qquad \int  \widehat \varphi (\xi) \, e\Big(  x_1\xi_1 + x' \cdot \xi' + R(x_{m+1} + Rt)\xi_{m+1}  + Rx'' \cdot \xi'' + \xi''' \cdot (x''' - 2tDl''')  \\
& \qquad \qquad \qquad  \qquad \qquad + t \left( R\xi_1\xi_{m+1} + R\xi' \cdot \xi'' - |\xi'''|^2 \right)  \Big)\, d\xi.
\end{split}
\end{equation}
We impose $t < 1/R$, as well as 
\begin{equation}\label{eq:Saddle_Conditions}
|x_1|, |x'|, |x'''| < 1,\quad x'' = 0\quad \textrm{and} \qquad R\, |x_{m+1} + Rt| < 1,
\end{equation}
so that the phase of the integral is small, and thus
\begin{equation}
\left|  T_t f_R(x) \right| \simeq R^m \left|  \sum_{|l'''| < R/D} \, e \left( x''' \cdot Dl''' - t D^2 |l'''|^2 \right)  \right|.
\end{equation}
Let $p''' \in \mathbb Z^{n-2m}$, $p_{m+1} \in \mathbb Z$, $q$ a prime, and 
\begin{equation}\label{eq:Saddle_Rationals}
x''' = \frac{1}{D} \, \frac{p'''}{q} \qquad \textrm{and} \qquad t = \frac{1}{D^2}\, \frac{p_{m+1}}{q}.
\end{equation}
Assuming that $q \simeq Q$ and by periodicity,
we transform the expression above into a Gauss sum,
\begin{equation}\label{eq:Saddle_Gauss_Sum}
\left|  T_t f_R(x) \right| \simeq R^m \, \left( \frac{R}{DQ} \right)^{n-2m} \left|  \sum_{l''' \in \mathbb F^{n-2m}_q} \, e \left( \frac{p''' \cdot l''' - p_{m+1}|l'''|^2}{q} \right)  \right| 
\end{equation}
By Deligne's Theorem~\ref{thm:Weil_Bound} and Lemma~\ref{thm:G_q}, 
we know that there is a set $G(q) \subset \mathbb Z^{n-2m+1}$
such that $|G(q) \cap [0,q]^{n-2m+1}| \simeq q^{n-2m+1}$ and 
 for values $(p_{m+1},p''') \in G(q)$
the Gauss sum in \eqref{eq:Saddle_Gauss_Sum} is $\simeq q^{(n-2m)/2}$.
As a consequence, 
\begin{equation}\label{eq:Saddle_Exponent_Proto}
\left|  T_t f_R(x) \right| \simeq R^m \, \left( \frac{R}{DQ^{1/2}} \right)^{n-2m} \qquad \Longrightarrow \qquad \frac{\left|  T_t f_R(x) \right| }{\lVert f_R \rVert_2} \simeq R^{m/2} \left( \frac{R}{DQ} \right)^{(n-2m)/2}.
\end{equation}
This result is analogue to Proposition~\ref{thm:Proposition_Pierce}
and will eventually give the desired Sobolev exponent. 
As in \eqref{eq:Datum_For_Saddle}, the counterexample is built summing $f_{R_k}$
for different scales $R_k = 2^k$,
which diverges in the set $F = \limsup_{k \to \infty}F_k$, 
where $F_k = F_{R_k}$ are defined by the conditions
\eqref{eq:Saddle_Conditions} and \eqref{eq:Saddle_Rationals}.
The procedure is the same as in Subsection~\ref{sec:Counterexample},
so we do not repeat it here.

Instead, let us focus on the Hausdorff dimension of $F$. 
Conditions \eqref{eq:Saddle_Conditions} and \eqref{eq:Saddle_Rationals}
suggest defining
\begin{equation}
H_R = ([-1,0]\times[-1,1]^{n - 2m}) \cap  \bigcup_{q \simeq Q} \, \bigcup_{p \in G(q)} B \left( \frac{R}{D^2}\, \frac{p_{m+1}}{q} , \frac{1}{R} \right)   \times B\left( \frac{1}{D}\, \frac{p'''}{q}, \frac{1}{R} \right).
\end{equation}
Indeed, adding an error $|\epsilon| \leq 1/R$ to the choice of $x'''$ does not alter 
the result of the Gauss sum.
Calling $H_{R_k} = H_k$, let $H = \limsup_{k \to \infty} H_k$. 
Since the rest of the variables can be taken $(x_1,x') \in [-1,1]^m$ and $x''=0$,
 the final set $F \subset \mathbb R^n$ 
will have $\dim_{\mathcal H} F = \dim_{\mathcal H} H + m$. 

To compute $\dim_{\mathcal H} H$, let us define the geometric parameters $u_1,u_2$ by 
\begin{equation}
R^{u_1} = \frac{D^2 Q}{R}, \quad R^{u_2} = DQ \qquad \Longleftrightarrow \qquad Q = R^{2u_2 - u_1 - 1}, \quad D = R^{u_1 - u_2 + 1}. 
\end{equation}
With them, we easily get an upper bound for $\dim_{\mathcal H} H $.
Indeed, $H_R$ is covered by $ Q R^{u_1 + (n-2m)u_2} $ balls of radius $1/R$.
Since $H$ is covered by $\bigcup_{k \geq K}H_k$, we immediately get
\begin{equation}
\mathcal H^s_{1/R_K} (H) \leq \sum_{k \geq K} Q R^{u_1 + u_2(n-2m) - s} = \sum_{k \geq K} R^{u_2(n-2m+2) - 1 - s},
\end{equation} 
and letting $K \to \infty$ implies that 
\begin{equation}\label{eq:Saddle_Dim_Upper_Bound}
\dim_{\mathcal H} H \leq u_2(n-2m+2) - 1.
\end{equation} 
For the lower bound, we use the Mass Transference Principle 
in Theorem~\ref{thm:MTP_Rectangles}. 
The original exponents are now $\boldsymbol{b} = (1, \ldots, 1)$, 
and the restrictions we have for $u_1,u_2$ are 
\\[4mm]
\begin{minipage}{0.6\textwidth}
\begin{itemize}
	\item Basic separation restrictions: $0 \leq u_1,u_2 \leq 1$;
	\item $Q\geq 1$: $2u_2 - u_1 - 1 \geq 0 $;
	\item Shrinking unit cell: $u_2 - u_1 \leq 1/2$.
\end{itemize}
\end{minipage}
\begin{minipage}{0.4\textwidth}
\begin{center}
\includegraphics[scale=0.8]{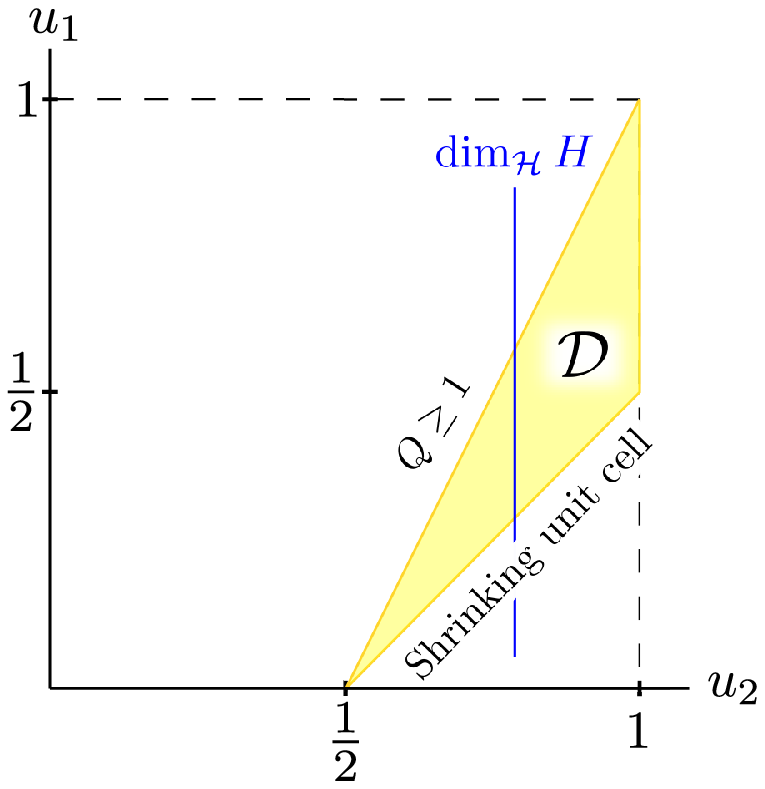}
\end{center}
\end{minipage}
\\[4mm]
\noindent Like we did in Subsection~\ref{sec:Preparing_For_MTP}, 
to check that the slabs form a uniform local ubiquity system 
it is enough to find $\boldsymbol{a} = (a_1,a_2,\ldots, a_2)$
such that $\mathcal H^n(\Omega_{R,\boldsymbol{a}}) \geq c > 0$, 
where
\begin{equation}
\Omega_{R,\boldsymbol{a}} =  \bigcup_{q \simeq Q} \, \bigcup_{p \in G(q) \cap [0,q]^{n-2m+1}} B \left( \frac{p_{m+1}}{q} , \frac{D^2}{R^{1+a_1}} \right)   \times B\left( \frac{p'''}{q}, \frac{D}{R^{a_2}} \right).
\end{equation}
According to Proposition~\ref{thm:Proposition_Pierce_Adapted}, given any $\epsilon>0$, it is enough to ask 
 \begin{equation}\label{eq:Saddle_auRelationship}
 \frac{D^2}{R^{1+a_1}} \, \left( \frac{D}{R^{a_2}}  \right)^{n-2m} \simeq \frac{1}{Q^{n-2m+2-\epsilon}} \quad \Longleftrightarrow \quad a_1 + (n-2m)a_2 = u_1 + (n-2m)u_2 + (1-\epsilon)(2u_2 - u_1 - 1),
 \end{equation}
where additionally $u_1 \leq a_1 < 1$ and $u_2 \leq a_2 < 1$. 
Given any valid $u_1,u_2$, there exist such $a_1,a_2$, 
so we apply the Mass Transference Principle in Theorem~\ref{thm:MTP_Rectangles} to get
\begin{equation}
\dim_{\mathcal H} H \geq \sum_{\ell \in K_1} 1 + \sum_{\ell \in K_2} (1 - (b_\ell - a_\ell)) + \sum_{\ell \in K_3} a_\ell,
\end{equation} 
where $K_2 = \{ \ell : b_\ell \leq 1 \}$ takes all the coordinates. 
In particular, $K_1 =  \emptyset = K_3$, so 
\begin{equation}
\dim_{\mathcal H} H \geq a_1 + a_2(n-2m).
\end{equation} 
Using \eqref{eq:Saddle_auRelationship}, we rewrite it as 
\begin{equation}
\dim_{\mathcal H} H \geq (n - 2m + 2)u_2 - 1 - \epsilon(2u_2-u_1 - 1).
\end{equation}
This is valid for every $\epsilon >0$, so given that $2u_2-u_1 - 1 \geq 0$ and together with \eqref{eq:Saddle_Dim_Upper_Bound}, we get
\begin{equation}
\dim_{\mathcal H} H = u_2(n-2m+2) - 1. 
\end{equation} 
Consequently, $\dim_{\mathcal H} F = u_2(n-2m+2) + m - 1$. 

To conclude, let us fix $\alpha = u_2(n-2m+2) + m - 1$. 
Since $1/2 \leq u_2 \leq 1$, the dimension is restricted to $n/2 \leq \alpha \leq n - m + 1$.
The Sobolev exponent, which is the exponent of $R$ in the last expression in \eqref{eq:Saddle_Exponent_Proto},
is $m/2 + (1-u_2)(n-2m)/2$, so replacing $u_2$ for $\alpha$ we get
\begin{equation}
s(\alpha) = \frac{n + (n - 2m)(n - \alpha)}{ 2(n - 2m + 2) }.
\end{equation}
\end{proof} 

We want to compare this theorem with the counterexamples constructed 
by Barron \textit{et al.} in Proposition~4.2 of \cite{MR4196386},
which morally imply that
\begin{equation}
s_c(\alpha) \ge \frac{n-m+(n-2m)(n-\alpha)}{2(n-2m+1)}, \qquad
\textrm{for } n-m\le \alpha\le n-m+1.
\end{equation}
This is better than the regularity in \eqref{eq:thm:reg_Saddle_Talbot}.
We say ``morally implies'' because
they were interested in \eqref{eq:thm:discrete},
not divergence, and
it is not always clear how to translate counterexamples 
for \eqref{eq:thm:discrete} to
counterexamples for divergence.

Their counterexamples are similar to those in Theorem~\ref{thm:Saddle_Lower_Alpha}.
In fact,
they are equal in the variables
$(\xi_1,\cdots,\xi_{2m})$, except for a translation.
In the remaining \textit{elliptic} variables $(\xi_{2m+1},\ldots,\xi_n)$
we use the Talbot effect, while
they used other counterexamples due to Du \cite[Theorem~1.2]{MR4186129}. 
They only constructed examples in the range $n - m \le \alpha \le n - m + 1$; however,
using modifications of Du's examples,
it might be possible to improve Theorem~\ref{thm:Saddle_Lower_Alpha} 
in the whole range $n/2<\alpha<n -m + 1$.

\subsection{When $\boldsymbol{n}$ is odd, $\boldsymbol{m = (n-1)/2}$ and $\boldsymbol{\alpha < (n+3)/2}$}
In this case, the exponent in the preceding subsection can be improved.
\begin{thm}
Let $n$ be odd and $P$ be a quadratic form with index $m = (n-1)/2$. 
Assume that $(n+1)/2 \le \alpha \leq (n+3)/2$ and that
\begin{equation} \label{eq:thm:saddle_n_odd}
s < \frac{n - \alpha + m + 1}{4}.
\end{equation}
Then, there exists  $f \in H^s(\mathbb R^n)$ such that $T_tf$ diverges
in a set of Hausdorff dimension $\alpha$.
\end{thm}

\begin{rmk}
Let $s(\alpha)$ be right hand side of \eqref{eq:thm:saddle_n_odd} and
$s_T(\alpha)$ the regularity in \eqref{eq:thm:reg_Saddle_Talbot}.
At one end point of $\alpha$ we have $s(n - m + 1) = s_T(n - m + 1)$,
while the slopes of $s(\alpha)$ and $s_T(\alpha)$ are 
$-1/4$ and $-(n-2m)/[2(n-2m + 2)]$, respectively.
Hence, $s(\alpha) > s_T(\alpha)$ for $\alpha\le n - m + 1$ if and only if $m > n/2 - 1$.
\end{rmk}

\begin{proof}
The example is similar to the one in the proof of Theorem~\ref{thm:saddle_dim_sharp}, 
and in particular,
we assume the same presentation of $P$ as
\begin{equation}
P(\xi) = \xi_1\xi_{m+1} + \cdots + \xi_m\xi_{2m} - \xi_n^2.
\end{equation}
We consider the initial datum
\begin{equation}
\widehat{f_R}(\xi) = \sum_{\abs{l}\le R^{1/2}/D}\widehat{\varphi}(\xi_1-R,\xi',\xi_{m+1}/R,\xi''/R,\xi_n-lD),
\end{equation}
where $(\xi_1,\ldots, \xi_m, \xi_{m+1},\ldots, \xi_{2m},\xi_n) = (\xi_1,\xi',\xi_{m+1},\xi'',\xi_n)$.
The solution can be written as
\begin{align}
\U{t}f_R(x) &= \sum_{\abs{l}\le R^{1/2}/D}\int\widehat{\varphi}(\xi_1-R,\xi',\xi_{m+1}/R,\xi''/R,\xi_n-lD)
			\,e(x\cdot\xi + tP(\xi))\,d\xi \\
	&= R^m \sum_{\abs{l}\le R^{1/2}/D}e^{2\pi i(Rx_1 +x_n\cdot lD)} \\
	&\hspace*{2cm} \int \widehat{\varphi}(\xi) \, 
		e\Big(x_1\xi_1+x'\cdot\xi'+ R(x_{m+1} + Rt)\xi_{m+1} + Rx''\cdot\xi''+ \\
	&\hspace*{4cm} (x_n-2tlD)\cdot\xi_n+tR(\xi_1,\xi')\cdot(\xi_{m+1},\xi'')-t(\xi_n^2+\abs{lD}^2)\Big)\,d\xi.
\end{align}
We evaluate it at times $0\le t\le 1/R$ and in the set $F$ given by the conditions
\begin{equation} \label{eq:saddle_Set_NTalbot}
\abs{(x_1,x')}\le 1;
\quad x_{m+1} + Rt = 0; \quad 
x'' = 0; \quad
\textrm{and} \quad 
x_n = D^{-1}\Z + \BigO(R^{-1/2}).
\end{equation}
Take the absolute value so that
\begin{equation}
\abs{\U{t}f_R(x)} \simeq  R^m \, \Big|\sum_{\abs{l}\le R^{1/2}/D}e^{2\pi ix_n\cdot lD}\Big| \simeq R^m \, \frac{R^{1/2}}{D}.
\end{equation}
Since $\norm{f_R}_2\simeq R^{m/2}(R^{1/2}/D)^{1/2}$, we get
\begin{equation}
\frac{\abs{\U{t}f_R(x)}}{\norm{f_R}_2} \simeq R^{m/2} \left( \frac{R^{1/2}}{D} \right)^{1/2} = R^{(n-2a)/4},
\end{equation}
where we set $D = R^a$ with $0\le a\le 1/2$. 
Denote
\begin{equation} \label{eq:regularity_NoTalbot}
s(a) = \frac{n-2a}{4},
\end{equation}
which is the exponent that will determine the regularity of the datum.

We define the scales $R_k = 2^k$ for $k\in\N$ and 
set the initial datum $f$ as in \eqref{eq:Datum_For_Saddle},
so $f\in H^s(\R^n)$, for $s<s(a)$, and 
$\U{t} f$ diverges in the set $F = \limsup_{k\to\infty} F_k$,
where $F_k$ are the sets given by \eqref{eq:saddle_Set_NTalbot}.
To compute the dimension of $F$, notice that it is a product
$F_k = [-1,1]^{m+1}\times H_k$, where
\begin{equation}
H_k = \left\{x_n\in [-1,1]\, : \,  \abs{x_n-D_k^{-1} \, p}\le cR_k^{-1/2} \textrm{ for some } p \in\Z\right\}, \qquad k \in \mathbb N.
\end{equation}
By the same methods used in the previous subsections, 
the dimension of this set is $2a$, so
$\alpha = \dim_{\mathcal H} F = m + 1 + 2a$ and 
the range of the dimension is $(n+1)/2\le \alpha \le (n+3)/2$.

To conclude the proof, replace $2a = \alpha - m - 1$ in \eqref{eq:regularity_NoTalbot}
to see that $s(\alpha) = (n - \alpha + m + 1)/4$.
\end{proof}

\bibliographystyle{acm}
\bibliography{FractalIDBib, fractal}

\end{document}